\documentclass[a4paper,english,8 pt]{article}
\usepackage[english]{babel}
\usepackage[T1]{fontenc}
\usepackage[latin1]{inputenc}
 \usepackage{amsmath,amssymb}
\usepackage[toc,page]{appendix}
\usepackage{graphicx}
\usepackage{amsthm}
\usepackage{dirtytalk}
\usepackage{float}
\usepackage{enumerate}
\usepackage{caption}
\usepackage{framed}
\usepackage{multirow}
\usepackage{mathtools}
\usepackage{color}
\usepackage{hyperref}
\usepackage{amsfonts}
\usepackage{etex}
\usepackage{mathrsfs}
\usepackage{epigraph}
\usepackage{latexsym}
\usepackage{geometry}
\usepackage{makeidx}
\usepackage{shuffle}
\usepackage{perpage} 
\MakePerPage{footnote} 
\DeclareMathSymbol{\mlq}{\mathord}{operators}{``}
\DeclareMathSymbol{\mrq}{\mathord}{operators}{`'}

\usepackage{fix-cm}     
\makeatletter
\newcommand\HUGE{\@setfontsize\Huge{30}{40}}
\makeatother
\usepackage{titlesec}
\newcommand*{\justifyheading}{\raggedleft}
\titleformat{\chapter}[display]
  {\normalfont\huge\bfseries\justifyheading}{\chaptertitlename\ \thechapter}
  {20pt}{\HUGE}
\usepackage{blindtext}
\usepackage[refpage]{nomencl}
\usepackage{xcolor}
\providecommand{\boldsymbol}[1]{\mbox{\boldmath $#1$}}

\newtheorem{theo}{Theorem}[section]

\newtheorem{defi}[theo]{Definition}
\newtheorem{coro}[theo]{Corollary}
\newtheorem{lemm}[theo]{Lemma}
\newtheorem{conj}[theo]{Conjecture}
\newtheorem*{theom}{Theorem}
\newtheorem*{conje}{Conjecture}

\newtheorem*{corol}{Corollary}

\makeindex\usepackage{babel}
\usepackage[all,cmtip]{xy}
\addto\extrasfrench{\providecommand{\fg}{\ifdim\lastskip>\z@\unskip\fi~\frqq}}

\newcommand\underrel[2]{\mathrel{\mathop{#2}\limits_{#1}}}
\makeatother
\makeatletter
\makeindex
\makenomenclature
 \def\makenomenclature{%
\newwrite\@nomenclaturefile
\immediate\openout\@nomenclaturefile=\jobname\@outputfileextension
\def\@nomenclature{%
 \@bsphack
\begingroup
\@sanitize
 \@ifnextchar[%
{\@@@nomenclature}{\@@@nomenclature[\nomprefix]}}%
\typeout{Writing nomenclature file \jobname\@outputfileextension}%
 \let\makenomenclature\@empty}

\begin{document}

\title{Unramified Euler sums and Hoffman $\star$ basis}
\author{Claire \textsc{Glanois}}

\maketitle

\abstract{When looking at how periods of  $\pi_{1}^{\mathfrak{m}}(\mathbb{P}^{1}\diagdown \lbrace 0, 1, \infty \rbrace )$, i.e. multiple zeta values, embeds into periods of $\pi_{1}^{\mathfrak{m}}(\mathbb{P}^{1}\diagdown \lbrace 0, \pm 1, \infty \rbrace )$, i.e. Euler sums, an explicit criteria via the coaction $\Delta$ acting on their motivic versions\footnote{Following Brown's point of view.} comes out. In this paper, adopting this Galois descent approach, we present a new basis for the space $\mathcal{H}^{1}$ of motivic multiple zeta values via motivic Euler sums. Up to an analytic conjecture\footnote{Similar to Zagier's one for Hoffman basis case done by F. Brown.},  we also prove that the motivic Hoffman star basis $\zeta^{\star, \mathfrak{m}} (2^{a_{1}},3,\cdots,3, 2^{a_{p}}, 3, 2^{b})$ is a basis of $\mathcal{H}^{1}$. Under a general motivic identity that we conjecture\footnote{The motivic version of a Linebarger Zhao's identity, $\cite{LZ}$.}, these bases are identical. Other examples of unramified ES with alternating patterns of even and odds are also provided.}

\tableofcontents

\section{Introduction}

\textbf{Euler sums} which we shall denote by \textbf{ES} are defined by:
\begin{equation}\label{eq:defes} \text{     }  \zeta\left(n_{1}, \ldots , n_{p}\right)\mathrel{\mathop:}= \sum_{0<k_{1}<k_{2} \cdots <k_{p}} \frac{\epsilon_{1}^{k_{1}} \cdots \epsilon_{p}^{k_{p}}}{k_{1}^{n_{1}} \cdots k_{p}^{n_{p}}} \text{, } \quad n_{i}\in\mathbb{Z}^{\ast}, \text{ } \epsilon_{i}\mathrel{\mathop:}= \textrm{sign}(n_{i}) \in \lbrace\pm 1\rbrace \text{, } n_{p}\neq 1.
\end{equation}

The \textit{weight}, often denoted $w$ below, is defined as $w\mathrel{\mathop:}= \sum \mid n_{i}\mid $, the \textit{depth} is the length $p$, whereas the \textit{height}, denoted $h$, is the number of $\mid n_{i}\mid$ greater than $1$. The weight is conjecturally a grading, whereas the depth is only a filtration. The special case for which all $n_{i}$ are positive corresponds to the well-known \textit{\textbf{multiple zeta values}}, denoted MZV. Let's introduce also $\boldsymbol{\mathcal{Z}^{2}}$, resp. $\boldsymbol{\mathcal{Z}^{1}}$ the $\mathbb{Q}$-vector space spanned by these Euler sums, resp. multiple zeta values.\\
These Euler sums are particularly interesting examples of \textit{periods} in the sense of Kontsevich-Zagier \footnote{Via their integral representation $\ref{eq:reprinteg}$, cf. $\cite{KZ}$}, and are flourishing in the literature, for their savory multiple connections with mixed Tate motives, quantum field theory (with Feynman amplitudes), modular forms, etc. From the zoo of relations satisfied by these Euler sums, let remember the regularized double shuffle \footnote{For a good introduction to MZV, cf. \cite{Ca}, \cite{Wa}. Conjecturally, it generates all relations between MZV.}, which turns $\mathcal{Z}^{1},\mathcal{Z}^{2}$ into algebras, and the so-called octagon relation\footnote{The hexagon relation for MZV (cf. $\cite{Fu}$) is turned into an \textit{octagon} relation (cf. $\ref{fig:octagon}$) for ES.}, which is used below.\\

This article mainly focuses on the \textit{\textbf{motivic}} versions of those numbers: i.e. motivic Euler sums, resp. motivic multiple zeta values, denoted $\boldsymbol{\zeta^{\mathfrak{m}}}(\cdot)$ and shortened \textbf{MES}, resp. \textbf{MMZV}. They are \textit{motivic periods} of the category of Mixed Tate Motives $\mathcal{MT}\left( \mathbb{Z}\left[ \frac{1}{2}\right] \right) $, resp. $\mathcal{MT}\left( \mathbb{Z} \right)$, generated by the motivic fundamental group of $\mathbb{P}^{1}\diagdown\lbrace 0, \pm 1, \infty\rbrace$, resp. $\mathbb{P}^{1}\diagdown\lbrace 0,  1, \infty\rbrace$. They span the $\mathbb{Q}$-vector spaces $\mathcal{H}^{2}$, resp. $\mathcal{H}^{1}$. Moreover,  $\mathcal{H}^{2}$ is an Hopf algebra, whose coaction is explicitly given by a combinatorial formula ($\ref{eq:coaction}$). This coaction is the dual of an action of a so-called \textit{motivic} Galois group $\boldsymbol{\mathcal{G}}$ on these specific motivic periods. \\
Furthermore, there is a surjective homomorphism, called the \textit{period map}, which is conjectured to be an isomorphism (special case of Grothendieck's period conjecture):
\begin{equation}\label{eq:per} \textbf{per} : w:\quad \mathcal{H}^{N} \rightarrow \mathcal{Z}^{N} \text{ ,  } \quad \zeta^{\mathfrak{m}} (\cdot) \mapsto \zeta ( \cdot ).
\end{equation}
Working on the motivic side, besides being conjecturally identical to the complex numbers side, turns out to be somehow simpler, since the motivic theory provides this Hopf Algebra structure. Notably, each identity between MES implies an identity for ES, by application of this period map; for instance, a motivic basis for MMZV, as the one obtained in section $4$ is hence a generating family (conjecturally basis) for MZV.\\
\\
Before delving into the main results of this paper, let introduce the following variants of ES (Definition $3.1$):
\begin{description}
\item[Euler $\star$ sums] are the analogue multiple sums than ES $(\ref{eq:defes})$ with large inequalities and verifies:\footnote{These have already been studied in many papers: $\cite{BBB}, \cite{IKOO}, \cite{KST}, \cite{LZ}, \cite{OZ}, \cite{Zh3}$.}
\begin{equation} \label{eq:esstar}\zeta ^{\star}(n_{1}, \ldots, n_{p})= \sum_{\circ=\mlq + \mrq \text{ or } ,} \zeta (n_{1}\circ \cdots \circ n_{p}).
\end{equation} 
\item[Euler $\sharp$ sums] are similarly linear combinations of MZV, with $2$-power coefficients:
\begin{equation} \label{eq:essharp} \zeta^{\sharp}(n_{1}, \ldots, n_{p})= \sum_{\circ=\mlq + \mrq \text{ or } ,} 2^{p-n_{+}} \zeta(n_{1}\circ \cdots \circ n_{p}), \quad \text{   with } n_{+} \text{ the number of  } +.
\end{equation}
\end{description} 
\texttt{Notations:}
\begin{itemize}
\item[$\cdot$]  This $\mlq + \mrq$ operation in $\mathbb{Z}$, is a summation of absolute values, while signs are multiplied.
\item[$\cdot$] A negative $n$ in a (motivic) Euler sum is also denoted below by an \textit{overline} $\overline{m}$, where $m\mathrel{\mathop:}= -n$ positive. For instance $\zeta^{\bullet}(\overline{3}, 1, \overline{5})=\zeta^{\bullet}(-3, 1, -5)$.

\end{itemize}

\paragraph{Main Results:}
\begin{itemize}
\item[$(i)$]
\begin{theom} Motivic Euler $\sharp$ sums with only positive $odd$ and negative $even$ integers as arguments are unramified: i.e. $\mathbb{Q}$ linear combinations of motivic multiple zeta values. 
\end{theom}
By $\cite{Br2}$, it means that they are Frobenius invariant geometric motivic periods of $\mathcal{MT}(\mathbb{Z})$. This family is even a generating family of MMZV from which we extract a basis:
\begin{theom}
A graded basis of $\mathcal{H}^{1}$, the space of motivic multiple zeta values is:
$$\lbrace\zeta^{\sharp,\mathfrak{m}} \left( 2a_{0}+1,2a_{1}+3,\cdots, 2 a_{p-1}+3, \overline{2a_{p}+2}\right) \text{ , } a_{i}\geq 0 \rbrace .$$
\end{theom}
The proof is based on the good behaviour of this family with respect to the coaction and the depth filtration; the suitable filtration corresponding to the \textit{motivic depth} for this family is the usual depth minus $1$. By application of the period map, this leads to:
\begin{corol}
Each multiple zeta value of depth $<d$ is a $\mathbb{Q}$ linear combination of elements $\zeta^{\sharp} \left( 2a_{0}+1,2a_{1}+3,\cdots, 2 a_{p-1}+3, \overline{2a_{p}+2}\right) $, of the same weight with $a_{i}\geq 0$, $p\leq d$.
\end{corol} 
\item[$(ii)$]
\begin{theom}\footnote{The analogous real family Hoffman $\star$ was already conjectured (in $\cite{IKOO}$) to be a basis of the space of MZV.}
If the analytic conjecture ($\ref{conjcoeff}$) holds, then the motivic \textit{Hoffman} $\star$ family $\lbrace \zeta^{\star,\mathfrak{m}} (\lbrace 2,3 \rbrace^{\times})\rbrace$ is a basis of $\mathcal{H}^{1}$, the space of MMZV.
\end{theom}
Denote by $\mathcal{H}^{2,3}$ the $\mathbb{Q}$-vector space spanned by the motivic Hoffman $\star$ family. The idea of the proof is similar as in the non-star case done by Francis Brown. We define an increasing filtration $\mathcal{F}^{L}_{\bullet}$ on $\mathcal{H}^{2,3}$, called the \textit{level}, such that:\footnote{It corresponds to the \textit{motivic depth}, as we will see through the proof.}
\begin{center}
$\mathcal{F}^{L}_{l}\mathcal{H}^{2,3}$ is spanned by $\zeta^{\star,\mathfrak{m}} (\boldsymbol{2}^{a_{0}},3,\cdots,3, \boldsymbol{2}^{a_{l}}) $, with less than \say{l} $3$.
\end{center}
One key feature is that the vector space $\mathcal{F}^{L}_{l}\mathcal{H}^{2,3}$ is stable under the action of $\mathcal{G}$. The linear independence is then proved thanks to a recursion on the level and on the weight, using the injectivity of a map $\partial$ where $\partial$ came out of the level and weight-graded part of the coaction $\Delta$ (cf. $\S 5.1$). The injectivity is proved via $2$-adic properties of some coefficients with Conjecture $\autoref{conjcoeff}$. One noteworthy difference is that, when computing the coaction on the motivic MZV$^{\star}$, some motivic MZV$^{\star\star}$ arise, which are a non convergent analogues of MZV$^{\star}$ and have to be renormalized. Therefore, where Brown in the non-star case needed an analytic formula proven by Don Zagier ($\cite{Za}$), we need some slightly more complicated identities (in Lemma $\autoref{lemmcoeff}$) because the elements involved, such as $\zeta^{\star \star,\mathfrak{m}} (\boldsymbol{2}^{a},3, \boldsymbol{2}^{b}) $ for instance, are not of depth $1$ but are linear combinations of products of depth $1$ motivic MZV times a power of $\pi$, as proved in Lemma $\autoref{lemmcoeff}$.\\
\item[$(iii)$] We also pave the way for a motivic version of a generalization of a Linebarger and Zhao's equality which expresses each motivic MZV$^{\star}$ as a motivic Euler $\sharp$ sums:
\begin{conje} 
For $a_{i},c_{i} \in \mathbb{N}^{\ast}$, $c_{i}\neq 2$,
$$\zeta^{\star, \mathfrak{m}} \left( \boldsymbol{2}^{a_{0}},c_{1},\cdots,c_{p}, \boldsymbol{2}^{a_{p}}\right)  =(-1)^{1+\delta_{c_{1}}}\zeta^{\sharp,  \mathfrak{m}} \left(B_{0},\boldsymbol{1}^{c_{1}-3 },\cdots,\boldsymbol{1}^{ c_{i}-3 },B_{i}, \ldots, B_{p}\right), $$
\begin{flushright}
where\footnote{With the Kronecker symbol $\delta_{c}=\delta_{c=1}=\left\lbrace \begin{array}{ll}
1 & \textrm{ if } c=1\\
0 & \textrm{ else }
\end{array} \right. $ and $\boldsymbol{1}^{\gamma}$ sequence of $\gamma$ 1 if $\gamma\in\mathbb{N}^{\ast}$, empty else.} $\left\lbrace \begin{array}{l}
B_{0}\mathrel{\mathop:}=  (-1)^{2a_{0}-\delta_{c_{1}}} (2a_{0}+1-\delta_{c_{1}})\\
B_{i}\mathrel{\mathop:}= (-1)^{2a_{i}-\delta_{c_{i}}-\delta_{c_{i+1}}} (2a_{i}+3-\delta_{c_{i}}-\delta_{c_{i+1}})\\
B_{p}\mathrel{\mathop:}=(-1)^{2a_{p}+1-\delta_{c_{p}}} ( 2 a_{p}+2-\delta_{c_{p}})
\end{array}\right.$.
\end{flushright}
\end{conje}
It extends the Two One formula [Ohno-Zudilin], the Three One Formula [Zagier], and Linebarger Zhao formula, and in particular, thanks to $(i)$, it implies $(ii)$, i.e. that the Hoffman$^{\star}$ family is a basis.\\
\texttt{Nota Bene}: Such a \textit{motivic relation} between MES is stronger than its analogue between ES since it contains more information; it implies many other relations because of its Galois conjugates. This explain why its is not always simple to lift an identity from ES to MES from the Theorem $\autoref{kerdn}$. If the family concerned is not stable via the coaction, such as $(iv)$ in Lemma $\autoref{lemmcoeff}$, we may need other analytic equalities before concluding.
\end{itemize}

\textsc{Remarks}
\begin{itemize}
\item The first (very naive) idea, when looking for a basis for the space of MZV, is to choose:
$$\lbrace \zeta\left( 2n_{1}+1,2n_{2}+1, \ldots, 2n_{p}+1 \right) (2 i \pi)^{2s}, n_{i}\in\mathbb{N}^{\ast}, s\in \mathbb{N} \rbrace .$$
However, considering Broadhurst-Kreimer conjecture, the depth filtration clearly does \textit{not} behave so nicely in the case of MZV \footnote{As we can see in $\cite{De}$, for motivic Euler sums, the depth filtration is dual of the descending central series of $\mathcal{U}$: in that sense, it does \textit{behave well}. }. Consequently, in order to find a basis of motivic MZV, we have to:
\begin{itemize}
\item[\texttt{Either}:] Pass by motivic Euler sums, as the Euler $\sharp$ basis ($(i)$ above). 
\item[\texttt{Or}:] Allow \textit{higher} depths, as Hoffman basis ($\cite{Br2}$), or Hoffman $\star$ basis ($(ii)$ above). 
\end{itemize}
\item Looking at how periods of $\mathcal{MT}(\mathbb{Z})$ embed into periods of $\mathcal{MT}(\mathbb{Z}[\frac{1}{2}])$, i.e. when a motivic Euler sums is \textit{unramified} (i.e. in $\mathcal{H}^{1}$) is a fragment of the Galois descent ideas, as we will explain below. We have at our disposal a recursive criteria to determine if an ES is unramified, stated in $\S 2.3.4$ via the coaction. 
\item Such results on linear independence of a family of motivic MZV are proved recursively, once we have found the \textit{appropriate level} filtration on the elements; ideally, the family considered is stable under the derivations \footnote{If the family is not \textit{a priori} \textit{stable} under the coaction, we need to incorporate in the recursion an hypothesis on the coefficients which appear when we express the right side with the elements of the family.}; the filtration, as we will see below, should correspond to the \textit{motivic depth} defined in $\S 2.3.3$, and decrease under the derivations \footnote{In the case of Hoffman $\star$ basis ($\S 5.$) it is the number of $3$, whereas in the case of Euler $\sharp$ sums basis ($\S 4.$), it is the depth minus one. The filtration by the level has to be stable under the coaction, and more precisely, the derivations $D_{r}$ decrease the level on the elements of the conjectured basis, which allows a recursion.}; if the derivations, modulo some spaces, act as a deconcatenation on these elements, linear independence follows naturally from this recursion. Nevertheless, to start this procedure, we need an analytic identity, which is left as a conjecture here. This conjecture is of an entirely different nature from the techniques developed here. We expect that it could be proved using analytic methods along the lines of $\cite{Za}, \cite{Li}$.
\item Finding a \textit{good} basis for the space of motivic multiple zeta values is a fundamental question. Hoffman basis may be unsatisfactory for various reasons, while this basis with Euler sums (linear combinations with $2$ power coefficients) may appear slightly more natural, in particular since the motivic depth is here the depth minus $1$. However, both of these two bases are not bases of the $\mathbb{Z}$ module and the primes appearing in the determinant of the \textit{passage matrix} are growing rather fast.\footnote{Don Zagier has checked this for small weights with high precision; he suggested that the primes involved in the case of this basis could have some predictable features, such as being divisor of $2^{n}-1$. The passage matrix is the inverse of the matrix expressing the considered basis in term of a $\mathbb{Z}$ basis.}
\item For these two theorems, in order to simplify the coaction, we crucially need a motivic identity in the coalgebra $\mathcal{L}$, proved in $\S 3.3$, coming from the octagon relation pictured in Figure $\ref{fig:octagon}$. More precisely, we need to consider the linearized version of the anti-invariant part by the Frobenius at infinity of this relation, in order to prove the following relation (Theorem $\autoref{hybrid}$), for $n_{i}\in\mathbb{Z}^{\ast}$:
$$\zeta^{\mathfrak{l}}_{k}\left(n_{0},\cdots, n_{p}  \right) + \zeta^{\mathfrak{l}}_{\mid n_{0}\mid +k}\left( n_{1}, \ldots, n_{p}  \right) \equiv (-1)^{w+1}\left(  \zeta^{\mathfrak{l}}_{k}\left( n_{p}, \ldots, n_{0} \right) + \zeta^{\mathfrak{l}}_{k+\mid n_{p}\mid}\left( n_{p-1}, \ldots,n_{0}\right)  \right).$$
Thanks to this hybrid relation, and the antipodal relations presented in $\S 3.2$, the coaction expression is considerably simplified in Appendix $A$.
\end{itemize}

\paragraph{\texttt{Contents}:}
After a quick overview on the motivic background in order to introduce the Hopf algebra of motivic Euler sums, we present the $\star$ and $\sharp$ versions, with some useful motivic relations (antipodal and hybrid). The fourth section focuses on some specific Euler $\sharp$ sums, starting by a broad subfamily of \textit{unramified} elements and extracting from it a new basis for $\mathcal{H}^{1}$, whereas the fifth section deals with the Hoffman star family, proving it is a basis of $\mathcal{H}^{1}$, up to an analytic conjecture ($\autoref{conjcoeff}$). The last section presents a conjectured motivic equality ($\autoref{lzg}$) which turns each motivic MZV $\star$ into a motivic Euler $\sharp$ sums of the previous honorary family; in particular, under this conjecture, the two previous bases are identical. Appendix A. gathers the coaction calculus used both in section 4 and 5, via relations proved in section 3. Appendix B simply lists the homographies of $\boldsymbol{\mathbb{P}^{1}\diagdown \lbrace 0, \mu_{N}, \infty\rbrace}$, for $N=1,2$. Some identities coming from the linearized octagon relation, not used in this paper, are presented in appendix C. Appendix D provides some examples of unramified ES up to depth 5, with alternated patterns of even and odd, while Appendix E gives the proof of Theorem $6.2$ and Appendix $F$ highlights some missing coefficients in Lemma $\autoref{lemmcoeff}$, although not needed for the proof of the Hoffman $\star$ Theorem $\autoref{Hoffstar}$.

\paragraph{\texttt{Aknowledgements}:}
 The author deeply thanks Francis Brown for a few discussions and fundamental suggestions\footnote{First by pointing out the paper of Linebarger, Zhao \cite{LZ}, linked with the Hoffman $\star$ proof and then by suggesting the octagon relation could bring the missing relations.} on this work, and Jianqiang Zhao for his careful reading of this article proofs, as included in my PhD. The author is also thankful to Don Zagier, who implements some part of it, to look the distance of these basis to a $\mathbb{Z}$ module basis, and to Herbert Gangl for pointing out the Yamamoto interpolation$\cite{Ya}$. This work was partly supported by ERC Grant 257638, and also is very grateful to the stay in the Dirk Kreimer team (Humboldt University, Berlin), where it really grew up. 

\section{Motivic Background}
This section sketches the motivic context in which this works mostly takes place: the category of Mixed Tate Motives, the motivic iterated integrals, the motivic Galois group, the motivic fundamental groupoid. For more details and enlightened surveys on some of these motivic aspects, we refer to $\cite{An}$ and $\cite{Ka}$ for motives\footnote{Motives are supposed to play the role of a universal (and algebraic) cohomology theory.}; $\cite{DM}$ for tannakian categories; $\cite{An2}$ for motivic Galois theory; $\cite{An3}$, $\cite{D1}$, $\cite{Br4}$, $\cite{Br6}$, $\cite{DG}$ for motivic periods and motivic iterated integrals; $\cite{Go2}$, $\cite{DG}$ for motivic fundamental groupoid; $\cite{Go1}$, $\cite{Br2}$ for the motivic Hopf algebra; and the recent report \cite{BJ} which gives a luxuriant overview of the motivic scene around the MZV. Notice that this paper is essentially an extract of the author PhD $\cite{Gl}$, whose Chapter $2$ provides more details on this (motivic) scenery too.\\
Let point out a few highlights for this section:
\begin{itemize}
\item[$\cdot$] The combinatorial expression of the coaction ($\ref{eq:coaction}$), dual of the motivic Galois action is the cornerstone of this work.
\item[$\cdot$] Theorem $\ref{kerdn}$ which states which elements are in the kernel of these derivations, sometimes allows to lift identities from MZV to motivic MZV, up to rational coefficients.\\
\texttt{Nota Bene}: A \textit{motivic relation} is stronger and it may require several relations between MZV in order to lift an identity to motivic MZV. An example of such a behaviour occurs with some Hoffman $\star$ elements, in Lemma $\autoref{lemmcoeff}$.
\item[$\cdot$] The Galois descent between MES and MMZV alluded in $\S 2.3.4$
\end{itemize}

\subsection{Motivic scenery}

\paragraph{Mixed Tate Motives}
For $k$ a number field (soon $\mathbb{Q}$), we have at our disposal:
\begin{center}
$\boldsymbol{\mathcal{MT}(k)}$, the tannakian \textit{category of Mixed Tate motives} over k with rational coefficients equipped with a weight filtration $W_{r}$ indexed by even integers such that:
\end{center}
\begin{itemize}
\item[$\cdot$]  Every object $M\in \mathcal{MT}(k)_{\mathbb{Q}}$ is an iterated extension of Tate motives\footnote{The \textit{Lefschetz motive} $\mathbb{L}\mathrel{\mathop:}=\mathbb{Q}(-1)= H^{1}(\mathbb{G}_{\mathfrak{m}})=H^{1}(\mathbb{P}^{1} \diagdown \lbrace 0, \infty\rbrace)$ is a pure motive with period $(2i\pi)$. Its dual is the so-called \textit{Tate motive} $\mathbb{T}\mathrel{\mathop:}=\mathbb{Q}(1)=\mathbb{L}^{\vee}$. More generally, Tate motives $\mathbb{Q}(-n)\mathrel{\mathop:}= \mathbb{Q}(-1)^{\otimes n}$ resp. $\mathbb{Q}(n)\mathrel{\mathop:}= \mathbb{Q}(1)^{\otimes n}$ have periods in $(2i\pi)^{n} \mathbb{Q}$ resp. $(\frac{1}{2i\pi})^{n} \mathbb{Q}$.} $\mathbb{Q}(n), n\in\mathbb{Z}$.\\
i.e., $\quad gr_{-2r}^{W}(M)$ is a sum of copies of $\mathbb{Q}(r)$ for $M\in \mathcal{MT}(k)$.
\item[$\cdot$] $\left\lbrace \begin{array}{ll}
\text{Ext}^{1}_{\mathcal{MT}(k)}(\mathbb{Q}(0),\mathbb{Q}(n) )\cong  K_{2n-1}(k)_{\mathbb{Q}} \otimes \mathbb{Q} &  \\
\text{Ext}^{i}_{\mathcal{MT}(k)}(\mathbb{Q}(0),\mathbb{Q}(n) )\cong 0 & \quad \text{ if } i>1 \text{ or } n\leq 0.\\
\end{array} \right. $
\end{itemize}
In particular, the weight defines a canonical \textit{fiber functor} $\omega$:
$$\begin{array}{lll}
\omega: & \mathcal{MT}(k) \rightarrow \text{Vec}_{\mathbb{Q}} &    \\
 &M \mapsto \oplus \omega_{r}(M) & \quad  \quad \text{ with  } \left\lbrace  \begin{array}{l}
  \omega_{r}(M)\mathrel{\mathop:}= \text{Hom}_{\mathcal{MT}(k)}(\mathbb{Q}(r), gr_{-2r}^{W}(M))\\
  \text{ i.e. }  \quad  gr_{-2r}^{W}(M)= \mathbb{Q}(r)\otimes \omega_{r}(M).
\end{array} \right. 
\end{array} $$
Moreover, this category is equivalent to the category of representations of the so-called \textit{motivic Galois group} $\boldsymbol{\mathcal{G}^{\mathcal{M}}}$:\footnote{This equivalence is a distinctive feature of tannakian categories, and the decomposition comes from the grading: $1 \rightarrow \mathcal{U}^{\mathcal{M}} \rightarrow \mathcal{G}^{\mathcal{M}} \leftrightarrows \mathbb{G}_{m} \rightarrow 1$ is an exact sequence.}
\begin{equation}\label{eq:catrep}
\mathcal{M} \mathrel{\mathop:}=\mathcal{MT}(k)_{\mathbb{Q}}\cong \text{Rep}_{k} \mathcal{G}^{\mathcal{M}} \cong \text{Comod } (\mathcal{O}(\mathcal{G}^{\mathcal{M}})) \quad \text{ where } \mathcal{G}^{\mathcal{M}}\mathrel{\mathop:}=\text{Aut}^{\otimes } \omega = \mathbb{G}_{m} \ltimes \mathcal{U}^{\mathcal{M}},
\end{equation}
and $\mathcal{U}^{\mathcal{M}}$ a prounipotent group scheme$_{\diagup\mathbb{Q}}$, acting trivially on the graded pieces $\omega(\mathbb{Q}(n))$.\\
Let introduce the \textit{fundamental Hopf algebra} of $\mathcal{M}$:\footnote{Note that the completion, $\mathfrak{u}$ of the pro-nilpotent graded Lie algebra of $\mathcal{U}^{\mathcal{M}}$ is free and graded with negative degrees from the $\mathbb{G}_{m}$-action 	and $\mathfrak{u}^{ab} \cong \bigoplus \text{Ext}^{1}_{\mathcal{M}} (\mathbb{Q}(0), \mathbb{Q}(n))^{\vee} \text{ in degree } n$.}
\begin{equation}
\label{eq:Amt}
\mathcal{A}^{\mathcal{M}}\mathrel{\mathop:}=\mathcal{O}(\mathcal{U}^{\mathcal{M}}) \cong T(\oplus_{n\geq 1} \text{Ext}^{1}_{\mathcal{M}_{N}} (\mathbb{Q}(0), \mathbb{Q}(n))^{\vee} ), \quad \text{ and } \mathcal{M} \cong \text{Comod}^{gr} \mathcal{A}^{\mathcal{M}}. 
\end{equation}
From now, let's restrict to $k=\mathbb{Q}$. Betti, resp. de Rham cohomology lead to the \textit{realization functors}:
$$\omega_{B}: \begin{array}{lll}
 \mathcal{MT}(\mathbb{Q}) &  \rightarrow &\text{Vec}_{\mathbb{Q}}\\
  M &\mapsto& M_{B}
\end{array}, \quad \textrm{ resp. } \quad \omega_{dR}\cong \omega:\begin{array}{lll}
\mathcal{MT}(\mathbb{Q}) &\rightarrow &\text{Vec}_{\mathbb{Q}}\\
M &\mapsto &M_{dR} 
\end{array}$$
Between all these realizations, we have comparison isomorphisms, such as $\text{comp}_{dR, B}: M_{B}\otimes_{\mathbb{Q}} \mathbb{C} \xrightarrow[\sim]{}  M_{dR} \otimes_{\mathbb{Q},B} \mathbb{C}$ or its inverse  $\text{comp}_{B,dR}$. Below, we will be looking at tensor-preserving isomorphisms such as: $\mathcal{G}_{B}\mathrel{\mathop:}=\text{Aut}^{\otimes}(\omega_{B})$ and the $(\mathcal{G}_{B}, \mathcal{G}^{\mathcal{M}})$ bitorsor $\boldsymbol{P_{B,\omega}}\mathrel{\mathop:}=\text{Isom}^{\otimes}(\omega,\omega_{B}).$\\
\\
\\
\textsc{Remarks:}
\begin{itemize}
\item[$\cdot$] A Mixed Tate motive over a $\mathbb{Q}$ is hence uniquely defined by its de Rham realization, a vector space $M_{dR}$, with an action of its motivic Galois group.
\item[$\cdot$] The tannakian category of Mixed Tate Motives over the ring of S-integers $\mathcal{O}_{S}$ of any number field $k$ (defined in \cite{DG}) has the same extensions groups than $\mathcal{MT}(k)$, apart from $\text{Ext}^{1}_{\cdot}(\mathbb{Q}(0),\mathbb{Q}(1))$. More precisely, for the categories of Mixed Tate motives $\boldsymbol{\mathcal{MT}\left( \mathbb{Z}\right)}$ and $\boldsymbol{\mathcal{MT}\left( \mathbb{Z}\left[ \frac{1}{2}\right]\right)}$ involved in this paper:
\begin{description}\label{extension}
\item[$\cdot$] $\text{Ext}^{1}_{\mathcal{MT}\left( \mathbb{Z}\right)}(\mathbb{Q}(0),\mathbb{Q}(n) )\cong  K_{2n-1}\left( \mathbb{Z}\right)_{\mathbb{Q}} \otimes \mathbb{Q} \cong  \left\lbrace  \begin{array}{ll}
 \mathbb{Z}^{\ast}\otimes_{\mathbb{Z}} \mathbb{Q}  & \text{ if } n=1 .\\
 \mathbb{Q}^{\frac{\phi(n)}{2}}  & \text{ if } n>1 \text{ odd }
\end{array} \right. .$
\item[$\cdot$] $\text{Ext}^{1}_{\mathcal{MT}\left( \mathbb{Z}\left[ \frac{1}{2}\right]\right)}(\mathbb{Q}(0),\mathbb{Q}(n) )\cong  K_{2n-1}\left( \mathbb{Z}\left[ \frac{1}{2}\right]\right)_{\mathbb{Q}} \otimes \mathbb{Q} \cong  \left\lbrace  \begin{array}{ll}
\left( \mathbb{Z}\left[ \frac{1}{2}\right]\right) ^{\ast}\otimes_{\mathbb{Z}} \mathbb{Q}  & \text{ if } n=1 .\\
 \mathbb{Q}^{\frac{\phi(n)}{2}}  & \text{ if } n>1 \text{ odd }
\end{array} \right. .$
\end{description}
\end{itemize}

\paragraph{Motivic periods} The \textit{algebra of motivic periods} of a tannakian category of mixed Tate motives $\mathcal{M}$, is defined as (cf. $\cite{D1}$, $\cite{Br6}$, $\cite{Br4}$):
$$\boldsymbol{\mathcal{P}_{\mathcal{M}}^{\mathfrak{m}}}\mathrel{\mathop:}=\mathcal{O}(\text{Isom}^{\otimes}_{\mathcal{M}}(\omega, \omega_{B}))=\mathcal{O}(P_{B,\omega}).$$
A \textbf{\textit{motivic period}} denoted as a triplet  $\boldsymbol{\left[M,v,\sigma \right]^{\mathfrak{m}}}$, element of $\mathcal{P}_{\mathcal{M}}^{\mathfrak{m}}$,  is constructed from a motive $M\in \text{ Ind } (\mathcal{M})$, and classes $v\in\omega(M)$, $\sigma\in\omega_{B}(M)^{\vee}$. It is a function $P_{B,\omega} \rightarrow \mathbb{A}^{1}$, which, on its rational points, is given by:
\begin{equation}\label{eq:mper}  \quad  P_{B,\omega} (\mathbb{Q}) \rightarrow \mathbb{Q}\text{  ,  } \quad  \alpha \mapsto \langle \alpha(v), \sigma\rangle .
\end{equation} 
Its \textit{period} is obtained by the evaluation on the complex point $\text{comp}_{B, dR}$:
\begin{equation}\label{eq:perm} 
per: \quad \begin{array}{lll}
\mathcal{P}_{\mathcal{M}}^{\mathfrak{m}} & \rightarrow & \mathbb{C} \\
\left[M,v,\sigma \right]^{\mathfrak{m}} & \mapsto &  \langle \text{comp}_{B,dR} (v\otimes 1), \sigma \rangle .
\end{array}
\end{equation}
\noindent
\texttt{Example}: The first example is the \textit{Lefschetz motivic period}: $\mathbb{L}^{\mathfrak{m}}\mathrel{\mathop:}=[H^{1}(\mathbb{G}_{m}), [\frac{dx}{x}], [\gamma_{0}]]^{\mathfrak{m}}$, period  of the Lefschetz motive $\mathbb{L}$; it can be seen as the $\mlq \textit{ motivic }  (2 i\pi)^{\mathfrak{m}}\mrq$.\\
\\
\textsc{Remarks}:
\begin{itemize}
\item[$\cdot$] This construction can be generalized for any pair of fiber functors. For instance, the \textit{de Rham periods} are $\mathcal{P}_{\mathcal{M}}^{dR}\mathrel{\mathop:}= \mathcal{O} \left( \text{Aut}^{\otimes}(\omega_{dR})\right)= \mathcal{O} \left(\mathcal{G}^{\mathcal{M}}\right)$. \textit{Unipotent variants} of these periods form the fundamental Hopf algebra defined above:
$$\mathcal{P}_{\mathcal{M}}^{\mathfrak{a}}\mathrel{\mathop:}=\mathcal{O} \left( \mathcal{U}^{\mathcal{M}}\right)= \mathcal{A}^{\mathcal{M}}, \quad \text{ with at our disposal a restriction map } \quad \mathcal{P}_{\mathcal{M}}^{\omega} \rightarrow  \mathcal{P}_{\mathcal{M}}^{\mathfrak{a}}.$$
\item[$\cdot$] Any structure carried by these fiber functors (weight grading on $\omega_{dR}$, complex conjugation on $\omega_{B}$, etc.) is transmitted to the corresponding ring of periods. In particular, $\mathcal{P}_{\mathcal{M}}^{\mathfrak{m}}$ inherits a weight grading and we can define:
\begin{center}
$\boldsymbol{\mathcal{P}_{\mathcal{M}}^{\mathfrak{m},+}} \subset \mathcal{P}_{\mathcal{M}}^{\mathfrak{m}}$, the ring of \textit{geometric periods}, is generated by periods of motives with non-negative weights:  $\left\lbrace  \left[M,v,\sigma \right]^{\mathfrak{m}}\in  \mathcal{P}_{\mathcal{M}}^{\mathfrak{m}} \mid W_{-1} M=0 \right\rbrace $.
\end{center} 
\begin{equation}\label{eq:projpiam}
\textrm{There is a morphism (details in $\cite{Br4}$, $\S 2.6$):}  \quad \quad \boldsymbol{\pi_{\mathfrak{a},\mathfrak{m}}}: \quad \mathcal{P}_{\mathcal{M}}^{\mathfrak{m},+}  \rightarrow \mathcal{P}_{\mathcal{M}}^{\mathfrak{a}}. 
\end{equation}
\item[$\cdot$] The groupoid structure (composition) on the isomorphisms of fiber functors on $\mathcal{M}$, by dualizing, leads to a coalgebroid structure on the spaces of motivic periods:\\
 $$\boldsymbol{\Delta^{\mathfrak{m}, \omega}}:\mathcal{P}_{\mathcal{M}}^{\mathfrak{m}}  \rightarrow \mathcal{P}_{\mathcal{M}}^{dR} \otimes \mathcal{P}_{\mathcal{M}}^{\mathfrak{m}}.$$
\item[$\cdot$] Bear in mind also the non-canonical isomorphisms, compatible with weight and coaction between these $\mathbb{Q}$ algebras:
\begin{equation} \label{eq:periodgeom}
 \mathcal{P}_{\mathcal{M}}^{\mathfrak{m}} \cong \mathcal{P}_{\mathcal{M}}^{\mathfrak{a}} \otimes_{\mathbb{Q}} \mathbb{Q} \left[ (\mathbb{L}^{\mathfrak{m}} )^{-1} ,\mathbb{L}^{\mathfrak{m}} \right], \quad \text{and} \quad \mathcal{P}_{\mathcal{M}}^{\mathfrak{m},+} \cong \mathcal{P}_{\mathcal{M}}^{\mathfrak{a}} \otimes_{\mathbb{Q}} \mathbb{Q} \left[ \mathbb{L}^{\mathfrak{m}} \right]. 
\end{equation}
\item[$\cdot$] The complex conjugation defines the \textit{real Frobenius} $\mathcal{F}_{\infty}: M_{B} \rightarrow M_{B}$, and induces an involution on motivic periods $\mathcal{F}_{\infty}: \mathcal{P}_{\mathcal{M}}^{\mathfrak{m}} \rightarrow\mathcal{P}_{\mathcal{M}}^{\mathfrak{m}}$, such that $\mathbb{L}^{\mathfrak{m}}$ is anti invariant. From now, we will mostly restrict to the following motivic periods:
\begin{center}
$\boldsymbol{\mathcal{P}_{\mathcal{M}, \mathbb{R} }^{\mathfrak{m},+}}$ the subset of $\mathcal{P}_{\mathcal{M}}^{\mathfrak{m},+}$ invariant under $\mathcal{F}_{\infty}$, which satisfies:
\end{center}
\begin{equation}\label{eq:periodgeomr}
\mathcal{P}_{\mathcal{M}}^{\mathfrak{m},+}\cong  \mathcal{P}_{\mathcal{M}, \mathbb{R}}^{\mathfrak{m},+} \oplus  \mathcal{P}_{\mathcal{M}, \mathbb{R}}^{\mathfrak{m},+}. \mathbb{L}^{\mathfrak{m}} \quad \text{and} \quad  \mathcal{P}_{\mathcal{M}, \mathbb{R}}^{\mathfrak{m},+}\cong  \mathcal{P}_{\mathcal{M}}^{\mathfrak{a}} \otimes_{\mathbb{Q}} \mathbb{Q}\left[ (\mathbb{L}^{\mathfrak{m}})^{2} \right] .
\end{equation}
\item[$\cdot$] The ring of motivic periods $\mathcal{P}_{\mathcal{M}}^{\mathfrak{m}} $ is a bitorsor under Tannaka groups $(\mathcal{G}^{\mathcal{M}}, \mathcal{G}_{B})$. If Grothendieck conjecture holds, via the period isomorphism, there is therefore a (left) action of the motivic Galois group $\mathcal{G}^{\mathcal{M}}$ on periods. More precisely, for each period $p$ there would exist well defined conjugates and an algebraic group over $\mathbb{Q}$, the Galois group of $p$, which transitively permutes the conjugates. It would extend the classical Galois theory for algebraic numbers to periods; cf. $\cite{An2}$.
\end{itemize}

\paragraph{Motivic fundamental groupoid}

Let $X\mathrel{\mathop:}=\mathbb{P}^{1}\diagdown \lbrace 0, \pm 1, \infty\rbrace$.\footnote{This whole set up could be applied to $\mathbb{P}^{1}\diagdown \lbrace 0, 1 \infty\rbrace$ but also to other roots of unity, as in $\cite{DG}, \cite{De}, \cite{Gl}$.} The fundamental group $\pi_{1}(X, x)$ is freely generated by $\gamma_{0}$, $\gamma_{\pm 1}$: loops around $0$, $\pm 1$. Its \textit{completed} Hopf algebra is:
$$\boldsymbol{\widehat{\pi_{1}}}\mathrel{\mathop:}= \varprojlim \mathbb{Q}[\pi_{1}(X, x)] \diagup I^{n} , \text{ with }I\mathrel{\mathop:}= \langle \gamma-1 , \gamma\in \Pi \rangle \text{ the augmentation ideal} ,$$
equipped with the completed coproduct $\Delta$ such that the elements of $\pi_{1}(X, x)$ are primitive. It is isomorphic to the Hopf algebra of non commutative formal series:
$$\widehat{\pi_{1}} \xrightarrow[\gamma_{i}\mapsto \exp(e_{i}) ]{\sim} \mathbb{Q} \langle\langle e_{0}, e_{-1},e_{1}\rangle\rangle. $$
The \textit{prounipotent completion} of  $\pi_{1}(X, x)$ is an affine group scheme $\pi_{1}^{un}(X, x)$:
\begin{equation}\label{eq:prounipcompletion}
\boldsymbol{\pi_{1}^{un}(X, x)}(R)=\lbrace x\in \widehat{\pi_{1}} \widehat{\otimes} R \mid \Delta x=x\otimes x\rbrace \cong \lbrace S\in R\langle\langle e_{0}, e_{-1}, e_{1} \rangle\rangle^{\times}\mid \Delta S=S\otimes S, \epsilon(S)=1\rbrace ,
\end{equation}
i.e. the set of non-commutative formal series with $N+1$ generators which are group-like for the completed coproduct for which $e_{i}$ are primitive. Its affine ring of regular function is the Hopf algebra for the shuffle product, and deconcatenation coproduct:
\begin{equation}
\boldsymbol{\mathcal{O}(\pi_{1}^{un}(X, x))}= \varinjlim \left(  \mathbb{Q}[\pi_{1}(X, x)] \diagup I^{n+1} \right) ^{\vee} \cong \mathbb{Q} \left\langle e^{0},  e^{-1}, e^{1} \right\rangle .
\end{equation}
This can be applied to the bitorsor $\pi_{1}(X,x,y)$, with $x,y$, rational points of $X$ and then $\mathcal{O}\left( \pi_{1}^{un}(X,x,y)\right)$ defines an Ind object in the category of Mixed Tate Motives over $k$:\footnote{Indeed, by by Beilinson theorem ($\cite{Go2}$, Theorem $4.1$): 
\begin{equation} \label{eq:opiunvarinjlim}
 \mathcal{O}(\pi_{1}^{un}(X,x,y)) \xrightarrow{\sim} \varinjlim_{n} H^{n}(X^{n}, Y^{(n)}), \text{ where }  \begin{array}{l}
Y_{i}\mathrel{\mathop:}= X^{i-1}\times \Delta \times X^{n-i-1} \\
Y_{0}\mathrel{\mathop:}= \lbrace x\rbrace \times X^{n-1},  Y_{n}\mathrel{\mathop:}=   X^{n-1} \times \lbrace y\rbrace\\
Y^{(n)}\mathrel{\mathop:}=\cup_{i} Y_{i}, \Delta \text{ the diagonal } \subset X \times X  
 \end{array}.
\end{equation}
And $Y_{I}^{(n)}\mathrel{\mathop:}=\cap Y_{i}^{(n)}$ is the complement of hyperplanes, hence of type Tate.}
\begin{equation}\label{eq:pi1unTate0}
\mathcal{O}\left( \pi_{1}^{un}(\mathbb{P}^{1}\diagdown \lbrace 0, \pm 1 ,\infty \rbrace,x,y)\right)  \in  \text{Ind } \mathcal{MT}(\mathbb{Q}).
\end{equation}
We denote it $\boldsymbol{\mathcal{O}\left( \pi_{1}^{\mathfrak{m}}(X,x,y)\right) }$, and $\mathcal{O}\left( \pi_{1}^{dR}(X,x,y)\right)$, $\mathcal{O}\left( \pi_{1}^{B}(X,x,y)\right)$ its realizations, 
resp. $\boldsymbol{\pi_{1}^{\mathfrak{m}}(X)}$ for the corresponding $\mathcal{MT}(\mathbb{Q})$-groupoid scheme, called the \textit{motivic fundamental groupoid}, with the composition of path.\\

This construction can be applied for \textit{tangential base points }\footnote{I.e. here non-zero tangent vectors in a point of $\lbrace 0, \pm 1, \infty\rbrace $ are seen as \say{base points at infinite}, cf. \cite{DG}.} The motivic torsor of path associated to tangential basepoints corresponding to the straight path between $x,y$ in $\lbrace 0, \pm 1\rbrace$ depends only on $x,y$, and we denote it $\boldsymbol{_{x}\Pi^{\mathfrak{m}}_{y}}=\pi_{1}^{\mathfrak{m}} (X, \overrightarrow{xy})$, and $_{x}\Pi_{y}=_{x}\Pi_{y}^{dR}$, resp. $_{x}\Pi_{y}^{B}$ its de Rham resp. Betti realizations. Furthermore, for tangential base points, the motivic torsor of path corresponding has good reduction outside N (cf. $\cite{DG}, \S 4.11$):
\begin{equation}\label{eq:pi1unTate}
\mathcal{O}\left( _{x}\Pi^{\mathfrak{m}}_{y} \right)  \in  \text{ Ind }  \mathcal{MT}\left( \mathbb{Z} \left[ \frac{1}{2} \right] \right) .
\end{equation}
\textsc{Remark}: There is an action of the dihedral group $Di_{2}= \mathbb{Z}\diagup 2 \mathbb{Z} \ltimes \mu_{2}$ on X (from $x \mapsto x^{-1}$, $x\mapsto \pm x$) and therefore on $\pi^{\mathfrak{m}}_{1}(X,x,y)$ by permuting the tangential base points. Because of the groupoid structure of $_{x}\Pi^{\mathfrak{m}}_{y}$, the $\mu_{2}$ equivariances and inertias (all respected by the Galois action, cf. $\cite{DG}$, $\S 5$), we can restrict our attention to: $_{0}\Pi^{\mathfrak{m}}_{-1}$ or equivalently to $_{0}\Pi^{\mathfrak{m}}_{1}.$\\
Furthermore, it is proven that (by Brown \cite{Br2}, resp. Deligne\cite{De}), for $\xi_{N}$ primitive:\footnote{I.e. generated by $\mathcal{O}(\pi_{1}^{\mathfrak{m}} (X,\overrightarrow{01}))$ by sub-objects, quotients, $\otimes$, $\oplus$, duals. }
\begin{framed}
For $N=1$, resp. $N=2$, the fundamental motivic groupoid  $\Pi^{\mathfrak{m}} (\mathbb{P}^{1} \diagdown \lbrace 0, 1, \infty \rbrace, \overrightarrow{0 \xi_{N}})$ generates the Tannakian subcategory of $\mathcal{MT}(\mathbb{Z}\left[ \frac{1}{N}\right] )$.
 \end{framed} 

Let us introduce $dch_{0,1}^{B}=_{0}1^{B}_{1}$, the image of the straight path (\textit{droit chemin}) in $_{0}\Pi_{1}^{B}(\mathbb{Q})$, and $\Phi_{KZ}$, called the \textit{Drinfeld associator}, the corresponding element in $_{0}\Pi^{dR}_{1}(\mathbb{C})$ via the Betti-De Rham comparison isomorphism:
\begin{equation}\label{eq:kz}
 \boldsymbol{\Phi_{KZ}}\mathrel{\mathop:}= dch_{0,1}^{dR}\mathrel{\mathop:}= \text{comp}_{dR,B}(_{0}1^{B}_{1})= \sum_{W\in \lbrace e_{0}, e_{-1}, e_{1} \rbrace^{\times}} \zeta_{\shuffle}(w) w \quad \in \mathbb{C} \langle\langle e_{0},e_{-1},e_{1} \rangle\rangle ,
\end{equation}
where the correspondence between MZV and words in $e_{0},e_{\pm 1}$ is similar to the iterated integral representation $(\ref{eq:reprinteg})$ below.

\subsection{Motivic Iterated Integrals}

\texttt{Reminder}: A fundamental feature of Euler sums, is their \textit{integral representation}:\footnote{The use of bold in the iterated integral indicates a repetition of the corresponding number, as $0$ here.}
\begin{equation}\label{eq:reprinteg}
\zeta \left(n_{1}, \ldots , n_{p} \right) = (-1)^{p} I(0; \eta_{1}, \boldsymbol{0}^{n_{1}-1}, \eta_{2},\boldsymbol{0}^{n_{2}-1}, \ldots, \eta_{p}, \boldsymbol{0}^{n_{p}-1} ;1),  
\end{equation}
\begin{equation}
\texttt{ where } \left\lbrace \begin{array}{l}
I(0; a_{1}, \ldots , a_{n} ;1)\mathrel{\mathop:}=  \int_{0< t_{1} < \cdots < t_{n} < 1} \frac{dt_{1} \cdots dt_{n}}{ (t_{1}-a_{1}) \cdots (t_{n}-a_{n}) }=\int_{0}^{1}  \omega_{a_{1}} \ldots \omega_{a_{n}}\\
\\
n_{i}\in\mathbb{N}^{\ast}, \quad \epsilon_{i}\mathrel{\mathop:}=\textrm{sign}(n_{i}), \quad \eta_{i}\mathrel{\mathop:}=  \epsilon_{i}\ldots \epsilon_{p}, \quad a_{i}\in \lbrace 0, \pm 1 \rbrace \texttt{ and } \omega_{a} \mathrel{\mathop:}=\frac{dt}{t-a}
\end{array}\right..
\end{equation}

\begin{defi}\label{peri}
Motivic Euler sums are the motivic periods associated to $M=\mathcal{O}(\pi^{\mathfrak{m}}_{1}(\mathbb{P}^{1}-\lbrace 0,\pm 1,\infty\rbrace ,\overrightarrow{xy} ))$ in the category $\mathcal{M}=\mathcal{MT}(\mathbb{Z}\left[ \frac{1}{2}\right] )$:
$$\boldsymbol{\zeta_{k}^{\mathfrak{m}} \left( n_{1}, \ldots , n_{p} \right) }\mathrel{\mathop:}= (-1)^{p} I^{\mathfrak{m}} \left(0;\boldsymbol{0}^{k}, \epsilon_{1}\cdots \epsilon_{p}, \boldsymbol{0}^{n_{1}-1} ,\cdots, \epsilon_{i}\cdots \epsilon_{p}, \boldsymbol{0}^{n_{i}-1} ,\cdots, \epsilon_{p}, \boldsymbol{0}^{n_{p}-1} ;1 \right), \textrm{ for } n_{i}\in\mathbb{Z}^{\ast} ,$$
where such a motivic iterated integral is the triplet\footnote{Referring to $\ref{eq:mper}$. Restricting to positive $n_{i}$, i.e. $a_{i}\in \lbrace0, 1 \rbrace$ leads to the motivic multiple zeta values.}:
$$\boldsymbol{I^{\mathfrak{m}}(x;w;y)}\mathrel{\mathop:}= \left[\mathcal{O} \left( \Pi^{\mathfrak{m}} \left( X_{N}, \overrightarrow{xy}\right) \right) ,w,_{x}dch_{y}^{B}\right]^{\mathfrak{m}}, \quad \text{ with } \begin{array}{l}
w\in \omega(\mathcal{O}(_{x}\Pi^{\mathfrak{m}}_{y}))\cong \mathbb{Q}  \left\langle  \omega_{0}, \omega_{-1}, \omega_{1}  \right\rangle\\
_{x}dch_{y}^{B}\in\omega_{B}(M)^{\vee}
\end{array}$$
$$\text{and whose period is:} \quad \text{per}(I^{\mathfrak{m}}(x;w;y))= I(x;w;y) =\int_{x}^{y}w= \langle \text{comp}_{B,dR}(w\otimes 1),_{x}dch_{y}^{B} \rangle \in\mathbb{C}.$$
\end{defi}
\noindent
\texttt{Notation: } We identify, for $a_{i}\in \lbrace 0, \pm 1\rbrace $: $I^{\mathfrak{m}} (a_{0}; a_{1}, \ldots, a_{n}; a_{n+1})\mathrel{\mathop:}= I^{\mathfrak{m}} (a_{0}; \omega_{a_{1}} \cdots \omega_{a_{n}} ; a_{n+1})$.\\
In a similar vein, define:
\begin{itemize}
\item[$\cdot$] De Rham motivic period in $\mathcal{O}(\mathcal{G})$, $I^{dR}(x;w;y)=\left[\mathcal{O} \left( _{x}\Pi^{\mathfrak{m}}_{y}\right)  ,w,_{x}1^{dR}_{y} \right]^{dR}$ resp. $\zeta^{dR}(\ldots)$, where $w\in\omega_{dR}(\mathcal{O} \left( _{x}\Pi^{\mathfrak{m}}_{y}\right) )$ and $_{x}1^{dR}_{y}\in \omega(M)^{\vee}=\mathcal{O}\left( _{x}\Pi_{y}\right)^{\vee}$.
\item[$\cdot$] Unipotent motivic periods, $\boldsymbol{I^{\mathfrak{a}}}$ resp. $\boldsymbol{\zeta^{\mathfrak{a}}}$, images of $I^{dR}$ resp. $\zeta^{dR}$ in $\mathcal{A}$. \footnote{By the projection $\mathcal{O}(\mathcal{G})\twoheadrightarrow \mathcal{O}(\mathcal{U})=\mathcal{A}$. These \textit{unipotent} motivic periods are the ones studied by Goncharov.}
\item[$\cdot$] $\boldsymbol{I^{\mathfrak{l}}}$ resp. $\boldsymbol{\zeta^{\mathfrak{l}}}$, the images of $I^{\mathfrak{a}}$ resp. $\zeta^{\mathfrak{a}}$ in the coalgebra $\mathcal{L}\mathrel{\mathop:}=\mathcal{A}_{>0} \diagup \mathcal{A}_{>0}. \mathcal{A}_{>0}$.\\
\end{itemize}

\paragraph{Properties.}\label{propii}
In depth $1$, by Deligne and Goncharov (\cite{DG}, Theorem $6.8$), the only relations satisfied in $\mathcal{A}^{2}$ are distribution relations\footnote{More generally, distribution relations for Euler sums are:
$$\forall n_{i}>0, \quad \zeta\left(  n_{1}, \ldots , n_{p}   \right) = 2^{\sum n_{i} - p} \sum_{\epsilon_{i}\in \lbrace \pm 1 \rbrace} \zeta \left( \epsilon_{1}\cdot n_{1}, \ldots , \epsilon_{p}\cdot n_{p} \right) .$$}
\begin{equation}\label{eq:depth1}
\forall r>0, \quad \zeta^{\mathfrak{a}} (2r+1)= \frac{2^{r-1}}{1-2^{r-1}} \zeta^{\mathfrak{a}}(-2r-1)  \quad  \textrm{ and }  \quad \zeta^{\mathfrak{a}} (2r)= \zeta^{\mathfrak{a}} (-2r)=0.
\end{equation}
Besides, they satisfy the usual properties for (motivic) iterated integrals (here $a_{i}\in\lbrace 0, \pm 1 \rbrace$):
\begin{itemize}
	\item[(i)] $I^{\mathfrak{m}}(a_{0}; a_{1})=1$.
	\item[(ii)] $I^{\mathfrak{m}}(a_{0}; a_{1}, \cdots a_{n}; a_{n+1})=0$ if $a_{0}=a_{n+1}$.
	\item[(iii)] Shuffle product:
	\begin{equation}\label{eq:shufflereg}
\hspace*{-0.5cm} \zeta_{k}^{\mathfrak{m}} \left( n_{1}, \ldots , n_{p} \right)= (-1)^{k}\sum_{i_{1}+ \cdots + i_{p}=k} \binom {\mid n_{1}\mid+i_{1}-1} {i_{1}} \cdots \binom {\mid n_{p}\mid +i_{p}-1} {i_{p}} \zeta^{\mathfrak{m}} \left( n_{1}\mlq + \mrq i_{1}, \ldots , n_{p}\mlq +\mrq i_{p} \right).
 \end{equation}
	\item[(iv)] Path composition: 
	$$ \forall x\in \mu_{N} \cup \left\{0\right\} ,  I^{\mathfrak{m}}(a_{0}; a_{1}, \ldots, a_{n}; a_{n+1})=\sum_{i=1}^{n} I^{\mathfrak{m}}(a_{0}; a_{1}, \ldots, a_{i}; x) I^{\mathfrak{m}}(x; a_{i+1}, \ldots, a_{n}; a_{n+1}) .$$
	\item[(v)] Path reversal: $I^{\mathfrak{m}}(a_{0}; a_{1}, \ldots, a_{n}; a_{n+1})= (-1)^n I^{\mathfrak{m}}(a_{n+1}; a_{n}, \ldots, a_{1}; a_{0}).$
	\item[(vi)] Homothety: $\forall \alpha \in \mu_{N}, I^{\mathfrak{m}}(0; \alpha a_{1}, \ldots, \alpha a_{n}; \alpha a_{n+1})  = I^{\mathfrak{m}}(0; a_{1}, \ldots, a_{n}; a_{n+1})$.
\end{itemize}
\textsc{Remark}: These (motivic) iterated integrals verify stuffle $\ast$ relations, but also pentagon, and hexagon or octagon (cf. $\cite{EF}, \cite{Fu}$).

\paragraph{Comodule of MES} Our main object of interest are the graded $\mathcal{A}^{N}$-comodules of the $\mathcal{F}_{\infty}$-invariant geometric periods of $\mathcal{M}=\mathcal{MT}(\mathbb{Z}\left[ \frac{1}{N}\right] )$, for $N=1,2$:\footnote{Cf. $(\ref{eq:periodgeom}), (\ref{eq:periodgeomr})$. $\mathbb{L}^{\mathfrak{m}}$, in degree 1, has a trivial coaction.}  
 \begin{equation}\label{eq:hn}
 \boldsymbol{\mathcal{H}^{N}} \mathrel{\mathop:}= \mathcal{P}_{\mathcal{M},\mathbb{R}}^{\mathfrak{m},+}= \mathcal{A}^{N} \otimes \mathbb{Q}\left[  (\mathbb{L}^{\mathfrak{m}})^{2} \right]  \quad  \subset \quad \mathcal{O}(\mathcal{G}^{N})=  \mathcal{A}^{N}\otimes  \mathbb{Q}[\mathbb{L}^{\mathfrak{m}}, (\mathbb{L}^{\mathfrak{m}})^{-1}].   
\end{equation}
\noindent
\texttt{Nota Bene:}  $\mathcal{H}^{2}$, resp. $\mathcal{H}^{1}$ is generated by the motivic Euler sums, resp. motivic MZV;  Euler sums are periods of the motivic fundamental groupoid of $\mathbb{P}^{1}\diagdown \lbrace 0, \infty, \pm 1 \rbrace$. \footnote{More precisely, ES of weight $n$ are periods of $X^{n}$ relative to $Y^{(n)}$, with the notations of $\ref{eq:opiunvarinjlim}$. The case of tangential base points requires blowing up to get rid of singularities. Most interesting periods are often those whose integration domain meets the singularities of the differential form.}\\
\\
From $(\ref{eq:projpiam})$, there is a surjective homomorphism called the \textbf{\textit{period map}}, conjectured to be isomorphism:
 \begin{equation}\label{eq:period}\text{per}:\mathcal{H} \rightarrow \mathcal{Z} \text{ ,  } \zeta^{\mathfrak{m}} \left(n_{1}, \ldots , n_{p} \right)\mapsto  \zeta\left(n_{1}, \ldots , n_{p} \right).
 \end{equation}
\texttt{Nota Bene:} Each identity between motivic Euler sums is then true for Euler sums. In particular. Conversely, we can sometimes \textit{lift} an identity between Euler sums to an identity between motivic Euler sums, via the coaction, by Theorem $\ref{kerdn}$; this is discussed below, and illustrated throughout this work in different examples or counterexamples, as in Lemma $\autoref{lemmcoeff}$.\\
\\
By results of P. Deligne ($N=2$, \cite{De}), resp. F. Brown ($N=1$, \cite{Br2}):
\begin{description}
\item[$\boldsymbol{N=2}$]: The Deligne family, $\lbrace\zeta^{\mathfrak{m}}(2a_{1}+1, \cdots, 2a_{p-1}+1, \overline{2a_{p}+1})(\mathbb{L}^{\mathfrak{m}})^{2s}\rbrace_{a_{i} \geq 0 \atop s\geq 0}$ is a basis of $\mathcal{H}^{2}$. In particular, the dimensions $d^{2}_{n}\mathrel{\mathop:}= \dim \mathcal{H}^{2}_{n}$ satisfy $d^{2}_{n}=d^{2}_{n-2}+d^{2}_{n-1}$, and we have a non-canonical isomorphism with the free Lie algebra:
\begin{equation}
\label{eq:LieAlg}
\mathfrak{u}^{2} \underrel{n.c}{\cong} L\mathrel{\mathop:}= \mathbb{L}_{\mathbb{Q}} \left\langle    \sigma^{2r+1} \mid  r\geq 0,  \sigma_{i} \text{ in degree } -i.\right\rangle 
\end{equation}
These generators $\sigma_{i}$ are non-canonical: only their classes in the abelianization are. Then, $\mathcal{A}^{2}$ is a cofree commutative graded Hopf algebra cogenerated by one element $f_{r}=(\sigma_{r})^{\vee}$ in each odd degree $r$ and:\footnote{Non-canonical isomorphisms: we can fix the image of algebraically independent elements with trivial coaction.}
\begin{equation}
\label{eq:phih}
\mathcal{A}^{2} \underrel{n.c}{\cong} A^{2}\mathrel{\mathop:}= \mathbb{Q} \left\langle  f_{1}, f_{3}, f_{5}, \cdots \right\rangle,  \quad \mathcal{H}^{2}    \xhookrightarrow[n.c. \sim]{\quad\phi^{2}\quad}  H^{2}\mathrel{\mathop:}=  \mathbb{Q} \left\langle f_{1}, f_{3}, f_{5}, \cdots \right\rangle  \otimes \mathbb{Q}\left[ \left( \mathbb{L}^{\mathfrak{m}}\right) ^{2} \right].
\end{equation}

\item[$\boldsymbol{N=1}$]: The Hoffman family, namely MMZV with only 2 and 3 as arguments, is a basis of $\mathcal{H}^{1}$, the space of MMZV. In particular, the dimensions satisfy $d^{1}_{n}=d^{1}_{n-2}+d^{1}_{n-3}$. Similarly than above,  except that there is no generator in degree $1$, we have non canonical isomorphisms with $L^{1}, A^{1}, H^{1}$, such as:
$$\mathcal{H}^{1}    \xhookrightarrow[n.c. \sim]{\quad\phi^{1}\quad}  H^{1}\mathrel{\mathop:}=  \mathbb{Q} \left\langle f_{3}, f_{5}, f_{7}, \cdots \right\rangle  \otimes \mathbb{Q}\left[  \left( \mathbb{L}^{\mathfrak{m}}\right) ^{2}\right].$$
\end{description}

\subsection{Motivic Hopf algebra}

Here, we assume (if not precised) $N=2$,  i.e. we are looking at the comodule of motivic Euler sums, omitting the exponent $2$; the case $N=1$ being here similar.

\subsubsection{Motivic Lie algebra}

Let $\mathfrak{g}$ the free graded Lie algebra generated by $e_{0}, e_{-1}, e_{1}$ in degree $-1$. Then, the completed Lie algebra $\mathfrak{g}^{\wedge}$ is the Lie algebra of $_{0}\Pi_{1}(\mathbb{Q})$ and the universal enveloping algebra $ U\mathfrak{g}$ is the cocommutative Hopf algebra which is the graded dual of $O(_{0}\Pi_{1})$:
\begin{equation} \label{eq:ug}  (U\mathfrak{g})_{n}=\left( \mathbb{Q}e_{0} \oplus \mathbb{Q}e_{-1}\oplus\mathbb{Q}e_{1} \right) ^{\otimes n}= (O(_{0}\Pi_{1})^{\vee})_{n}.
\end{equation}
The product is the concatenation, and the coproduct is such that $e_{0},e_{\pm 1}$ are primitive.\\
\\
The motivic Drinfeld associator is (cf. $\ref{eq:kz}$), with $n_{i}\mathrel{\mathop:}=\epsilon_{i}m_{i}$ and $\epsilon_{i}\mathrel{\mathop:}=\eta_{i}\eta_{i+1}$:
\begin{equation} \label{eq:associator}
\boldsymbol{\Phi^{\mathfrak{m}}}\mathrel{\mathop:}= \sum_{w} \zeta^{\mathfrak{m}} (w) w \in \mathcal{H}\left\langle \left\langle e_{0}, e_{-1}, e_{1} \right\rangle \right\rangle \text{, where }  \zeta^{\mathfrak{m}} (e_{0}^{n}e_{\eta_{1}}e_{0}^{m_{1}-1}\cdots e_{\eta_{p}}e_{0}^{m_{p}-1})  =\zeta^{\mathfrak{m}}_{n}\left( n_{1}, \ldots, n_{p} \right)
\end{equation} 
$$\textrm{ defines a map} : \quad \oplus \mathcal{H}_{n}^{\vee} \rightarrow U \mathfrak{g} \quad \textrm{ which induces: } \oplus \mathcal{L}_{n}^{\vee} \rightarrow U \mathfrak{g}.$$ 
The \textit{motivic Lie algebra}, $\boldsymbol{\mathfrak{g}^{\mathfrak{m}}}$ is the image of $\oplus \mathcal{L}_{n}^{\vee}$ in $U \mathfrak{g}$: $\oplus \mathcal{L}_{n}^{\vee} \xrightarrow{\sim} \mathfrak{g}^{\mathfrak{m}} \hookrightarrow U \mathfrak{g}$. It is equipped with the Ihara bracket and non-canonically isomorphic to the free Lie algebra $L$ defined in $(\ref{eq:LieAlg})$, generated by $(\sigma_{i})'s$. 

\subsubsection{Coaction}

The motivic Galois group $\mathcal{G}$ and hence $\mathcal{U}$ acts on the de Rham realization $_{0}\Pi_{1}$ (cf. $\cite{DG}, \S 4.12$) and (non trivially) it factorizes through the Ihara action $\circ$: $_{0}\Pi_{1}  \times _{0}\Pi_{1}  \rightarrow  _{0}\Pi_{1} $. By duality, this action gives rise to a coaction $\Delta^{\mathcal{MT}}$. The  combinatorial coaction $\Delta^{c}$ (on words on $0, \pm 1$) factorizes through $\Delta^{\mathcal{MT}}$ which factorizes through $\mathcal{A}$, since $\mathcal{U}$ is the quotient of $\mathcal{U}$ by the kernel of its action on $_{0}\Pi_{1}$. By passing to the quotient, it induces a coaction $\Delta$ on $\mathcal{H}$:
$$  \xymatrix{
\mathcal{O}(_{0}\Pi_{1}) \ar[r]^{\Delta^{c}} \ar[d]^{\sim} & \mathcal{A} \otimes_{\mathbb{Q}} \mathcal{O} (_{0}\Pi_{1})  \ar[d] \\
\mathcal{O}(_{0}\Pi_{1}) \ar[d]\ar[r]^{\Delta^{\mathcal{MT}}} & \mathcal{A}^{N} \otimes_{\mathbb{Q}} \mathcal{O} (_{0}\Pi_{1})  \ar[d]\\
\mathcal{H} \ar[r]^{\Delta}  & \mathcal{A} \otimes \mathcal{H}. 
}$$

\begin{theom}[Goncharov $\cite{Go1}$, Brown$\cite{Br2}$] \label{eq:coaction}
The coaction $\Delta: \mathcal{H} \rightarrow \mathcal{A} \otimes_{\mathbb{Q}} \mathcal{H}$ is given by:
$$\Delta^{c} I^{\mathfrak{m}}(a_{0}; a_{1}, \cdots a_{n}; a_{n+1}) =$$
$$\sum_{k ;i_{0}= 0<i_{1}< \cdots < i_{k}<i_{k+1}=n+1} \left( \prod_{p=0}^{k} I^{\mathfrak{a}}(a_{i_{p}}; a_{i_{p}+1}, \cdots a_{i_{p+1}-1}; a_{i_{p+1}}) \right) \otimes I^{\mathfrak{m}}(a_{0}; a_{i_{1}}, \cdots a_{i_{k}}; a_{n+1}) .$$
\end{theom}
\noindent
\textsc{Remark:} Considering the $a_{i}$ as vertices on a half-circle, it is: 
$$\sum_{\text{ polygons on circle  } \atop \text{ with vertices } (a_{i_{p}})}     \prod_{p}  I^{\mathfrak{a}}\left(  \text{ arc between consecutive vertices  } \atop \text{ from } a_{i_{p}} \text{ to } a_{i_{p+1}} \right)  \otimes I^{\mathfrak{m}}(\text{  vertices } ).$$
\texttt{Example}: In the reduced coaction ($\Delta':=\Delta-1\otimes id-id\otimes 1$) of $\zeta^{\mathfrak{m}}(-1,3)$, there are three non zero cuts:
$$\hspace*{-0.5cm}\begin{array}{ll}
\Delta'\left( \zeta^{\mathfrak{m}}(-1,3)\right) & =\Delta'\left( I^{\mathfrak{m}}(0; -1,1,0,0;1)\right)\\
& = I^{\mathfrak{a}}(0; -1;1) \otimes I^{\mathfrak{m}}(0; 1,0,0;1)+ I^{\mathfrak{a}}(-1; 1;0) \otimes I^{\mathfrak{m}}(0; -1,0,0;1)+ I^{\mathfrak{a}}(-1; 1,0,0;1) \otimes I^{\mathfrak{m}}(0; -1;1)\\
& =\zeta^{\mathfrak{a}}(-1)\otimes \zeta^{\mathfrak{m}}(3)-\zeta^{\mathfrak{a}}(-1)\otimes \zeta^{\mathfrak{m}}(-3)+ (\zeta^{\mathfrak{a}}(3)-\zeta^{\mathfrak{a}}(-3) )\otimes \zeta^{\mathfrak{m}}(-1)
\end{array}$$  
Define for $r\geq 1$, the \textit{derivation operators}, weight graded parts of the coaction:
\begin{equation}\label{eq:dr}
\boldsymbol{D_{r}}: \mathcal{H} \rightarrow \mathcal{L}_{r} \otimes_{\mathbb{Q}} \mathcal{H},
\end{equation}
 composite of $\Delta'= \Delta^{c}- 1\otimes id$ with $\pi_{r} \otimes id$, where $\pi_{r}$ is the projection $\mathcal{A} \rightarrow \mathcal{L} \rightarrow \mathcal{L}_{r}$.\\
\\
These maps $D_{r}$ are derivations: $D_{r} (XY)= (1\otimes X) D_{r}(Y) + (1\otimes Y) D_{r}(X)$
and:
\begin{framed}
\begin{equation}
\label{eq:Der}
D_{r}I^{\mathfrak{m}}(a_{0}; a_{1}, \cdots, a_{n}; a_{n+1})= 
\end{equation}
$$\sum_{p=0}^{n-1} I^{\mathfrak{l}}(a_{p}; a_{p+1}, \cdots, a_{p+r}; a_{p+r+1}) \otimes I^{\mathfrak{m}}(a_{0}; a_{1}, \cdots, a_{p}, a_{p+r+1}, \cdots, a_{n}; a_{n+1}) .$$
\end{framed}
\noindent
\texttt{Example}: By the previous example:
$$D_{3}(\zeta^{\mathfrak{m}}(-1,3))=(\zeta^{\mathfrak{a}}(3)-\zeta^{\mathfrak{a}}(-3) )\otimes \zeta^{\mathfrak{m}}(-1), \texttt{and} \quad  D_{1}(\zeta^{\mathfrak{m}}(-1,3))= \zeta^{\mathfrak{a}}(-1)\otimes ( \zeta^{\mathfrak{m}}(3)- \zeta^{\mathfrak{m}}(-3)) $$

\paragraph{Kernel of $\boldsymbol{D_{<n}}$. }

Let first highlight that since $\mathcal{L}_{2r}=0$, we can consider only $\left\lbrace D_{\textrm{2r+1}}\right\rbrace _{r\geq 0}$.\\
A key point for the use of these derivations is the ability to prove some relations (and possibly lift some from MZV to motivic MZV) up to rational coefficients. This comes from the following theorem, looking at primitive elements:
\begin{theo}\label{kerdn}
Let $D_{<n}\mathrel{\mathop:}= \oplus_{2r+1<n} D_{2r+1}$, $n>1$. Then:
$$\ker D_{<n}\cap \mathcal{H}^{N}_{n} =  \mathbb{Q}\zeta^{\mathfrak{m}}\left( n  \right) .$$
\end{theo}
\begin{proof} Straight-forward via the isomorphism of graded Hopf comodules $(\ref{eq:phih})$, since the analogue statement for $H^{2}$ is obviously true.
\end{proof}
\noindent
\texttt{Nota Bene} By this result proving an identity between motivic MZV (resp. MES), amounts to:
\begin{enumerate}
\item Prove that the coaction is identical on both sides, computing $D_{r}$ for $r>0$ smaller than the weight. If the families are not stable under the coaction, this step would require other identities.
\item Use the corresponding analytic result for MZV (resp. Euler sums) to deduce the remaining rational coefficient; if the analytic equivalent is unknown, we can at least evaluate numerically this rational coefficient.
\end{enumerate}
Another important use of this corollary, is the decomposition of (motivic) multiple MZV into a conjectured basis, which has been explained by F. Brown in \cite{Br1}; similarly for MES.

\subsubsection{Depth filtration}

The inclusion $\mathbb{P}^{1}\diagdown \lbrace 0, \pm 1,\infty\rbrace \subset \mathbb{P}^{1}\diagdown \lbrace 0,\infty\rbrace$ implies a surjection for the de Rham realizations of fundamental groupoid: $  _{0}\Pi_{1} \rightarrow  \pi_{1}^{dR}(\mathbb{G}_{m}, \overrightarrow{01})$. The dual leads to the inclusion:
\begin{equation} \label{eq:drsurjdual}
 \mathcal{O}  \left( \pi_{1}^{dR}(\mathbb{G}_{m}, \overrightarrow{01} ) \right)  \cong \mathbb{Q} \left\langle e^{0} \right\rangle  \xhookrightarrow[\quad \quad]{} \mathcal{O} \left(  _{0}\Pi_{1} \right) \cong \mathbb{Q} \left\langle e^{0}, e^{-1}, e^{1} \right\rangle  .
   \end{equation}
This leads to the definition of  an increasing \textit{depth filtration} $\mathcal{F}^{\mathfrak{D}}$ such that:\footnote{Dual to the filtration given by the descending central series of the kernel of $ _{0}\Pi_{1} \rightarrow  \pi_{1}^{dR}(\mathbb{G}_{m}, \overrightarrow{01})$.}
 \begin{equation}\label{eq:filtprofw} \boldsymbol{\mathcal{F}_{p}^{\mathfrak{D}}\mathcal{O}(_{0}\Pi_{1})} \mathrel{\mathop:}= \left\langle  \text{ words } w \text{ in } \left\lbrace e^{0},e^{-1},e^{1} \right\rbrace   \text{ such that }  deg _{e^{1}}(w) + deg _{e^{-1}} \leq p \right\rangle _{\mathbb{Q}}.
 \end{equation}
This filtration, preserved by $\Delta$, descends to $\mathcal{H}$: $\mathcal{F}_{p}^{\mathfrak{D}}\mathcal{H}\mathrel{\mathop:}= \left\langle  \zeta^{\mathfrak{m}}\left( n_{1}, \ldots, n_{r} \right) , r\leq p \right\rangle _{\mathbb{Q}}$.\\
Idem for $\mathcal{F}_{p}^{\mathfrak{D}}\mathcal{A}$ and  $\mathcal{F}_{p}^{\mathfrak{D}}\mathcal{L}$, and the graded spaces $gr^{\mathfrak{D}}_{p}$ are the quotient $\mathcal{F}_{p}^{\mathfrak{D}}/\mathcal{F}_{p-1}^{\mathfrak{D}}$. Beware, the grading on $\mathcal{O}(_{0}\Pi_{1})$ is not motivic and the depth is not a grading on $\mathcal{H}$.\\
Similarly, the increasing depth filtration on $U\mathfrak{g}$ (degree in $\lbrace e_{1}, e_{-1}\rbrace$) passes to the motivic Lie algebra $\mathfrak{g}^{\mathfrak{m}}$ such that the graded pieces $gr^{r}_{\mathfrak{D}} \mathfrak{g}^{\mathfrak{m}}$ are dual to $gr^{\mathfrak{D}}_{r} \mathcal{L}$.\\
\\
The \textit{motivic depth} of an element in $\mathcal{H}^{N}$ is defined, via $\ref{eq:phih}$, as the degree of the polynomial in the $(f_{i})$. It can also be defined recursively as, for $\mathfrak{Z}\in \mathcal{H}^{N}$:
$$\begin{array}{lll}
\mathfrak{Z} \text{ of motivic depth } 1 & \text{ if and only if } & \mathfrak{Z}\in  \mathcal{F}_{1}^{\mathfrak{D}}\mathcal{H}^{N}.\\
\mathfrak{Z} \text{ of motivic depth } \leq p & \text{ if and only if } &  \left( \forall r< n, D_{r}(\mathfrak{Z})  \text{ of motivic depth } \leq p-1 \right).
\end{array}.$$
Clearly: $ \text {depth } p \geq p_{c} \geq $ motivic depth $p_{\mathfrak{m}}$,  where $ p_{c}$ is the smallest $i$ such that $\mathfrak{Z}\in \mathcal{F}_{i}^{\mathfrak{D}}\mathcal{H}^{2}$. In fact, $p_{\mathfrak{m}}$ always coincides with $p_{c}$ for MES, whereas for MMZV, they may differ.

\paragraph{Depth graded derivation}

Translating ($\ref{eq:dr}$) for motivic Euler sums:\footnote{This sum $\mlq\sum\mrq$, as a $\mlq + \mrq$ indicate that the absolute values are summed whereas the signs are multiplied.}
\begin{lemm}
\label{drz}
$$\hspace*{-0.3cm}\begin{array}{ll}
 D_{2r+1} \left(\zeta^{\mathfrak{m}} \left(n_{1}, \ldots , n_{p}  \right)\right) =  \quad \quad \delta_{2r+1=\mid n_{1}\mid + \cdots +\mid n_{i}\mid} \zeta^{\mathfrak{l}} \left(n_{1}, \cdots, n_{i} \right) &\otimes \zeta^{\mathfrak{m}} \left(  n_{i+1},\cdots, n_{p} \right)\\
 &\\
   \sum_{1 \leq i<j\leq p \atop \lbrace 2r+1 \leq \sum_{k=i}^{j} \mid n_{k}\mid -1\rbrace}  
 \left[ \begin{array}{l}
- \delta_{ \sum_{k=i}^{j-1} \mid n_{k}\mid \leq 2r+1}  \zeta^{\mathfrak{l}}_{2r+1-  \sum_{k=i}^{j-1}n_{k}} \left( n_{j-1}, \cdots, n_{i} \right) \\
 +  \delta_{ \sum_{k=i+1}^{j} n_{k} \leq 2r+1}  \zeta^{\mathfrak{l}}_{2r+1-  \sum_{k=i+1}^{j}n_{k}} \left( n_{i+1}, \ldots , n_{j} \right)
 \end{array} \right]  & \otimes \zeta^{\mathfrak{m}} \left( \cdots, \mlq\sum_{k=i}^{j}\mrq n_{k} -2r-1,\cdots \right) 
\end{array}$$
\end{lemm}
\noindent
A key point is that the Galois action respects the weight grading and depth filtration:
$$D_{2r+1} (\mathcal{H}_{n}) \subset \mathcal{L}_{2r+1} \otimes_{\mathbb{Q}} \mathcal{H}_{n-2r-1}, \quad \textrm{ and } \quad D_{2r+1} (\mathcal{F}_{p}^{\mathfrak{D}}  \mathcal{H}_{n}) \subset \mathcal{L}_{2r+1} \otimes_{\mathbb{Q}} \mathcal{F}_{p-1}^{\mathfrak{D}} \mathcal{H}_{n-2r-1}.$$
Since these derivations $D_{r}$ decrease the depth, it enables some depth recursion. Passing to the depth-graded, we could define:
$$gr^{\mathfrak{D}}_{p} D_{2r+1}: gr_{p}^{\mathfrak{D}} \mathcal{H} \rightarrow \mathcal{L}_{2r+1} \otimes gr_{p-1}^{\mathfrak{D}} \mathcal{H} \text{, as the composition } (id\otimes gr_{p-1}^{\mathfrak{D}}) \circ D_{2r+1 |gr_{p}^{\mathfrak{D}}\mathcal{H}}.$$
The terms in the left side of $gr^{\mathfrak{D}}_{p} D_{2r+1}$ have depth $1$ (by Lemma $\autoref{drz}$). \\
\\
\textsc{Remark}: The derivations on the $f_{i}$ alphabet, on $H$ (cf. $\ref{eq:phih}$) are deconcatenation-like:
\begin{equation}\label{eq:derf}
D_{2r+1} : \quad H_{n} \quad \longrightarrow \quad L_{2r+1} \otimes H_{n-2r-1} \quad \quad\text{ such that :}
\end{equation}
$$ f_{i_{1}} \cdots f_{i_{k}}\longmapsto\left\{
\begin{array}{ll}
  f_{i_{1}} \otimes f_{i_{2}} \ldots f_{i_{k}} & \text{ if } i_{1}=2r+1 .\\
  0 & \text{ else }.\\
\end{array}
\right.$$ 
As we will see in both section 4 and 5, for linear independance proof, an idea is to find an appropriate filtration on the conjectural basis, such that the derivations in the graded space act on this family, modulo some space, as deconcatenations.

\subsubsection{Ramification}

Looking at how periods of $\mathcal{MT}(\mathbb{Z})$ embed into periods of $\mathcal{MT}(\mathbb{Z}[\frac{1}{2}])$, is a fragment of the Galois descent ideas, pictured by:\\
\includegraphics[]{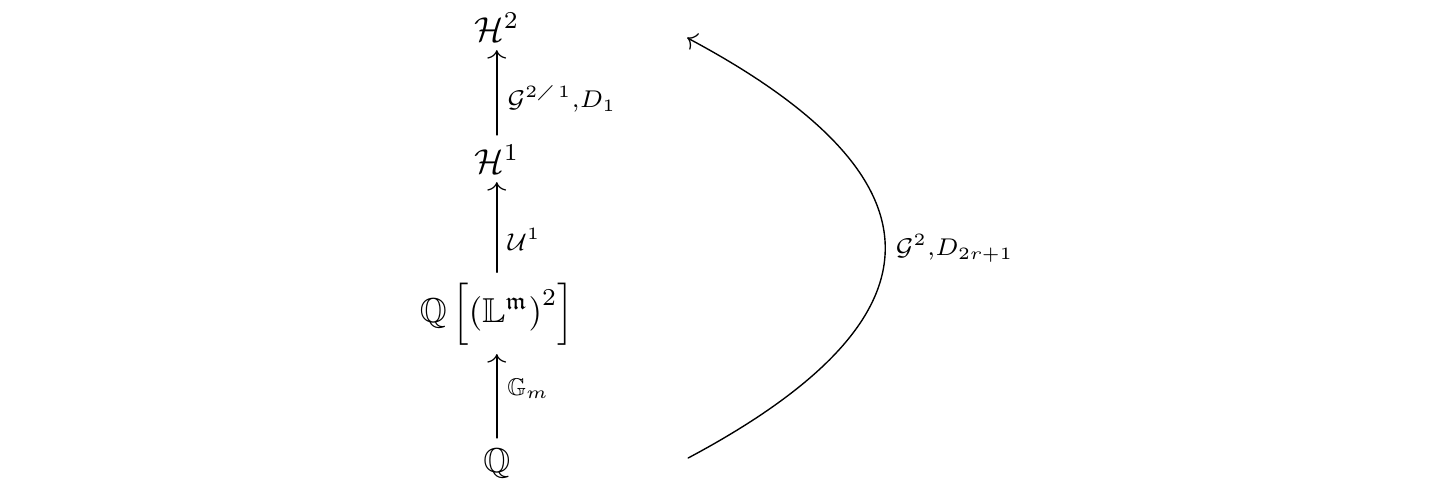}\\
A motivic ES is said to be \textit{unramified} if it can be expressed as a $\mathbb{Q}$-linear combinations of MMZV, i.e. is in $\mathcal{H}^{1}$; the corresponding ES is then also unramified, or \textit{honorary MZV} from D. Broadhurst terminology (cf. $\cite{BBB}$).\\
From motivic theory, we have at our disposal a recursive criterion for MES to be unramified:\footnote{Note that this criterion has been generalised to MZV at other roots of unity : cf. $\cite{Gl}$ which looks at different Galois descents, and higher ramification level.}
 \begin{coro}\label{criterehonoraire}
Let $\mathfrak{Z}\in\mathcal{H}^{2}$, a motivic Euler sum.\\
Then $\mathfrak{Z}\in\mathcal{H}^{1}$, i.e. $\mathfrak{Z}$ is a motivic multiple zeta value if and only if:
$$D_{1}(\mathfrak{Z})=0 \quad \textrm{  and  } \quad D_{2r+1}(\mathfrak{Z})\in\mathcal{H}^{1}.$$ 
\end{coro}
It enables to prove than a (motivic) ES is unramified by recursion on depth and weight. Some examples of unramified motivic ES, are provided in section $4$ or in Appendix D, with alternated patterns of even and odds.

\section{Definitions, antipode and hybrid relations}

\subsection{Star, Sharp versions}

In the motivic iterated integrals as defined in $\S 2.2$, $I^{\mathfrak{m}}(\cdots, a_{i}, \cdots)$, $a_{i}$ were in $\lbrace 0, \pm 1 \rbrace$. By linearity, we can extend it to $a_{i}\in \lbrace \pm \star, \pm \sharp\rbrace$, referring to the following differential forms:
$$\boldsymbol{\omega_{\pm\star}}\mathrel{\mathop:}= \omega_{\pm 1}- \omega_{0}=\frac{dt}{t(\pm t -1)} \quad \text{ and }  \quad \boldsymbol{\omega_{\pm\sharp}}\mathrel{\mathop:}=2 \omega_{\pm 1}-\omega_{0}=\frac{(t \pm 1)dt}{t(t\mp 1)}.$$
Now let introduce the variants of motivic Euler sums needed later \footnote{Possibly regularized with $(\ref{eq:shufflereg})$.}
\begin{defi}
Using the expression in terms of motivic iterated integrals ($\ref{eq:reprinteg}$), motivic Euler sums are, with $n_{i}\in\mathbb{Z}^{\ast}$, $\epsilon_{i}\mathrel{\mathop:}=\textrm{sign}(n_{i})$:
\begin{equation}\label{eq:mes}
\zeta^{\mathfrak{m}}_{k} \left(n_{1}, \ldots , n_{p} \right) \mathrel{\mathop:}= (-1)^{p}I^{\mathfrak{m}} \left(0; 0^{k}, \epsilon_{1}\cdots \epsilon_{p}, 0^{\mid n_{1}\mid -1} ,\ldots, \epsilon_{i}\cdots \epsilon_{p}, 0^{\mid n_{i}\mid -1} ,\ldots, \epsilon_{p}, 0^{\mid n_{p}\mid-1} ;1 \right).
\end{equation}
\begin{description}
\item[MES${\star}$] are defined by a similar integral representation as $(\ref{eq:mes})$ with $\omega_{\pm \star}$ (instead of $\omega_{\pm 1}$), $\omega_{0}$ and a $\omega_{\pm 1}$ at the beginning:
$$\zeta_{k}^{\star,\mathfrak{m}} \left(n_{1}, \ldots , n_{p} \right) \mathrel{\mathop:}= (-1)^{p} I^{\mathfrak{m}} \left(0; 0^{k}, \epsilon_{1}\cdots \epsilon_{p}, 0^{\mid n_{1}\mid-1}, \epsilon_{2}\cdots \epsilon_{p}\star, 0^{\mid n_{2}\mid -1}, \ldots, \epsilon_{p}\star, 0^{\mid n_{p}\mid-1} ;1 \right).$$
\item[MES${\star\star}$,] with only $\omega_{\pm \star}, \omega_{0}$:
$$\zeta_{k}^{\star\star,\mathfrak{m}} \left(n_{1}, \ldots , n_{p} \right) \mathrel{\mathop:}= (-1)^{p} I^{\mathfrak{m}}  \left(0; 0^{k}, \epsilon_{1}\cdots \epsilon_{p}\star, 0^{\mid n_{1}\mid-1}, \epsilon_{2}\cdots \epsilon_{p}\star, 0^{\mid n_{2}\mid-1}, \ldots, \epsilon_{p}\star, 0^{\mid n_{p}\mid-1} ;1 \right).$$
\item[MES${\sharp}$,] with $\omega_{\pm \sharp},\omega_{0} $ and a $\omega_{\pm 1}$ at the beginning:
$$\zeta_{k}^{\sharp,\mathfrak{m}} \left(n_{1}, \ldots , n_{p} \right) \mathrel{\mathop:}= 2 (-1)^{p} I^{\mathfrak{m}}  \left(0; 0^{k}, \epsilon_{1}\cdots \epsilon_{p}, 0^{\mid n_{1}\mid-1}, \epsilon_{2}\cdots \epsilon_{p}\sharp, 0^{\mid n_{2}\mid -1}, \ldots, \epsilon_{p}\sharp, 0^{\mid n_{p}\mid-1} ;1 \right).$$
\item[MES$\sharp\sharp$,] with only $\omega_{\pm \sharp}, \omega_{0}$:
$$\zeta_{k}^{\sharp\sharp,\mathfrak{m}} \left(n_{1}, \ldots , n_{p} \right) \mathrel{\mathop:}= (-1)^{p} I^{\mathfrak{m}}  \left(0; 0^{k}, \epsilon_{1}\cdots \epsilon_{p}\sharp, 0^{\mid n_{1}\mid-1},  \epsilon_{2}\cdots \epsilon_{p}\sharp, 0^{\mid n_{2}\mid-1}, \ldots, \epsilon_{p}\sharp, 0^{\mid n_{p}\mid-1} ;1 \right).$$
\end{description}
\end{defi}

\textsc{Remarks}:
\begin{itemize}
\item[$\cdot$] The Lie algebra of the fundamental group $\pi_{1}^{dR}(\mathbb{P}^{1}\diagdown \lbrace 0, 1, \infty\rbrace)=\pi_{1}^{dR}(\mathcal{M}_{0,4})$ is generated by $e_{0}, e_{1},e_{\infty}$ with the only condition than $e_{0}+e_{1}+e_{\infty}=0$\footnote{ For the case of motivic Euler sums, it is the Lie algebra generated by $e_{0}, e_{1}, e_{-1}, e_{\infty}$ with the only condition than $e_{0}+e_{1}+e_{-1}+e_{\infty}=0$. Note that $e_{i}$ corresponds to the class of the residue around $i$ in $H_{dR}^{1}(\mathbb{P}^{1} \diagdown \lbrace 0, \pm 1, \infty \rbrace)^{\vee}$. }. If we keep $e_{0}$ and $e_{\infty}$ as generators, instead of the usual $e_{0},e_{1}$, it leads towards MMZV  $^{\star\star}$ up to a sign, instead of MMZV since $-\omega_{0}+\omega_{1}- \omega_{\star}=0$. We could also choose $e_{1}$ and $e_{\infty}$ as generators, which leads to another version of MMZV, not well-spread. 
\item[$\cdot$] Euler $\star$ sums, which already appear in the literature, correspond to the summation $\ref{eq:defes}$ with large inequalities:
$$ \zeta^{\star}\left(n_{1}, \ldots , n_{p} \right) =  \sum_{0 < k_{1}\leq k_{2} \leq \cdots \leq k_{p}} \frac{\epsilon_{1}^{k_{1}} \cdots \epsilon_{p}^{k_{p}}}{k_{1}^{\mid n_{1}\mid} \cdots k_{p}^{\mid n_{p}\mid}}, \quad \epsilon_{i}\mathrel{\mathop:}=\textrm{sign}(n_{i}), \quad n_{i}\in\mathbb{Z}^{\ast}, n_{p}\neq 1. $$
\item[$\cdot$]  Yamamoto in $\cite{Ya}$ defined a polynomial in t, $\zeta^{t}(\bullet)$ which naturally interpolates $\zeta,\zeta^{\star}, \zeta^{\sharp}$ versions:  $\zeta^{t}(n_{1}, \cdots, n_{p})= \sum 2^{n_{+}} \zeta(n_{1} \circ \cdots \circ n_{p})$.\footnote{He proved in particular generalisations of sum and cyclic formulas.}
\item[$\cdot$] By linearity, for $A,B$ sequences in $\lbrace 0, \pm 1, \pm \star, \pm \sharp \rbrace$:
\begin{equation} \label{eq:miistarsharp}
 I^{\mathfrak{m}}(A, \pm \star, B)= I^{\mathfrak{m}}(A, \pm 1, B) - I^{\mathfrak{m}}(A, 0, B),  \text{ and }  I^{\mathfrak{m}}(A, \pm \sharp, B)=  2 I^{\mathfrak{m}}(A, \pm 1, B) - I^{\mathfrak{m}}(A, 0, B).
\end{equation} 
Therefore by linearity and $\shuffle$-regularisation $(\ref{eq:shufflereg})$, all these versions ($\star$, $\star\star$, $\sharp$ or $\sharp\sharp$) are $\mathbb{Q}$-linear combination of MES. Indeed, with $n_{+}$ the number of $+$ among $\circ$:\footnote{As before, the $\mlq + \mrq$ is a summation of absolute values while signs are multiplied.}
$$\begin{array}{llll}
\zeta ^{\star,\mathfrak{m}}(n_{1}, \ldots, n_{p})  &=& \sum_{\circ=\mlq + \mrq \text{ or } ,} & \zeta ^{\mathfrak{m}}(n_{1}\circ \cdots \circ n_{p}) \\
\\
\zeta ^{\mathfrak{m}}(n_{1}, \ldots, n_{p})  &=& \sum_{\circ=\mlq + \mrq \text{ or } ,} (-1)^{n_{+}} & \zeta ^{\star,\mathfrak{m}}(n_{1}\circ \cdots \circ n_{p})  \\
\\
 \zeta ^{ \sharp,\mathfrak{m}}(n_{1}, \ldots, n_{p}) &=& \sum_{\circ=\mlq + \mrq \text{ or } ,} 2^{p-n_{+}} & \zeta ^{\mathfrak{m}}(n_{1}\circ \cdots \circ n_{p}) \\
 \\
 \zeta ^{\mathfrak{m}}(n_{1}, \ldots, n_{p}) &=&  \sum_{\circ=\mlq + \mrq \text{ or } ,} (-1)^{n_{+}} 2^{-p} & \zeta ^{ \sharp,\mathfrak{m}}(n_{1}\circ \cdots \circ n_{p})   \\
 \\
  \zeta ^{\star\star,\mathfrak{m}}(n_{1}, \ldots, n_{p}) &=& \sum_{i=0}^{p-1} & \zeta ^{\star,\mathfrak{m}}_{\mid n_{1}\mid+\cdots+\mid n_{i}\mid}(n_{i+1}, \cdots , n_{p}) \\
  & = & \sum_{\circ=\mlq + \mrq \text{ or } ,\atop i=0}^{p-1} & \zeta^{\mathfrak{m}}_{\mid n_{1}\mid+\cdots+\mid n_{i}\mid}(n_{i+1}\circ \cdots \circ n_{p})\\
  \\
 \zeta ^{ \sharp\sharp,\mathfrak{m}}(n_{1}, \ldots, n_{p}) &=& \sum_{ i=0}^{p-1} & \zeta ^{\sharp,\mathfrak{m}}_{\mid n_{1}\mid+\cdots+\mid n_{i}\mid}(n_{i+1}, \cdots , n_{p})  \\
 & = & \sum_{\circ=\mlq + \mrq \text{ or } ,\atop i=0}^{p-1} 2^{p-i-n_{+}} & \zeta^{\mathfrak{m}}_{\mid n_{1}\mid+\cdots+\mid n_{i}\mid}(n_{i+1}\circ \cdots \circ n_{p}) \\
 \\
  \zeta ^{\star,\mathfrak{m}}(n_{1}, \ldots, n_{p})  &=& \zeta ^{\star\star,\mathfrak{m}}(n_{1}, \ldots, n_{p})& -\quad \zeta ^{\star\star,\mathfrak{m}}_{\mid n_{1}\mid}(n_{2}, \ldots, n_{p})
 \\
\\
\zeta ^{\sharp,\mathfrak{m}}(n_{1}, \ldots, n_{p}) &=& \zeta ^{\sharp\sharp}(n_{1}, \ldots, n_{p})
&-\quad \zeta ^{\sharp\sharp,\mathfrak{m}}_{\mid n_{1}\mid}(n_{2}, \ldots, n_{p}) \\ 

\end{array}$$
\end{itemize}
\texttt{Examples:} Expressing them as $\mathbb{Q}$ linear combinations of motivic Euler sums\footnote{To get rid of the $0$ in front of the MZV, as in the last example, we use the shuffle regularisation $\ref{eq:shufflereg}$.}:
$$\begin{array}{lll}
\zeta^{\star,\mathfrak{m}}(2,\overline{1},3) & = & -I^{\mathfrak{m}}(0;-1,0,-\star, \star,0,0; 1) \\
& = &  \zeta^{\mathfrak{m}}(2,\overline{1},3)+ \zeta^{\mathfrak{m}}(\overline{3},3)+ \zeta^{\mathfrak{m}}(2,\overline{4})+\zeta^{\mathfrak{m}}(\overline{6}) \\
\zeta^{\sharp,\mathfrak{m}}(2,\overline{1},3) &=& - 2 I^{\mathfrak{m}}(0;-1,0,-\sharp, \sharp,0,0; 1) \\
 &=& 8\zeta^{\mathfrak{m}}(2,\overline{1},3)+ 4\zeta^{\mathfrak{m}}(\overline{3},3)+ 4\zeta^{\mathfrak{m}}(2,\overline{4})+2\zeta^{\mathfrak{m}}(\overline{6})\\
 \zeta^{\star\star,\mathfrak{m}}(2,\overline{1},3) &=& -I^{\mathfrak{m}}(0;-\star,0,-\star, \star,0,0; 1) \\
 &=& \zeta^{\star,\mathfrak{m}}(2,\overline{1},3)+ \zeta^{\star,\mathfrak{m}}_{2}(\overline{1},3)+\zeta^{\star,\mathfrak{m}}_{3}(3) \\
  &=&  \zeta^{\star,\mathfrak{m}}(2, \overline{1}, 3)+\zeta^{\star,\mathfrak{m}}(\overline{3},3)+3 \zeta^{\star,\mathfrak{m}}(\overline{2},4)+6\zeta^{\star,\mathfrak{m}}(1, 5)-10\zeta^{\star,\mathfrak{m}}(6)\\
 &=& 11\zeta^{\mathfrak{m}}(\overline{6})+2\zeta^{\mathfrak{m}}(\overline{3}, 3)+\zeta^{\mathfrak{m}}(2, \overline{4})+\zeta^{\mathfrak{m}}(2,\overline{1}, 3)+3\zeta^{\mathfrak{m}}(\overline{2}, 4)+6\zeta^{\mathfrak{m}}(\overline{1}, 5)-10\zeta^{\mathfrak{m}}(6)\\
\end{array}$$

\textsc{Remark:} Remark that the shuffle relation, coming from the iterated integral representation is clearly \textit{motivic}, and true for MES. It was not obvious that the stuffle relation, coming from the multiplication of series, was motivic but it is now proven for MII.\footnote{It can be deduced from works by Goncharov on mixed Hodge structures, but was also proved directly by G. Racinet, in his thesis, or I. Souderes in $\cite{So}$ via blow-ups. } In particular, for $\textbf{ a }= (a_{1}, \ldots, a_{r} ), \textbf{ b }= (b_{1}, \ldots, b_{s})$ (inner order in $\textbf{ a }$, resp. $\textbf{ b }$ preserved):
$$\begin{array}{lll}
\zeta^{\mathfrak{m}}\left( \textbf{ a } \right) \zeta^{\mathfrak{m}}\left( \textbf{ b } \right) & = \sum_{c_{j}   = \left( a_{i} , b_{i'} \right)  \text{ or }\left(  a_{i}\mlq + \mrq b_{i'} \right)} & \zeta^{\mathfrak{m}}\left( c_{1}, \ldots, c_{m} \right)\\
\zeta^{\star,\mathfrak{m}}\left(\textbf{ a } \right) \zeta^{\star, \mathfrak{m}}\left(  \textbf{ b } \right) & =\sum_{c_{j} = (a_{i},b_{i'})  \text{ or }\left(a_{i}\mlq + \mrq b_{i'} \right) } (-1)^{r+s+m} &\zeta^{\star, \mathfrak{m}}\left( c_{1}, \ldots, c_{m}  \right)\\
\zeta^{\sharp , \mathfrak{m}}\left( \textbf{a} \right) \zeta^{\sharp, \mathfrak{m}}\left( \textbf{ b } \right)& =\sum_{k\geq 0, \left(  c_{j} \right)  = \left( a_{i}\mlq + \mrq \sum_{l=1}^{k} a_{i+l} \mlq + \mrq b_{i'+l}\right)  \atop \text{ or }  \left( b_{i'}\mlq + \mrq \sum_{l=1}^{k} a_{i+l} \mlq + \mrq b_{i'+l} \right) }  (-1)^{\frac{r+s-m}{2}} & \zeta^{\sharp, \mathfrak{m}}\left( c_{1}, \ldots, c_{m}\right).
\end{array}$$
Note that in the depth graded, stuffle corresponds to shuffle the sequences $\boldsymbol{a} , \boldsymbol{b}$.

\subsection{Antipode relation}

We have at our disposal two combinatorial Hopf algebra structures, which lead to some so-called \textit{antipodal} relations for MES in the coalgebra $\mathcal{L}$, i.e. modulo products.\\
First recall that if $A$ is a graded connected bialgebra, there exists an unique antipode S leading to a Hopf algebra structure: the graded map defined by, using Sweedler notations\footnote{There, $\Delta (x)= 1\otimes x+ x\otimes 1+ \sum x_{(1)}\otimes x_{(2)}= \Delta'(x)+ 1\otimes x+ x\otimes 1 .$}:
\begin{equation} \label{eq:Antipode} S(x)= -x-\sum S(x_{(1)}) \cdot x_{(2)},
\end{equation}
Hence, in the quotient $A/ A_{>0}\cdot A_{>0} $: $S(x) \equiv -x $.
\subsubsection{The $\shuffle$ Hopf algebra}

Let $X=\lbrace a_{1},\cdots, a_{n} \rbrace$. Then $A_{\shuffle}\mathrel{\mathop:}=\mathbb{Q} \langle X^{\times} \rangle$ the $\mathbb{Q}$-vector space of non commutative polynomials in $a_{i}$ is a connected graded Hopf algebra, called the \textit{shuffle Hopf algebra}, with the $\shuffle$ shuffle product, the deconcatenation coproduct $\Delta_{D}$ and antipode $S_{\shuffle}$:
\begin{equation} \label{eq:shufflecoproduct}  \Delta_{D}(a_{i_{1}}\cdots a_{i_{n}})= \sum_{k=0}^{n} a_{i_{1}}\cdots a_{i_{k}} \otimes a_{i_{k+1}} \cdots a_{i_{n}}, \quad S_{\shuffle} (a_{i_{1}} \cdots a_{i_{n}})= (-1)^{n} a_{i_{n}} \cdots a_{i_{1}}.
\end{equation}
Via the equivalence of category between $\mathbb{Q}$-Hopf algebra and $\mathbb{Q}$-Affine group scheme, it corresponds to:
\begin{equation} \label{eq:gpschshuffle} 
G=\text{Spec} A_{\shuffle} : R \rightarrow \text{Hom}(\mathbb{Q} \langle X \rangle, R)=\lbrace S\in R\langle\langle a_{i} \rangle\rangle \mid \Delta_{\shuffle} S= S\widehat{\otimes} S, \epsilon(S)=1 \rbrace,
\end{equation} 
where $\Delta_{\shuffle}$ is the coproduct dual to $\shuffle$: $\Delta_{\shuffle}(a_{i_{1}}\cdots a_{i_{n}})= \left( 1\otimes a_{i_{1}}+ a_{i_{1}}\otimes 1\right) \cdots \left( 1\otimes a_{i_{n}}+ a_{i_{n}}\otimes 1\right)$.
Let restrict now to $X=\lbrace 0,-1,+1 \rbrace$. The shuffle relation for motivic Euler sums amounts to say:
\begin{equation}\label{eq:shuffleim}
I^{\mathfrak{m}}(0; \cdot ; 1) \text{ is a morphism of Hopf algebra from } A_{\shuffle} \text{ to }  (\mathbb{R},\times): 
\end{equation}

\begin{lemm}[\textbf{Antipode $\shuffle$}] \label{eq:antipodeshuffle2}
In the coalgebra $\mathcal{L}$, with $w$ the weight, $W$ any word in $\lbrace 0,\pm 1\rbrace$, $\lbrace 0, \pm \star\rbrace$ or $\lbrace 0, \pm\sharp\rbrace$, and $\bullet$ standing for MES, MES$^{\star\star}$ or MES$^{\sharp\sharp}$:
$$\begin{array}{l}
 \zeta^{\bullet,\mathfrak{l}}_{n-1}\left( n_{1}, \ldots, n_{p} \right) \equiv(-1)^{w+1}\zeta^{\bullet,\mathfrak{l}}_{\mid n_{p}\mid -1}\left( n_{p-1}, \ldots, n_{1},\epsilon_{1}\cdots \epsilon_{p} \cdot n \right), \quad \text{with } \epsilon_{i}:=\textrm{sign}(n_{i}).\\
 \text{ } \\
 I^{\mathfrak{l}}(0;W;\epsilon)\equiv (-1)^{w} I^{\mathfrak{l}}(\epsilon;\widetilde{W};0) \equiv (-1)^{w+1} I^{\mathfrak{l}}(0; \widetilde{W}; \epsilon) , \quad\text{with } \widetilde{W} \text{ the \textit{reversed} word }.
\end{array}$$
\end{lemm}
\noindent
This antipode $\shuffle$ corresponds to a composition of path followed by a reverse of path in $\mathcal{L}$.
\begin{proof}
For motivic iterated integrals: $ S_{\shuffle} (I^{\mathfrak{m}}(0; a_{1}, \ldots, a_{n}; 1))= (-1)^{n}I^{\mathfrak{m}}(0; a_{n}, \ldots, a_{1}; 1)$.\footnote{Equivalently: $S_{\shuffle}\left( \zeta^{\bullet,\mathfrak{l}}_{n-1}\left( n_{1}, \ldots, n_{p}  \right) \right)  \equiv (-1)^{w}\zeta^{\bullet,\mathfrak{l}}_{\mid n_{p}\mid -1}\left( n_{p-1}, \ldots, n_{1}, \epsilon_{1}\cdot\ldots\cdot\epsilon_{p} \cdot n  \right)$.} Then, looking at the antipode recursive formula $\eqref{eq:Antipode}$ in the coalgebra $\mathcal{L}$:
$$ S_{\shuffle} (I^{\mathfrak{l}}(0; a_{1}, \ldots, a_{n}; 1))\equiv - I^{\mathfrak{l}}(0; a_{1}, \ldots, a_{n}; 1).$$
\end{proof}

\subsubsection{The $\ast$ Hopf algebra}
Let $Y=\lbrace \cdots, y_{-n}, \ldots, y_{-1}, y_{1},\cdots, y_{n}, \cdots \rbrace$ an infinite alphabet, and $y_{0}=1$ the empty word. Then $A_{\ast}\mathrel{\mathop:}=\mathbb{Q} \langle Y^{\times} \rangle$ the non commutative polynomials in $y_{i}$ is a graded connected Hopf algebra called the\textit{ stuffle Hopf algebra}, with the stuffle $\ast$ product and the coproduct:
\begin{equation} \label{eq:stufflecoproduct} 
\Delta_{D\ast}(y_{n_{1}} \cdots y_{n_{p}})= \sum_{} y_{n_{1}} \cdots y_{n_{i}}\otimes y_{n_{i+1}}, \ldots, y_{n_{p}}, \quad n_{i}\in \mathbb{Z}^{\ast}. 
\end{equation}
The completed dual is the Hopf algebra of series $\mathbb{Q}\left\langle \left\langle Y \right\rangle \right\rangle $ with the coproduct:
$$\Delta_{\ast}(y_{n})= \sum_{ k =0 \atop \textrm{sign}(n)=\epsilon_{1}\epsilon_{2}}^{\mid n \mid} y_{\epsilon_{1} k} \otimes y_{\epsilon_{2}( n-k)}.$$
Stuffle for MES amounts to say that
$\zeta^{\mathfrak{m}}(\cdot)$ is a morphism of Hopf algebra from $A_{\ast}$ to $(\mathbb{R},\times)$.\\
Now, let introduce the notations, with $n_{i}\in\mathbb{Z}^{\ast}$:\footnote{Here $\star$ resp. $\sharp$ refers naturally to the Euler $\star$ resp. $\sharp$, sums, as we see in the next lemma.}
$$(y_{n_{1}} \cdots y_{n_{p}})^{\star} \mathrel{\mathop:}=  \sum_{1=i_{0}< i_{1} < \cdots < i_{k-1}\leq  i_{k+1}=p \atop k\geq 0} y_{n_{i_{0}}\mlq + \mrq\cdots \mlq + \mrq n_{i_{1}-1}} \cdots y_{n_{i_{j}}\mlq + \mrq\cdots \mlq + \mrq n_{i_{j+1}-1}} \cdots  y_{n_{i_{k}}\mlq + \mrq \cdots \mlq + \mrq n_{i_{k+1}}}.$$
$$(y_{n_{1}} \cdots y_{n_{p}})^{\sharp} \mathrel{\mathop:}=  \sum_{1=i_{0}< i_{1} < \cdots < i_{k-1}\leq  i_{k+1}=p \atop k\geq 0} 2^{k+1} y_{n_{i_{0}}\mlq + \mrq\cdots \mlq + \mrq n_{i_{1}-1}} \cdots y_{n_{i_{j}} \mlq + \mrq \cdots \mlq + \mrq n_{i_{j+1}-1}} \cdots  y_{n_{i_{k}}\mlq + \mrq \cdots \mlq + \mrq n_{i_{k+1}}},$$
The operation $\mlq + \mrq$ still indicates that signs are multiplied whereas absolute values are summed. It is straightforward to check that:
\begin{equation} 
\Delta_{D\ast}(w^{\star})=(\Delta_{D\ast}(w))^{\star} , \quad \text{ and } \quad  \Delta_{D\ast}(w^{\sharp})=(\Delta_{D\ast}(w))^{\sharp}.
\end{equation}

\begin{lemm}[\textbf{Antipode $\ast$}]
In the coalgebra $\mathcal{L}$, with $n_{i}\in\mathbb{Z}^{\ast}$
$$\zeta^{\mathfrak{l}}_{n-1}(n_{1}, \ldots, n_{p}) \equiv (-1)^{p+1}\zeta^{\star,\mathfrak{l}}_{n-1}(n_{p}, \ldots, n_{1}).$$
$$\zeta^{\sharp,\mathfrak{l}}_{n-1}(n_{1}, \ldots, n_{p})\equiv (-1)^{p+1}\zeta^{\sharp,\mathfrak{l}}_{n-1}(n_{p}, \ldots, n_{1}).$$
\end{lemm}
\begin{proof}
By recursion, using $\eqref{eq:Antipode}$, and the following 
identity (left to the reader):
$$\sum_{i=0}^{p-1} (-1)^{i}(y_{n_{i}} \cdots y_{n_{1}})^{\star} \ast (y_{n_{i+1}} \cdots y_{n_{p}})= -(-1)^{p} (y_{n_{p}} \cdots y_{n_{1}})^{\star}, $$
we deduce: $S_{\ast} (y_{n_{1}} \cdots y_{n_{p}})= (-1)^{p} (y_{n_{p}} \cdots y_{n_{1}})^{\star}$. Similarly:
$$\begin{array}{ll}
S_{\ast}((y_{n_{1}} \cdots y_{n_{p}})^{\sharp}) &= -\sum_{i=0}^{n-1} S_{\ast}((y_{n_{1}} \cdots y_{n_{i}})^{\sharp}) \ast (y_{n_{i+1}} \cdots y_{n_{p}})^{\sharp}\\
& =-\sum_{i=0}^{n-1} (-1)^{i}(y_{n_{i}} \cdots y_{n_{1}})^{\sharp} \ast (y_{n_{i+1}} \cdots y_{n_{p}})^{\sharp}\\
&=(-1)^{p}(y_{n_{p}} \cdots y_{n_{1}})^{\sharp}
\end{array}$$
Since $\zeta^{\mathfrak{m}}(\cdot)$ is a morphism of Hopf algebra, the lemma straightly follows.
\end{proof}

\subsection{Hybrid relation in $\mathcal{L}$}

In this part, we look at a new relation called \textit{hybrid relation} between motivic Euler sums in the coalgebra $\mathcal{L}$, i.e. modulo products, which comes from the motivic version of the octagon relation (cf. $\cite{EF}$)\footnote{The hexagon relation for MZV (i.e. $e^{i\pi e_{0}} \Phi(e_{\infty}, e_{0}) e^{i\pi e_{\infty}} \Phi(e_{1},e_{\infty}) e^{i\pi e_{1}}\Phi( e_{0},e_{1})=1$) turns to an octagon for MZV at greater roots of unity. It is pictured by: \begin{center}
\scalebox{0.5}{\includegraphics[]{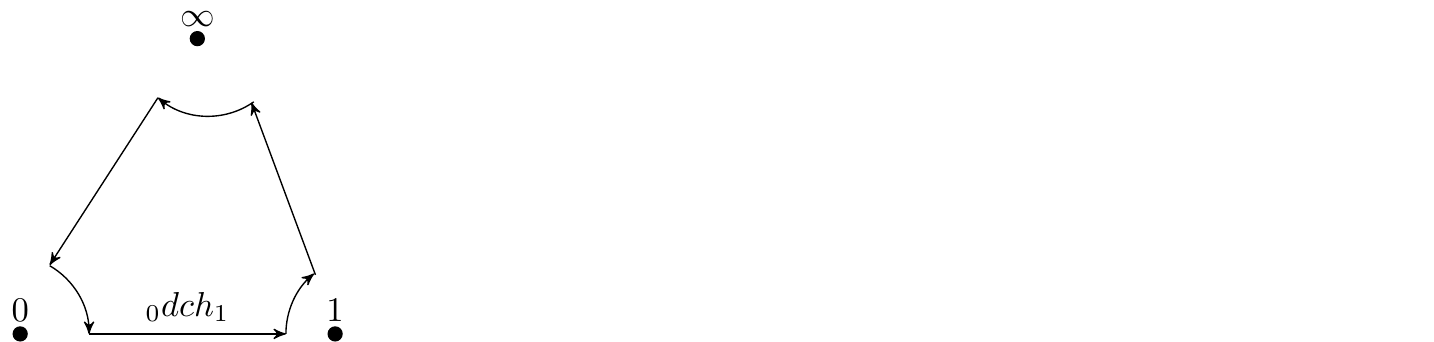}} $\quad \quad$ resp. $\quad \quad \quad $\scalebox{0.5}{\includegraphics[]{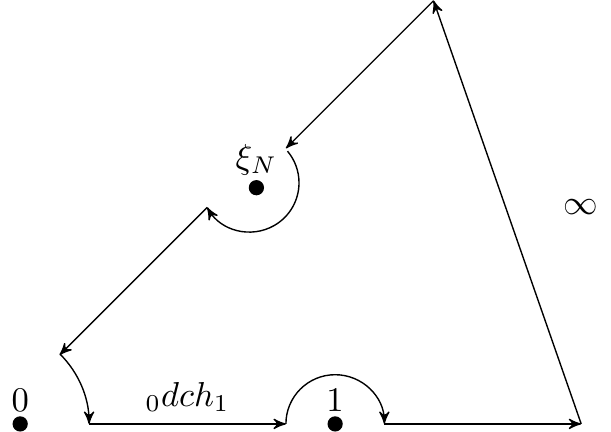}}
\end{center} }. It corresponds to a contractible path in the moduli space of genus $0$ curves with $4$ ordered marked points $\mathcal{M}_{0,4}=\mathbb{P}^{1}\diagdown\lbrace 0,\pm 1, \infty\rbrace$. Seeing the path on the Riemann sphere, for ES, it is pictured by:
\begin{figure}[H]
\centering
\scalebox{0.8}{\includegraphics[]{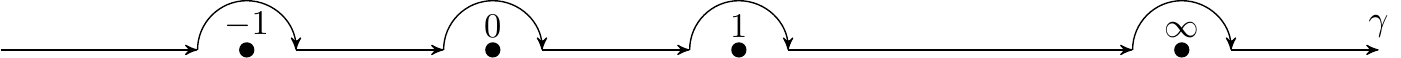}}
\caption{Octagon relation} \label{fig:octagon}
\end{figure}
\noindent This relation is motivic: it is valid for the motivic Drinfeld associator $\Phi^{\mathfrak{m}}$ ($\ref{eq:associator}$) when replacing $2 i \pi$ by the Lefschetz motivic period $\mathbb{L}^{\mathfrak{m}}$: \\
\begin{equation}\label{eq:octagon}
\hspace*{-0.5cm} e^{\frac{\mathbb{L}^{\mathfrak{m}} e_{-1}}{2}} \Phi^{\mathfrak{m}}(e_{0}, e_{-1},e_{1})^{-1} e^{ \frac{\mathbb{L}^{\mathfrak{m}}e_{0}}{2}} \Phi^{\mathfrak{m}}(e_{0},e_{1},e_{-1}) e^{\frac{\mathbb{L}^{\mathfrak{m}} e_{1}}{2}}\Phi^{\mathfrak{m}}( e_{\infty},e_{1},e_{-1})^{-1} e^{\frac{\mathbb{L}^{\mathfrak{m}} e_{\infty}}{2}}\Phi^{\mathfrak{m}}( e_{\infty},e_{-1},e_{1})=1
\end{equation}
Remind that $e_{0} + e_{1}  + e_{-1} +e_{\infty} =0$.\\
\\
Let $X= \mathbb{P}^{1}\diagdown \left\lbrace 0, \pm 1, \infty \right\rbrace $. The action of the \textit{real Frobenius} $\boldsymbol{\mathcal{F}_{\infty}}$ on $X(\mathbb{C})$ is induced by complex conjugation. The real Frobenius acts on the Betti realization $\pi^{B}(X (\mathbb{C}))$\footnote{It is compatible with the groupoid structure of $\pi^{B}$, and the local monodromy. }, and induces an involution on motivic periods, compatible with the Galois action:
$$\mathcal{F}_{\infty}: \mathcal{P}_{\mathcal{MT}(\mathbb{Z}[\frac{1}{2}])}^{\mathfrak{m}} \rightarrow\mathcal{P}_{\mathcal{MT}(\mathbb{Z}[\frac{1}{2}])}^{\mathfrak{m}}.$$
The Lefschetz motivic period $\mathbb{L}^{\mathfrak{m}}$ is anti-invariant by $\mathcal{F}_{\infty}$: $\mathcal{F}_{\infty} \mathbb{L}^{\mathfrak{m}}= -\mathbb{L}^{\mathfrak{m}},$
whereas terms corresponding to real paths in Figure $\ref{fig:octagon}$, such as Drinfeld associator terms, are $\mathcal{F}_{\infty}$-invariant.\\
Recall that the space of motivic periods of $\mathcal{MT}\left( \mathbb{Z}[\frac{1}{2}]\right)$ decomposes as (cf. $\ref{eq:periodgeomr}$):
\begin{equation}\label{eq:perioddecomp2}
\mathcal{P}_{\mathcal{MT}\left( \mathbb{Z}[\frac{1}{2}]\right)}^{\mathfrak{m}}= \mathcal{H}^{2} \oplus  \mathcal{H}^{2}. \mathbb{L}^{\mathfrak{m}}, \quad \text{ where } \begin{array}{l}
\mathcal{H}^{2} \text{ is } \mathcal{F}_{\infty} \text{ invariant} \\
\mathcal{H}^{2}. \mathbb{L}^{\mathfrak{m}} \text{ is } \mathcal{F}_{\infty} \text{ anti-invariant} 
\end{array}.
\end{equation}
The linearized $\mathcal{F}_{\infty}$-anti-invariant part of the octagon relation $\ref{eq:octagon}$ leads to the following: 

\begin{theo}\label{hybrid}
In the coalgebra $\mathcal{L}^{2}$, with $n_{i}\in \mathbb{Z}^{\ast}$, $w$ the weight:
$$\zeta^{\mathfrak{l}}_{k}\left( n_{0}, n_{1},\ldots, n_{p} \right) + \zeta^{\mathfrak{l}}_{\mid n_{0} \mid +k}\left( n_{1}, \ldots, n_{p} \right) \equiv (-1)^{w+1}\left(  \zeta^{\mathfrak{l}}_{k}\left( n_{p}, \ldots, n_{1}, n_{0} \right) + \zeta^{\mathfrak{l}}_{k+\mid n_{p}\mid}\left( n_{p-1}, \ldots, n_{1},n_{0} \right)\right)$$
Equivalently, for $W$ any word in $\lbrace 0, \pm 1 \rbrace$, with $\widetilde{W}$ the reversed word:
$$\begin{array}{llll}
(\boldsymbol{\dagger}) & I^{\mathfrak{l}} (0; 0^{k}, \star, W; 1)& \equiv I^{\mathfrak{l}} (0; W, \star, 0^{k}; 1)& \equiv (-1)^{w+1} I^{\mathfrak{l}} (0; 0^{k}, \star, \widetilde{W}; 1)\\
&&&\\
(\boldsymbol{\ddagger}) & I^{\mathfrak{l}} (0; 0^{k}, -\star, W; 1)& \equiv I^{\mathfrak{l}} (0;- W, -\star, 0^{k}; 1)&\equiv (-1)^{w+1} I^{\mathfrak{l}} (0; 0^{k}, -\star, -\widetilde{W}; 1) 

\end{array}$$
\end{theo}
The proof is given below, firstly for $k=0$, using octagon relation ($\ref{eq:octagon}$). The generalization for any $k >0$ is deduced directly from the shuffle regularization $(\ref{eq:shufflereg})$.\\
\\
\textsc{Remarks}:
\begin{itemize} 
\item[$\cdot$] Equivalently, this statement is true for $W$ any word in $\lbrace 0, \pm \star \rbrace$, by linearity (cf. $\ref{eq:miistarsharp}$).
\item[$\cdot$] The point of view adopted by Francis Brown in $\cite{Br3}$, and its use of commutative polynomials (also seen in Ecalle work), applied in the coalgebra $\mathcal{L}$ leads to a new proof of Theorem $\ref{hybrid}$ in the case of MMZV only (not for MES).
\end{itemize}
Since Antipode $\ast$ expresses $\zeta^{\mathfrak{l}}_{n-1}(n_{1}, \ldots, n_{p})+(-1)^{p} \zeta^{\mathfrak{l}}_{n-1}(n_{p}, \ldots, n_{1})$ in terms of smaller depth (cf. Lemma $3.3$), a MES can be expressed by smaller depth when weight and depth have not the same parity. This \textit{depth-drop phenomena}, for motivic ES:\footnote{Erik Panzer recently generalised this parity theorem for MZV at roots of unity (not motivic there), which appears as a special case of some functional equations of polylogarithms in several variables: $\cite{Pa}$.}
\begin{coro}\label{hybridc}
If $w+p$ odd, a motivic Euler sum in $\mathcal{L}$ is reducible to smaller depth:
$$2\zeta^{\mathfrak{l}}_{n-1}(n_{1}, \ldots, n_{p}) \equiv$$
$$-\zeta^{\mathfrak{l}}_{n+\mid n_{1}\mid -1}(n_{2}, \ldots, n_{p})+(-1)^{p} \zeta^{\mathfrak{l}}_{n+\mid n_{p}\mid -1}(n_{p-1}, \ldots, n_{1})+\sum_{\circ=+ \text{ or } ,\atop \text{at least one } +} (-1)^{p+1} \zeta^{\mathfrak{l}}_{n-1}(n_{p}\circ \cdots \circ n_{1}).$$
\end{coro}

\paragraph{Proof of Theorem $\ref{hybrid}$}
First, the octagon relation ($\ref{eq:octagon}$) is equivalent to:
\begin{lemm} 
\begin{itemize}
\item[$(i)$] The octagon relation, in  $\mathcal{P}_{\mathcal{MT}\left( \mathbb{Z}[\frac{1}{2}]\right)}^{\mathfrak{m}}\left\langle \left\langle e_{0}, e_{1}, e_{-1}\right\rangle \right\rangle $:
\begin{equation}\label{eq:octagon21}
\Phi^{\mathfrak{m}}(e_{0}, e_{1},e_{-1}) e^{\frac{\mathbb{L}^{\mathfrak{m}} e_{0}}{2}} \Phi^{\mathfrak{m}}(e_{-1}, e_{0},e_{\infty}) e^{\frac{\mathbb{L}^{\mathfrak{m}} e_{-1}}{2}} \Phi^{\mathfrak{m}}(e_{\infty}, e_{-1},e_{1}) e^{\frac{\mathbb{L}^{\mathfrak{m}} e_{\infty}}{2}} \Phi^{\mathfrak{m}}(e_{1}, e_{\infty},e_{0}) e^{\frac{\mathbb{L}^{\mathfrak{m}} e_{1}}{2}}  =1,
\end{equation}

\item[$(ii)$] The linearized octagon relation, 
\begin{multline}\label{eq:octagonlin}
 - e_{0} \Phi^{\mathfrak{l}}(e_{-1}, e_{0},e_{\infty})+  \Phi^{\mathfrak{l}}(e_{-1}, e_{0},e_{\infty})e_{0} +(e_{0}+e_{-1}) \Phi^{\mathfrak{l}}(e_{\infty}, e_{-1},e_{1}) - \Phi^{\mathfrak{l}}(e_{\infty}, e_{-1},e_{1}) (e_{0}+e_{-1})\\
   - e_{1} \Phi^{\mathfrak{l}}(e_{1}, e_{\infty},e_{0}) + \Phi^{\mathfrak{l}}(e_{1}, e_{\infty},e_{0}) e_{1}  \equiv 0.
\end{multline}
\end{itemize}
\end{lemm}
\begin{proof}
\begin{itemize}
\item[$(i)$] In order to deduce $\ref{eq:octagon21}$ from $\ref{eq:octagon}$, let's first remark that:
$$ \Phi^{\mathfrak{m}}(e_{0}, e_{1},e_{-1})= \Phi^{\mathfrak{m}}(e_{1}, e_{0},e_{\infty})^{-1} .$$
Indeed, the coefficient in $\Phi^{\mathfrak{m}}(e_{1}, e_{0},e_{\infty})$ of a word $e_{0}^{a_{0}} e_{\eta_{1}} e_{0}^{a_{1}} \cdots e_{\eta_{r}} e_{0}^{a_{r}}$, where $\eta_{i}\in \lbrace\pm 1 \rbrace$ is (cf. Appendice $C$):
 $$ I^{\mathfrak{m}} \left(0;  (\omega_{1}-\omega_{-1})^{a_{0}} (-\omega_{\mu_{1}})  (\omega_{1}-\omega_{-1})^{a_{1}} \cdots  (-\omega_{\mu_{r}})(\omega_{1}-\omega_{-1})^{a_{r}} ;1 \right) \texttt{ with } \mu_{i}\mathrel{\mathop:}= \left\lbrace \begin{array}{ll}
-\star& \texttt{if } \eta_{i}=1\\
-1 & \texttt{if } \eta_{i}=-1
\end{array}  \right. .$$
Applying the homography $\phi_{\tau\sigma}= \phi_{\tau\sigma}^{-1}:t \mapsto \frac{1-t}{1+t}$ (cf. Appendix $(\ref{homography2})$) to this gives:
$$ I^{\mathfrak{m}} \left(1;  \omega_{0}^{a_{0}} \omega_{\eta_{1}} \omega_{0}^{a_{1}} \cdots  \omega_{\eta_{r}}  \omega_{0}^{a_{r}} ;0 \right).$$
Hence, summing over words $w$ in $e_{0},e_{1},e_{-1}$:
$$ \Phi^{\mathfrak{m}}(e_{1}, e_{0},e_{\infty})= \sum I^{\mathfrak{m}}(1; w; 0) w$$
Therefore:
$$\Phi^{\mathfrak{m}}(e_{0}, e_{1},e_{-1})\Phi^{\mathfrak{m}}(e_{1}, e_{0},e_{\infty})= \sum_{w, w=uv} I^{\mathfrak{m}}(0; u; 1) I^{\mathfrak{m}}(1; v; 0) w= 1.$$
We used the composition formula for iterated integral to conclude, since for $w$ non empty, $\sum_{w=uv} I^{\mathfrak{m}}(0; u; 1) I^{\mathfrak{m}}(1; v; 0)= I^{\mathfrak{m}}(0; w; 0) =0$. Similarly:
$$\Phi^{\mathfrak{m}}(e_{0}, e_{-1},e_{1})= \Phi^{\mathfrak{m}}(e_{-1}, e_{0},e_{\infty})^{-1} , \quad \text{ and } \quad  \Phi^{\mathfrak{m}}(e_{\infty}, e_{1},e_{-1})= \Phi^{\mathfrak{m}}(e_{1}, e_{\infty},e_{0})^{-1}.$$ 

\item[$(ii)$] Now, let consider both paths on the Riemann sphere $\gamma$ and $\overline{\gamma}$, its conjugate: \footnote{Path $\gamma$ corresponds to the cycle $\sigma$, $1 \mapsto \infty \mapsto -1 \mapsto 0 \mapsto 1$ (cf. in Appendix $\ref{homography2}$). Obviously in the figure the position of both path is not completely accurate in order to distinguish them.} \\
\\
\includegraphics[]{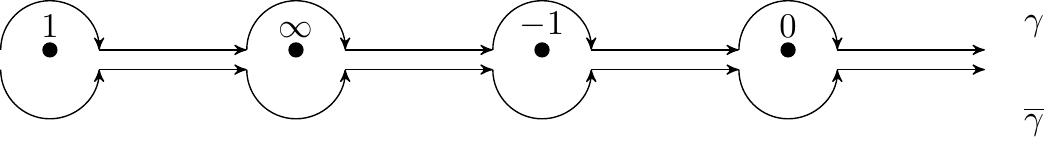}\\
Applying $(id-\mathcal{F}_{\infty})$ to the octagon identity $\ref{eq:octagon21}$ \footnote{Applying $\mathcal{F}_{\infty}$ to the path $\gamma$ corresponds to the path $\overline{\gamma}$ represented.} leads to:
\begin{small}
\begin{multline}\label{eq:octagon22}
\Phi^{\mathfrak{m}}(e_{0}, e_{1},e_{-1}) e^{\frac{\mathbb{L}^{\mathfrak{m}} e_{0}}{2}} \Phi^{\mathfrak{m}}(e_{-1}, e_{0},e_{\infty}) e^{\frac{\mathbb{L}^{\mathfrak{m}} e_{-1}}{2}} \Phi^{\mathfrak{m}}(e_{\infty}, e_{-1},e_{1}) e^{\frac{\mathbb{L}^{\mathfrak{m}} e_{\infty}}{2}} \Phi^{\mathfrak{m}}(e_{1}, e_{\infty},e_{0}) e^{\frac{\mathbb{L}^{\mathfrak{m}} e_{1}}{2}}\\
-\Phi^{\mathfrak{m}}(e_{0}, e_{1},e_{-1}) e^{-\frac{\mathbb{L}^{\mathfrak{m}} e_{0}}{2}} \Phi^{\mathfrak{m}}(e_{-1}, e_{0},e_{\infty}) e^{-\frac{\mathbb{L}^{\mathfrak{m}} e_{-1}}{2}} \Phi^{\mathfrak{m}}(e_{\infty}, e_{-1},e_{1}) e^{-\frac{\mathbb{L}^{\mathfrak{m}} e_{\infty}}{2}} \Phi^{\mathfrak{m}}(e_{1}, e_{\infty},e_{0}) e^{-\frac{\mathbb{L}^{\mathfrak{m}} e_{1}}{2}}=0.
\end{multline}
\end{small}
By $(\ref{eq:perioddecomp2})$, the left side of $(\ref{eq:octagon22})$, being anti-invariant by $\mathcal{F}_{\infty}$, lies in $ \mathcal{H}^{2}\cdot \mathbb{L}^{\mathfrak{m}} \left\langle \left\langle e_{0}, e_{1}, e_{-1} \right\rangle \right\rangle $. Consequently, we can divide it by $\mathbb{L}^{\mathfrak{m}}$ and consider its projection $\pi^{\mathcal{L}}$ in the coalgebra $\mathcal{L} \left\langle \left\langle e_{0}, e_{1}, e_{-1} \right\rangle \right\rangle $, which gives firstly:
\begin{small}
\begin{multline}\label{eq:octagon23}\hspace*{-0.5cm} 
0=\Phi^{\mathfrak{l}}(e_{0}, e_{1},e_{-1}) \pi^{\mathcal{L}} \left( (\mathbb{L}^{\mathfrak{m}})^{-1} \left[ e^{\frac{\mathbb{L}^{\mathfrak{m}} e_{0}}{2}} e^{\frac{\mathbb{L}^{\mathfrak{m}} e_{-1}}{2}}e^{\frac{\mathbb{L}^{\mathfrak{m}} e_{\infty}}{2}} e^{\frac{\mathbb{L}^{\mathfrak{m}} e_{1}}{2}} - e^{-\frac{\mathbb{L}^{\mathfrak{m}} e_{0}}{2}} e^{-\frac{\mathbb{L}^{\mathfrak{m}} e_{-1}}{2}}e^{-\frac{\mathbb{L}^{\mathfrak{m}} e_{\infty}}{2}} e^{-\frac{\mathbb{L}^{\mathfrak{m}} e_{1}}{2}} \right] \right)  \\
\hspace*{-0.5cm}  +\pi^{\mathcal{L}} \left( (\mathbb{L}^{\mathfrak{m}})^{-1} \left[  e^{\frac{\mathbb{L}^{\mathfrak{m}} e_{0}}{2}}  \Phi^{\mathfrak{l}}(e_{-1}, e_{0},e_{\infty})  e^{\frac{\mathbb{L}^{\mathfrak{m}} e_{-1}}{2}} e^{\frac{\mathbb{L}^{\mathfrak{m}} e_{\infty}}{2}}  e^{\frac{\mathbb{L}^{\mathfrak{m}} e_{1}}{2}}- e^{-\frac{\mathbb{L}^{\mathfrak{m}} e_{0}}{2}}  \Phi^{\mathfrak{l}}(e_{-1}, e_{0},e_{\infty})  e^{-\frac{\mathbb{L}^{\mathfrak{m}} e_{-1}}{2}} e^{-\frac{\mathbb{L}^{\mathfrak{m}} e_{\infty}}{2}}  e^{-\frac{\mathbb{L}^{\mathfrak{m}} e_{1}}{2}} \right]  \right) \\
 \hspace*{-0.5cm} + \pi^{\mathcal{L}} \left( (\mathbb{L}^{\mathfrak{m}})^{-1} \left[  e^{\frac{\mathbb{L}^{\mathfrak{m}} e_{0}}{2}}  e^{\frac{\mathbb{L}^{\mathfrak{m}} e_{-1}}{2}}  \Phi^{\mathfrak{l}}(e_{\infty}, e_{-1},e_{1}) e^{\frac{\mathbb{L}^{\mathfrak{m}} e_{\infty}}{2}}  e^{\frac{\mathbb{L}^{\mathfrak{m}} e_{1}}{2}} -  e^{-\frac{\mathbb{L}^{\mathfrak{m}} e_{0}}{2}}  e^{-\frac{\mathbb{L}^{\mathfrak{m}} e_{-1}}{2}}  \Phi^{\mathfrak{l}}(e_{\infty}, e_{-1},e_{1}) e^{-\frac{\mathbb{L}^{\mathfrak{m}} e_{\infty}}{2}}  e^{-\frac{\mathbb{L}^{\mathfrak{m}} e_{1}}{2}} \right]  \right) \\  
\hspace*{-0.5cm}  +\pi^{\mathcal{L}} \left( (\mathbb{L}^{\mathfrak{m}})^{-1} \left[    e^{\frac{\mathbb{L}^{\mathfrak{m}} e_{0}}{2}}  e^{\frac{\mathbb{L}^{\mathfrak{m}} e_{-1}}{2}} e^{\frac{\mathbb{L}^{\mathfrak{m}} e_{\infty}}{2}} \Phi^{\mathfrak{l}}(e_{1}, e_{\infty},e_{0})  e^{\frac{\mathbb{L}^{\mathfrak{m}} e_{1}}{2}} -  e^{-\frac{\mathbb{L}^{\mathfrak{m}} e_{0}}{2}}  e^{-\frac{\mathbb{L}^{\mathfrak{m}} e_{-1}}{2}} e^{-\frac{\mathbb{L}^{\mathfrak{m}} e_{\infty}}{2}} \Phi^{\mathfrak{l}}(e_{1}, e_{\infty},e_{0})  e^{-\frac{\mathbb{L}^{\mathfrak{m}} e_{1}}{2}}  \right] \right)
\end{multline}
\end{small}
The first line is zero (since $e_{0}+e_{1}+ e_{-1}+e_{\infty}=0$) whereas each other line will contribute by two terms, in order to give $(\ref{eq:octagonlin})$. Indeed, the projection $\pi^{\mathcal{L}}(x)$, when seeing $x$ as a polynomial (with only even powers) in $\mathbb{L}^{\mathfrak{m}}$, only keep the constant term; hence, for each term, only one of the exponentials above $e^{x}$ contributes by its linear term i.e. $x$, while the others contribute simply by $1$. For instance, if we examine carefully the second line of $(\ref{eq:octagon23})$, we get:
$$\begin{array}{ll}
= & e_{0}  \Phi^{\mathfrak{l}}(e_{-1}, e_{0},e_{\infty}) + \Phi^{\mathfrak{l}}(e_{-1}, e_{0},e_{\infty}) (e_{-1}+e_{\infty}+e_{1})\\
&  - (-e_{0})  \Phi^{\mathfrak{l}}(e_{-1}, e_{0},e_{\infty}) -    \Phi^{\mathfrak{l}}(e_{-1}, e_{0},e_{\infty}) ( - e_{-1}- e_{\infty}- e_{1}) \\
= & 2 \left[ e_{0}  \Phi^{\mathfrak{l}}(e_{-1}, e_{0},e_{\infty}) -  \Phi^{\mathfrak{l}}(e_{-1}, e_{0},e_{\infty}) e_{0}\right] 
\end{array}.$$
Similarly, the third line of  $(\ref{eq:octagon23})$ is equal to $(e_{0}+e_{-1}) \Phi^{\mathfrak{l}}(e_{\infty}, e_{-1},e_{1}) - \Phi^{\mathfrak{l}}(e_{\infty}, e_{-1},e_{1}) (e_{0}+e_{-1})$ and the last line is equal to $ -e_{1} \Phi^{\mathfrak{l}}(e_{1}, e_{\infty},e_{0}) + \Phi^{\mathfrak{l}}(e_{1}, e_{\infty},e_{0}) e_{1}$. Therefore, $(\ref{eq:octagon23})$ is equivalent to $(\ref{eq:octagonlin})$, as claimed.
\end{itemize}
\end{proof}

This linearized octagon relation $\ref{eq:octagonlin}$, while looking at the coefficient of a specific word in $\lbrace e_{0},e_{1}, e_{-1}\rbrace$, provides an identity between some $ \zeta^{\star\star,\mathfrak{l}} (\bullet)$ and $\zeta^{\mathfrak{l}} (\bullet)$ in the coalgebra $\mathcal{L}$. The different identities obtained in this way are detailed in Appendice C, and two of these are used in the following proof.

\begin{proof}[Proof of Theorem $\ref{hybrid}$]
The identity with MES is equivalent to, in terms of motivic iterated integrals, according to the sign of $\prod_{i=0}^{p} \epsilon_{i}=\pm 1$:
$$(\boldsymbol{\dagger}) \quad I^{\mathfrak{l}} (0; 0^{k}, \star, W; 1)\equiv I^{\mathfrak{l}} (0; W, \star, 0^{k}; 1), \quad  (\boldsymbol{\ddagger}) \quad I^{\mathfrak{l}} (0; 0^{k}, -\star, W; 1)\equiv I^{\mathfrak{l}} (0;- W, -\star, 0^{k}; 1).$$
Furthermore, by shuffle regularization formula ($\ref{eq:shufflereg}$), spreading the first $0$ further inside the iterated integrals, the identity  $I^{\mathfrak{l}} (0;\boldsymbol{0}^{k}, \star, W; 1)\equiv  (-1)^{w+1} I^{\mathfrak{l}} (0;\boldsymbol{0}^{k}, \star, \widetilde{W}; 1)$ boils down to the case $k=0$. \\
\texttt{Notations:} As usual $\epsilon_{i}=\textrm{sign} (n_{i})$, $\epsilon_{i}=\eta_{i}\eta_{i+1}$,$\epsilon_{p}= \eta_{p}$, $n_{i}=\epsilon_{i}(a_{i}+1)$.
\begin{itemize}
\item[$(\boldsymbol{\dagger})$] In $(\ref{eq:octagonlin})$, if we look at the coefficient of a specific word in $\lbrace e_{0},e_{1}, e_{-1}\rbrace$ ending and beginning with $e_{-1}$ (as in $\S 4.6$), only two terms contribute, i.e.:
\begin{equation}\label{eq:octagonlinpart1}
 e_{-1}\Phi^{\mathfrak{l}}(e_{\infty}, e_{-1},e_{1})- \Phi^{\mathfrak{l}}(e_{\infty}, e_{-1},e_{1})e_{-1} 
 \end{equation}
The coefficient of $e_{0}^{a_{0}}e_{\eta_{1}} e_{0}^{a_{1}} \cdots e_{\eta_{p}} e_{0}^{a_{p}}$ in $\Phi^{\mathfrak{m}}(e_{\infty}, e_{-1},e_{1})$ is $(-1)^{n+p}\zeta^{\star\star,\mathfrak{m}}_{n_{0}-1} \left( n_{1}, \cdots, n_{p-1}, -n_{p}\right)$.\footnote{The expressions of these associators are more detailed in the proof of Lemma $\ref{lemmlor}$.} Hence, the coefficient in $(\ref{eq:octagonlinpart1})$ (as in $(\ref{eq:octagonlin})$) of the word $e_{-1} e_{0}^{a_{0}} e_{\eta_{1}} \cdots  e_{\eta_{p}} e_{0}^{a_{p}} e_{-1}$ is:
$$ \zeta^{\star\star, \mathfrak{l}}_{\mid n_{0}\mid -1}(n_{1}, \cdots, - n_{p}, 1) -  \zeta^{\star\star, \mathfrak{l}}(n_{0}, n_{1}, \cdots, n_{p-1}, -n_{p})=0, \quad \text{ with} \prod_{i=0}^{p} \epsilon_{i}=1.$$
In terms of iterated integrals, reversing the first one with Antipode $\shuffle$, it is:
$$ I^{\mathfrak{l}} \left(0;-W , \star ;1 \right)\equiv I^{\mathfrak{l}} \left(0; \star, -W ;1 \right), \text{ with } W\mathrel{\mathop:}=0^{n_{0}-1} \eta_{1} 0^{n_{1}-1} \cdots \eta_{p} 0^{n_{p}-1}.$$
Therefore, since $W$ can be any word in $\lbrace 0, \pm \star \rbrace$, by linearity this is also true for any word W in $\lbrace 0, \pm 1 \rbrace$: $ I^{\mathfrak{l}} \left(0;W, \star ;1 \right)\equiv I^{\mathfrak{l}} \left(0; \star, W ;1 \right)$.
\item[$(\boldsymbol{\ddagger})$] Now, let look at the coefficient of a specific word in $\lbrace e_{0},e_{1}, e_{-1}\rbrace$ beginning by $e_{1}$, and ending by $e_{-1}$. Only two terms in the left side of $(\ref{eq:octagonlin})$ contribute, i.e.:
\begin{equation}\label{eq:octagonlinpart2}
 -e_{1}\Phi^{\mathfrak{l}}(e_{1}, e_{\infty},e_{0})- \Phi^{\mathfrak{l}}(e_{\infty}, e_{-1},e_{1})e_{-1} 
 \end{equation}
The coefficient in this expression of the word $e_{1} e_{0}^{a_{0}} e_{\eta_{1}} \cdots  e_{\eta_{p}} e_{0}^{a_{p}} e_{-1}$ is:
$$ \zeta^{\star\star, \mathfrak{l}}_{\mid n_{0}\mid -1}(n_{1}, \cdots, n_{p}, -1) -  \zeta^{\star\star, \mathfrak{l}}(n_{0}, n_{1}, \cdots, n_{p})=0, \quad \text{ with} \prod_{i=0}^{p} \epsilon_{i}=-1.$$
In terms of iterated integrals, reversing the first one with Antipode $\shuffle$, it is:
$$ I^{\mathfrak{l}} \left(0; - W , -\star ;1 \right)\equiv I^{\mathfrak{l}} \left(0; -\star, W ;1 \right).$$
Therefore, since $W$ can be any word in $\lbrace 0, \pm \star \rbrace$, by linearity this is also true for any word W in $\lbrace 0, \pm 1 \rbrace$.
\end{itemize}
\end{proof}

\paragraph{For Euler $\boldsymbol{\star\star}$ sums. }

\begin{coro}
In the coalgebra $\mathcal{L}^{2}$, with $n_{i}\in\mathbb{Z}^{\ast}$, $n\geq 1$:
\begin{equation}\label{eq:antipodestaresss}
\zeta^{\star\star,\mathfrak{l}}_{n-1}(n_{1}, \ldots, n_{p})\equiv (-1)^{w+1}\zeta^{\star\star,\mathfrak{l}}_{n-1}(n_{p}, \ldots, n_{1}).
\end{equation}
Motivic Euler $\star\star$ sums of depth $p$ in $\mathcal{L}$ form a dihedral group of order $p+1$:
$$\textsc{(Shift)   } \quad \zeta^{\star\star,\mathfrak{l}}_{\mid n\mid -1}(n_{1}, \ldots, n_{p})\equiv \zeta^{\star\star,\mathfrak{l}}_{\mid n_{1}\mid -1}(n_{2}, \ldots, n_{p},n) \quad \text{ where } \textrm{sign}(n)\mathrel{\mathop:}= \prod_{i} \textrm{sign}(n_{i}).$$
\end{coro}
\noindent
Identity $(\ref{eq:antipodestaresss})$, respectively $\textsc{Shift}$, corresponds to the action of a reflection resp. of a cycle of order $p$ on motivic Euler $\star\star$ sums of depth $p$ in $\mathcal{L}$.
\begin{proof}
Writing $\zeta^{\star\star,\mathfrak{m}}$ as a sum of Euler sums:
\begin{small}
$$\hspace*{-0.5cm}\zeta^{\star\star,\mathfrak{m}}_{n-1}(n_{1}, \ldots, n_{p})=\sum_{i=1}^{p} \zeta^{\mathfrak{m}}_{n-1+\mid n_{1}\mid+\cdots+ \mid n_{i-1}\mid}(n_{i} \circ \cdots \circ n_{p})=\sum_{r \atop A_{i}} \left( \zeta^{\mathfrak{m}}_{n-1}(A_{1}, \ldots, A_{r})+ \zeta^{\mathfrak{m}}_{n-1+\mid A_{1}\mid}(A_{2}, \ldots, A_{r})\right),$$
\end{small}
where the last sum is over $(A_{i})_{i}$ such that each $A_{i}$ is a non empty \say{sum} of consecutive $(n_{j})'s$, preserving the order; the absolute value being summed whereas the sign of the $n_{i}$ involved are multiplied; moreover, $\mid A_{1}\mid \geq \mid n_{1}\mid $ resp. $\mid A_{r} \mid \geq \mid n_{p}\mid $.\\
Using Theorem $(\ref{hybrid})$ in the coalgebra $\mathcal{L}$, the previous equality turns into:
$$(-1)^{w+1}\sum_{r \atop A_{i}} \left( \zeta^{\mathfrak{l}}_{n-1}(A_{r}, \ldots, A_{1})+ \zeta^{\mathfrak{l}}_{n-1+\mid A_{r}\mid}(A_{r-1}, \ldots, A_{1})\right) \equiv (-1)^{w+1}\zeta^{\star\star,\mathfrak{m}}_{n-1}(n_{p}, \ldots, n_{1}). $$
Then, $\textsc{Shift}$ is obtained as the composition with Antipode $\shuffle$ $(\ref{eq:antipodeshuffle2})$.
\end{proof}

\paragraph{For Euler $\boldsymbol{\sharp\sharp}$ sums.}

\begin{coro}\label{esharprel}
In the coalgebra $\mathcal{L}$, for $n\in\mathbb{N}$, $n_{i}\in\mathbb{Z}^{\ast}$, $\epsilon_{i}\mathrel{\mathop:}=\textrm{sign}(n_{i})$:\footnote{Here, $\mlq - \mrq$ denotes the operation where absolute values are subtracted whereas sign multiplied.}\\
\begin{tabular}{lll}
\textsc{Reverse} & $\zeta^{\sharp\sharp,\mathfrak{l}}_{n}(n_{1}, \ldots, n_{p})+ (-1)^{w}\zeta^{\sharp\sharp,\mathfrak{l}}_{n}(n_{p}, \ldots, n_{1}) \equiv
\left\{
\begin{array}{l}
0   . \\
\zeta^{\sharp,\mathfrak{l}}_{n}(n_{1}, \ldots, n_{p})
  \end{array}\right.$ & $ \begin{array}{l}
\textrm{       if } w+p \textrm{ even } . \\
 \textrm{    if  } w+p \textrm{ odd } \end{array} .$\\
 &&\\
\textsc{Shift} &  $\zeta^{\sharp\sharp,\mathfrak{l}}_{ n -1}(n_{1}, \ldots, n_{p})\equiv \zeta^{\sharp\sharp,\mathfrak{l}}_{\mid n_{1}\mid-1}(n_{2}, \ldots, n_{p},\epsilon_{1}\cdots \epsilon_{p} \cdot n)$ & for  $w+p$ even.\\
&&\\
\textsc{Cut} & $\zeta^{\sharp\sharp,\mathfrak{l}}_{n}(n_{1},\cdots, n_{p}) \equiv \zeta^{\sharp\sharp,\mathfrak{l}}_{n+\mid n_{p}\mid}(n_{1},\cdots, n_{p-1}),$ & for $w+p$ odd.\\
&&\\
\textsc{Minus} & $\zeta^{\sharp\sharp,\mathfrak{l}}_{n-i}(n_{1},\cdots, n_{p}) \equiv  \zeta^{\sharp\sharp,\mathfrak{l}}_{n}\left( n_{1},\cdots, n_{p-1},  n_{p} \mlq - \mrq i)\right)$, & for $\begin{array}{l}
w+p \text{ odd }\\
i \leq \min(n,\mid n_{p}\mid)
\end{array}$.\\
&&\\
\textsc{Sign} & $\zeta^{\sharp\sharp,\mathfrak{l}}_{n}(n_{1},\cdots,  n_{p-1}, n_{p}) \equiv  \zeta^{\sharp\sharp,\mathfrak{l}}_{n}(n_{1},\cdots,  n_{p-1},-n_{p})$,& for $w+p$ odd.
\end{tabular}
$$\quad\Rightarrow \forall W \in \lbrace 0, \pm \sharp\rbrace^{\times} \text{with an odd number of } 0,\quad I^{\mathfrak{l}}(-1; W ; 1) \equiv 0 .$$
\end{coro}
\noindent
\textsc{Remark}:
In the coaction of Euler sums, terms with $\overline{1}$ can appear\footnote{Referring to Lemma $\autoref{lemmt}$, a $\overline{1}$ can appear in terms of the type $T_{\epsilon, -\epsilon}$ for a cut between $\epsilon$ and $-\epsilon$.}, which are clearly not motivic MZV. The left side corresponding to such a term in $D_{2r+1}(\cdot)$ is $I^{ \mathfrak{l}}(1;W;-1)$, W odd weight with $0, \pm \sharp$. It is worth underlying that, for the $\sharp$ family with $\lbrace \overline{even}, odd\rbrace$, these terms disappear by $\textsc{Sign}$, since by constraint on parity, $W$ is always of even depth for such a cut. This $\sharp$ family is then more suitable for an unramified criterion (cf. section 4.).
\begin{proof}
These are consequences of the hybrid relation in Theorem $\ref{hybrid}$.
\begin{itemize}
\item[$\cdot$] \textsc{Reverse:}
Writing $\zeta^{\sharp\sharp,\mathfrak{l}}$ as a sum of Euler sums:
\begin{flushleft}
$\hspace*{-0.5cm}\zeta^{\sharp\sharp,\mathfrak{m}}_{k}(n_{1}, \ldots, n_{p}) +(-1)^{w} \zeta^{\sharp\sharp,\mathfrak{m}}_{k}(n_{p}, \ldots, n_{1})$
\end{flushleft}
\begin{small}
$$\hspace*{-0.5cm}\begin{array}{l}
=\sum_{i=1}^{p} 2^{p-i+1-n_{+}}  \zeta^{\mathfrak{m}}_{k+n_{1}+\cdots+ n_{i-1}}(n_{i} \circ \cdots \circ n_{p})   + (-1)^{w}2^{i-n_{+}} \zeta^{\mathfrak{m}}_{k+n_{p}+\cdots+ n_{i+1}}(n_{i} \circ \cdots \circ n_{1})\\
 \\
=\sum_{r \atop A_{i}}2^{r-1} \left(\right. 2 \zeta^{\mathfrak{m}}_{k}(A_{1}, \ldots, A_{r}) +2 (-1)^{w} \zeta^{\mathfrak{m}}_{k}(A_{r}, \ldots, A_{1}) + \zeta^{\mathfrak{m}}_{k+A_{1}}(A_{2}, \ldots, A_{r}) +(-1)^{w} \zeta^{\mathfrak{m}}_{k+A_{r}}(A_{r-1}, \ldots, A_{1}) \left.  \right) 
\end{array}$$
\end{small}
where the sum is over $(A_{i})$ such that each $A_{i}$ is a non empty \say{sum} of consecutive $(n_{j})'s$, preserving the order; i.e. absolute values of $n_{i}$ are summed whereas signs are multiplied; moreover, $A_{1}$ resp. $A_{r}$ are no less than $n_{1}$ resp. $n_{p}$.\\
By Theorem $\ref{hybrid}$, the previous equality turns into, in $\mathcal{L}$:
$$\sum_{r \atop A_{i}}2^{r-1} \left( \zeta^{\mathfrak{l}}_{k}(A_{1}, \ldots, A_{r})+  (-1)^{w} \zeta^{\mathfrak{l}}_{k}(A_{r}, \ldots, A_{1})\right)$$
$$ \equiv 2^{-1} \left( \zeta^{\sharp,\mathfrak{l}}_{k}(n_{1}, \ldots, n_{p})+  (-1)^{w} \zeta^{\sharp,\mathfrak{l}}_{k}(n_{p}, \ldots, n_{1})\right)\equiv 2^{-1}  \zeta^{\sharp,\mathfrak{l}}_{k}(n_{1}, \ldots, n_{p}) \left( 1+ (-1)^{w+p+1} \right).$$
By the Antipode $\star$ relation applied to $\zeta^{\sharp,\mathfrak{l}}$, it implies the result stated, splitting the cases $w+p$ even and $w+p$ odd.
\item[$\cdot$] \textsc{Shift:} Obtained when combining \textsc{Reverse} and \textsc{Antipode} $\shuffle$, when $w+p$ even.
\item[$\cdot$] \textsc{Cut:} Reverse in the case $w+p$ odd implies:
$$\zeta^{\sharp\sharp,\mathfrak{l}}_{n+\mid n_{1}\mid }(n_{2}, \ldots, n_{p})+ (-1)^{w}\zeta^{\sharp\sharp,\mathfrak{l}}_{n}(n_{p}, \ldots, n_{1})  \equiv 0,$$
Which, reversing the variables, gives the Cut rule.
\item[$\cdot$] \textsc{Minus} follows from \textsc{Cut} since, by \textsc{Cut}, both sides are equal to $\zeta^{\sharp\sharp,\mathfrak{l}}_{n-i+ \mid  n_{p}\mid}(n_{1},\cdots, n_{p-1})$.
\item[$\cdot$] In \textsc{Cut}, the sign of $n_{p}$ does not matter, hence, using \textsc{Cut} in both directions, with different signs leads to \textsc{Sign}:
$$\zeta^{\sharp\sharp,\mathfrak{l}}_{n} (n_{1},\ldots,n_{p})\equiv \zeta^{\sharp\sharp,\mathfrak{l}}_{n+ \mid n_p\mid } (n_{1}, \ldots,n_{p-1})\equiv \zeta^{\sharp\sharp,\mathfrak{l}}_{n} ( n_{1},\ldots,-n_{p}).$$
Translating in terms of iterated integrals, it leads to:
$$ I^{\mathfrak{l}}(0; W ; 1) \equiv  I^{\mathfrak{l}}(0; -W; 1), \quad \text{ for } W \text {any sequence of } 0, \pm \sharp, \text{ with } w+p \text{ odd }, \footnote{The weight $w$ is the length of W, whereas the depth $p$ is the number of $\pm \sharp$. Hence this condition is equivalent for W to have an odd number of $0$.}$$ 
where $-W$ is obtained from $W$ after exchanging $\sharp$ and $-\sharp$. Moreover, $I^{\mathfrak{l}}(0; -W; 1)\equiv I^{\mathfrak{l}}(0; W; -1) \equiv - I^{\mathfrak{l}}(-1; W; 0)$. Hence, we obtain, using the composition rule of iterated integrals modulo product:
$$I^{\mathfrak{l}}(0; W ; 1) + I^{\mathfrak{l}}(-1; W; 0)\equiv I^{\mathfrak{l}}(-1; W ; 1)  \equiv 0.$$
\end{itemize} 
\end{proof}

\section{Some Unramified Euler $\sharp$ sums}

Let's look at the following family with only positive odd and negative even integers:
$$\boldsymbol{\zeta^{\sharp,  \mathfrak{m}} \left( \lbrace \overline{\textbf{even }} ,  \textbf{odd } \rbrace^{\times} \right)}.$$
In the iterated integral, this condition means that we see only the following sequences:
\begin{center}
$\epsilon 0^{2a} \epsilon$, $\quad$ or $\quad\epsilon 0^{2a+1} -\epsilon$, $\quad$  with $\quad\epsilon\in \lbrace \pm\sharp \rbrace$.
\end{center}
It is the family of Euler sums appearing in the conjecture $(\autoref{lzg})$, and is unramified:
\begin{theo}\label{ESsharphonorary}
The motivic Euler sums $\zeta^{\sharp,  \mathfrak{m}} \left( \lbrace \overline{\text{even }},  \text{odd } \rbrace^{\times} \right) $ are motivic geometric$^{+}$ periods of $\mathcal{MT}(\mathbb{Z})$ and therefore $\mathbb{Q}$ linear combinations of motivic multiple zeta values.
\end{theo}
The proof ($\S 4.2$) relies mainly upon the stability under the coaction of this family. This motivic family is even a generating family of motivic MZV:
\begin{theo}\label{ESsharpbasis}
The following family is a basis of $\mathcal{H}^{1}$:
$$\mathcal{B}^{\sharp}\mathrel{\mathop:}= \left\lbrace \zeta^{ \sharp,\mathfrak{m}} (2a_{0}+1,2a_{1}+3,\cdots, 2 a_{p-1}+3, \overline{2a_{p}+2})\text{ , } a_{i}\geq 0\right\rbrace .$$
\end{theo}
\noindent This family is conjecturally (Conjecture $(\autoref{lzg})$) the Hoffman star family $\zeta^{\star} (\boldsymbol{2}^{a_{0}},3,\cdots, 3, \boldsymbol{2}^{a_{p}})$.\\
The proof uses the increasing \textit{depth filtration} $\mathcal{F}^{\mathfrak{D}}$ on $\mathcal{H}^{2}$:
\begin{center}
  $\mathcal{F}_{p}^{\mathfrak{D}} \mathcal{H}^{2}$ is generated by Euler sums of depth smaller than $p$.
 \end{center} 
It is not a grading but the associated graded is defined as the quotient $gr_{p}^{\mathfrak{D}}\mathrel{\mathop:}=\mathcal{F}_{p}^{\mathfrak{D}} \diagup \mathcal{F}_{p-1}^{\mathfrak{D}}$. The vector space $\mathcal{F}_{p}^{\mathfrak{D}}\mathcal{H}$ is stable under the action of $\mathcal{G}$. The linear independence of $\mathcal{B}^{\sharp}$ elements is proved below thanks to a recursion on the depth and on the weight, using the injectivity of a map $\partial$ where $\partial$ came out of the depth and weight-graded part of the coaction $\Delta$, and based on the stability of $\mathcal{B}^{\sharp}$ under the coaction.

\subsection{Depth graded Coaction}

Let also introduce the following depth graded derivations: \footnote{Projecting on the right side, since $gr^{\mathcal{D}}_{1} \mathcal{L}_{2r+1}=\mathbb{Q}\zeta^{\mathfrak{l}}(\pm(2r+1))$ using depth 1 $\ref{eq:depth1}$.}
\begin{itemize}
\item[$\cdot$] $\boldsymbol{D^{\pm 1}_{2r+1,p}}: gr_{p}^{\mathfrak{D}}\mathcal{H}_{n} \rightarrow  gr_{p-1}^{\mathfrak{D}} \mathcal{H}_{n-2r-1}, \quad \textrm{is the composition }  \pi^{\pm 1} \circ D_{2r+1,p}$, where:
$$ \pi^{\pm 1}: gr_{1}^{\mathfrak{D}} \mathcal{L}_{2r+1}\otimes gr_{p-1}^{\mathfrak{D}} \mathcal{H}\rightarrow  gr_{p-1}^{\mathfrak{D}} \mathcal{H}
 \quad  \text{ such that } \left\lbrace \begin{array}{ll}
\pi^{1}\left( \zeta^{\mathfrak{m}}(\epsilon (2r+1)) \otimes X\right)  & =\frac{\epsilon 2^{r}+1-\epsilon}{2^{r}} X \\
 \pi^{-1}\left( \zeta^{\mathfrak{m}}(\epsilon (2r+1)) \otimes X \right)  &=\frac{\epsilon 2^{r}+1-\epsilon}{2-2^{r}} X
\end{array}\right. .$$
\item[$\cdot$] The following map, whose injectivity is fundamental to the Theorem $\ref{ESsharpbasis}$:
\begin{equation}
\label{eq:pderivnp}
\boldsymbol{\partial _{n,p}} \mathrel{\mathop:}=\oplus_{2r+1<n} D^{-1}_{2r+1,p} : gr_{p}^{\mathfrak{D}}\mathcal{H}_{n}  \rightarrow \oplus_{1\leq 2r+1<n } gr_{p-1}^{\mathfrak{D}}\mathcal{H}_{n-2r-1}
\end{equation}
\end{itemize}

\begin{lemm}Their explicit expression is: \footnotemark[2]\footnotetext[2]{To be accurate, the term $i=0$ in the first sum has to be understood as:
$$ \frac{2^{2r+1}}{1-2^{2r}}\binom{2r}{2a_{1}+2} \zeta^{\sharp,\mathfrak{m}} (2\alpha+3,  2 a_{2}+3,\cdots, \overline{2a_{p}+2}) . $$
Meanwhile the terms $i=1$, resp. $i=p$ in the second sum have to be understood as:
$$ \frac{2^{2r+1}}{1-2^{2r}}\binom{2r}{2a_{0}+2} \zeta^{\sharp,\mathfrak{m}} (2\alpha+3,  2 a_{2}+3,\cdots, \overline{2a_{p}+2}) \quad \text{ resp. } \quad \frac{2^{2r+1}}{1-2^{2r}}\binom{2r}{2a_{p-1}+2} \zeta^{\sharp,\mathfrak{m}} (\cdots, 2 a_{p-2}+3, \overline{2\alpha+2}).$$}
\begin{multline} \label{eq:dgrderiv} 
D^{-1}_{2r+1,p} \left(  \zeta^{\sharp,\mathfrak{m}} (2a_{0}+1,2a_{1}+3,\cdots, 2 a_{p-1}+3, \overline{2a_{p}+2}) \right) = \\
\delta_{r=a_{0}} \frac{2^{2r+1}}{1-2^{2r}}\binom{2r}{2r+2} \zeta^{\sharp,\mathfrak{m}} (2 a_{1}+3,\cdots, \overline{2a_{p}+2})\\
+ \sum_{0 \leq i \leq p-2, \quad  \alpha \leq a_{i}\atop r=a_{i+1}+a_{i}+1-\alpha} \frac{2^{2r+1}}{1-2^{2r}}\binom{2r}{2a_{i+1}+2} \zeta^{\sharp,\mathfrak{m}} (\cdots, 2 a_{i-1}+3,2\alpha+3,  2 a_{i+2}+3,\cdots, \overline{2a_{p}+2})\\
+ \sum_{1 \leq i \leq p-1, \quad \alpha \leq a_{i} \atop r=a_{i-1}+a_{i}+1-\alpha} \frac{2^{2r+1}}{1-2^{2r}}\binom{2r}{2a_{i-1}+2} \zeta^{\sharp,\mathfrak{m}} (\cdots, 2 a_{i-2}+3,2\alpha+3,  2 a_{i+1}+3,\cdots, \overline{2a_{p}+2})\\
+ \textsc{(Deconcatenation)} \sum_{\alpha \leq a_{p} \atop r=a_{p-1}+a_{p}+1-\alpha} 2 \binom{2r}{2a_{p}+1}\zeta^{\sharp,\mathfrak{m}} (\cdots, 2 a_{p-1}+3,\overline{2\alpha+2}).
\end{multline}
\end{lemm}
\begin{proof}
Looking at the Appendix $A$ expression for $D_{2r+1}$ and keeping only the cuts of depth one (removing exactly one non zero element), in the depth graded:
$$\begin{array}{llll}
=& +\sum_{i, \alpha \leq a_{i}\atop r=a_{i+1}+a_{i}+1-\alpha} 2  & \zeta^{\mathfrak{l}} _{2a_{i}-2\alpha}(2a_{i+1}+3)  & \otimes \zeta^{\sharp,\mathfrak{m}} (\cdots, 2 a_{i-1}+3,2\alpha+3,  2 a_{i+2}+3,\cdots, \overline{2a_{p}+2})\\
&+\sum_{i, \alpha \leq a_{i} \atop r=a_{i-1}+a_{i}+1-\alpha} 2 &\zeta^{\mathfrak{l}} _{2a_{i}-2\alpha}(2a_{i-1}+3)  &\otimes \zeta^{\sharp,\mathfrak{m}} (\cdots, 2 a_{i-2}+3,2\alpha+3,  2 a_{i+1}+3,\cdots, \overline{2a_{p}+2})\\
&+ \sum_{\alpha \leq a_{p} \atop r=a_{p-1}+a_{p}+1-\alpha} 2 &\zeta^{\mathfrak{l}} _{2a_{p-1}-2\alpha+1}(\overline{2a_{p}+2}) & \otimes \zeta^{\sharp,\mathfrak{m}} (\cdots, 2 a_{p-1}+3,\overline{2\alpha+2})
\end{array}
$$
To lighten the result, some cases at the borders ($i=0$, or $i=p$) similar to the other terms (apart from index problems), have been included in the sum: these are clarified in the previous footnote\footnotemark[2]. In particular, with notations of the Lemma $\autoref{lemmt}$, $T_{0,0}$ terms can be neglected as they decrease the depth by at least $2$; same for the $T_{0,\epsilon}$ and  $T_{\epsilon,0}$ for cuts between $\epsilon$ and $\pm \epsilon$. It remains to check the coefficient of $\zeta^{\mathfrak{l}}(\overline{2r+1})$ for each term in the left side, using the known identities:
$$  \zeta^{\mathfrak{l}}(2r+1)= \frac{-2^{2r}}{2^{2r}-1} \zeta^{\mathfrak{l}}(\overline{2r+1})\quad  \text{  and  } \quad \zeta^{\mathfrak{l}}_{2r+1-a}(a)=(-1)^{a+1}\binom{2r}{a-1} \zeta^{\mathfrak{l}}(2r+1).$$
\end{proof}

\subsection{Proofs of Theorem $\autoref{ESsharphonorary}$ and $\autoref{ESsharpbasis}$}

\begin{proof}[\textbf{Proof of Theorem $\boldsymbol{\autoref{ESsharphonorary}}$}]
By Corollary $\ref{criterehonoraire}$, we can prove it in two steps:
\begin{itemize}
\item[$\cdot$] First, it is rather obvious that $D_{1}(\cdot)=0$ on this family since there is no sequence of the type $\lbrace 0, \epsilon, -\epsilon \rbrace$ or $\lbrace \epsilon, -\epsilon, 0 \rbrace$ in these iterated integrals.
\item[$\cdot$] Secondly, we have to prove that $D_{2r+1}(\cdot)$, for $r> 0$, are unramified for this family. This follows straight, by recursion on weight, from this stability statement:
\begin{center}
The family $\zeta^{\sharp,  \mathfrak{m}} \left( \lbrace \overline{\text{even }},  + \text{odd } \rbrace^{\times} \right) $ is stable under $D_{2r+1}$.
\end{center}
This is proved in Lemma $A.3$, using the relations of $\S.3$ in order to get rid of the \textit{unstable cuts}, i.e. cuts where a sequence of type $\epsilon, 0^{2a+1}, \epsilon$ or $\epsilon, 0^{2a}, -\epsilon$ appears (which corresponds to a $\text{even}$ or $\overline{\text{odd}}$ in the Euler $\sharp$ sum).
\end{itemize}
One features used in the lemma about this family is: for a subsequence of odd length from the iterated integral, because of these patterns of $\epsilon, \boldsymbol{0}^{2a}, \epsilon$, or $\epsilon, \boldsymbol{0}^{2a+1}, -\epsilon$, we can relate the depth $p$, the weight $w$ and $s$ the number of sign changes among the $\pm\sharp$: $w\equiv p-s  \pmod{2}$. Hence, if we have a cut $\epsilon_{0},\cdots \epsilon_{p+1}$ of odd weight, then:
\begin{center}
\hspace*{-0.5cm}\textsc{Either:} Depth $p$ is odd, $s$ even, $\epsilon_{0}=\epsilon_{p+1}$,  $\quad$  \textsc{Or:} Depth $p$ is even, $s$ odd, $\epsilon_{0}=-\epsilon_{p+1}$.
\end{center}
\end{proof}

\begin{proof}[\textbf{Proof of Theorem $\boldsymbol{4.2}$}]
By a cardinality argument, it is sufficient to prove the linear independence of the family, which is based on the injectivity of $\partial_{<n,p}$. Let us define: 
\begin{center}
$\mathcal{H}^{odd\sharp}$:  $\mathbb{Q}$-vector space generated by $\zeta^{ \sharp,\mathfrak{m}} (2a_{0}+1,2a_{1}+3,\cdots, 2 a_{p-1}+3, \overline{2a_{p}+2})$.
\end{center}
The first point, thanks to Lemma $A.4$, is that $\mathcal{H}^{odd\sharp}$ is stable under these derivations:
$$D_{2r+1} (\mathcal{H}_{n}^{odd\sharp}) \subset  \mathcal{L}_{2r+1} \otimes \mathcal{H}_{n-2r-1}^{odd\sharp},$$
Now, let consider the restriction on $\mathcal{H}^{odd\sharp}$ of $\partial_{<n,p}$ and prove:
$$\partial_{<n,p}: gr^{\mathfrak{D}}_{p} \mathcal{H}_{n}^{odd\sharp} \rightarrow \oplus_{2r+1<n}  gr^{\mathfrak{D}}_{p-1}\mathcal{H}_{n-2r-1}^{odd\sharp} \text{  is bijective. }$$
The formula $\eqref{eq:dgrderiv}$ gives the explicit expression of this map. Let us prove more precisely:
\begin{center}
\texttt{Claim 1}:  $M^{\mathfrak{D}}_{n,p}$ the matrix of $\partial_{<n,p}$ on $\left\lbrace \zeta^{ \sharp ,\mathfrak{m}} (2a_{0}+1,2a_{1}+3,\cdots, 2 a_{p-1}+3, \overline{2a_{p}+2})\right\rbrace$ $\quad\quad$ in terms of $\left\lbrace \zeta^{ \sharp ,\mathfrak{m}} (2b_{0}+1,2b_{1}+3,\cdots, 2 b_{p-2}+3, \overline{2b_{p-1}+2})\right\rbrace $ is invertible.
\end{center}
\texttt{Nota Bene}: The matrix $M^{\mathfrak{D}}_{n,p}$ is well (uniquely) defined provided that the $\zeta^{ \sharp ,\mathfrak{m}}$ of the second line are linearly independent. So first, we have to consider the formal matrix associated $\mathbb{M}^{\mathfrak{D}}_{n,p}$ defined explicitly (combinatorially) by the formula for the derivations given and prove $\mathbb{M}^{\mathfrak{D}}_{n,p}$ is invertible. Afterwards, we could state that $M^{\mathfrak{D}}_{n,p}$ is well defined and invertible too since equal to $\mathbb{M}^{\mathfrak{D}}_{n,p}$.
\begin{proof}[\texttt{Proof of Claim 1}]
The invertibility comes from the fact that the (strictly) smallest terms $2$-adically in $\eqref{eq:dgrderiv}$ are the deconcatenation ones, which is an injective operation. More precisely, let $\widetilde{M}^{\mathfrak{D}}_{n,p}$ be the matrix $\mathbb{M}_{n,p}$ where we have multiplied each line corresponding to $D_{2r+1}$ by ($2^{-2r}$). Then, order elements on both sides by lexicographical order on  ($a_{p}, \ldots, a_{0}$), resp. ($r,b_{p-1}, \ldots, b_{0}$), such that the diagonal corresponds to $r=a_{p}+1$ and $b_{i}=a_{i}$ for $i<p$. The $2$ -adic valuation of all the terms in $(\ref{eq:dgrderiv})$ (once divided by $2^{2r}$) is at least $1$, except for the deconcatenation terms since:
$$v_{2}\left( 2^{-2r+1} \binom{2r}{2a_{p}+1} \right)  \leq 0 \Longleftrightarrow  v_{2}\left( \binom{2r}{2a_{p}+1} \right)  \leq 2r-1.$$
Then, modulo $2$, only the deconcatenation terms remain, so the matrix $\widetilde{M}^{\mathfrak{D}}_{n,p}$ is triangular with $1$ on the diagonal. This implies that $\det (\widetilde{M}^{\mathfrak{D}}_{n,p})\equiv 1   \pmod{2}$, and in particular is non zero: the matrix $\widetilde{M}^{\mathfrak{D}}_{n,p}$ is invertible, and so does $\mathbb{M}^{\mathfrak{D}}_{n,p}$.
\end{proof}
This allows us to complete the proof since it implies:
\begin{center}
\texttt{Claim 2}:  The elements of $\mathcal{B}^{\sharp}$ are linearly independent.
\end{center}
\begin{proof}[\texttt{Proof of Claim 2}]
First, let prove the linear independence of this family of the same depth and weight, by recursion on $p$. Depth $0$ is obvious since $\zeta^{\mathfrak{m}}(\overline{2n})$ is a rational multiple of $\pi^{2n}$.\\
Assuming by recursion on the depth that the elements of weight $n$ and depth $p-1$ are linearly independent, since $M^{\mathfrak{D}}_{n,p}$ is invertible, this means both that the  $\zeta^{ \sharp,\mathfrak{m}} (2a_{0}+1,2a_{1}+3,\cdots, 2 a_{p-1}+3, \overline{2a_{p}+2})$ of weight $n$ are linearly independent and that $\partial_{<n,p}$ is bijective, as announced before.\\
The last step is just to realize that the bijectivity of $\partial_{<n,l}$ also implies that elements of different depths are also linearly independent. The proof could be done by contradiction: by applying $\partial_{<n,p}$ on a linear combination where $p$ is the maximal depth appearing, we arrive at an equality between same level elements.
\end{proof}
\end{proof}

\section{Hoffman $\star$ basis}

\begin{theo}\label{Hoffstar}
If the analytic conjecture  ($\ref{conjcoeff}$) holds, then the motivic \textit{Hoffman} $\star$ family $\lbrace \zeta^{\star,\mathfrak{m}} (\lbrace 2,3 \rbrace^{\times})\rbrace$ is a basis of $\mathcal{H}^{1}$, the space of MMZV.
\end{theo}
\noindent For that purpose, we define an increasing filtration $\mathcal{F}^{L}_{\bullet}$ on $\mathcal{H}^{2,3}$, called \textbf{level}, such that:
\begin{equation}\label{eq:levelf}
 \mathcal{F}^{L}_{l}\mathcal{H}^{2,3} \text{ is spanned by } \zeta^{\star,\mathfrak{m}} (\boldsymbol{2}^{a_{0}},3,\cdots,3, \boldsymbol{2}^{a_{p}}) \text{, with less than 'l' } 3. 
\end{equation}
It corresponds to the motivic depth for this family, as we see through the proof below.

\subsection{Level graded coaction}

Let use the following form for a MMZV$^{\star}$, gathering the $2$\footnote{This writing is suitable for the Galois action (and coaction), since by the antipode relations ($\S 3.2$), many of the cuts from a $2$ to a $2$ get simplified (cf. Appendix $A$).}:
$$\zeta^{\star, \mathfrak{m}} (\boldsymbol{2}^{a_{0}},c_{1},\cdots,c_{p}, \boldsymbol{2}^{a_{p}}), \quad c_{i}\in\mathbb{N}^{\ast}, c_{i}\neq 2.$$
The expression obtained for the derivations by Lemma $A.2$:
\begin{flushleft}
\hspace*{-0.7cm}$D_{2r+1}   \zeta^{\star, \mathfrak{m}} (\boldsymbol{2}^{a_{0}},3,\cdots,3, \boldsymbol{2}^{a_{p}})$
\end{flushleft}
\begin{multline} \label{eq:dr3}
\hspace*{-1.3cm}= \delta_{2r+1}\sum_{i<j} \left[  
\begin{array}{lll}
 + \quad \zeta^{\star\star, \mathfrak{l}}_{1} (\boldsymbol{2}^{a_{i+1}},3,\cdots,3, \boldsymbol{2}^{\leq a_{j}}) & \otimes & \zeta^{\star, \mathfrak{m}} (\cdots,3, \boldsymbol{2}^{1+a_{i}+ \leq a_{j}},3, \cdots)\\
 - \quad \zeta^{\star\star, \mathfrak{l}}_{1} (\boldsymbol{2}^{\leq a_{i}},3,\cdots,3, \boldsymbol{2}^{ a_{j-1}}) & \otimes & \zeta^{\star, \mathfrak{m}} (\cdots,3, \boldsymbol{2}^{1+a_{j}+ \leq a_{i}},3, \cdots)\\
 + \left(  \zeta^{\star\star, \mathfrak{l}}_{2} (\boldsymbol{2}^{a_{i+1}},3,\cdots, \boldsymbol{2}^{a_{j}},3) +  \zeta^{\star\star, \mathfrak{l}}_{1} (\boldsymbol{2}^{<a_{i}},3,\cdots, \boldsymbol{2}^{a_{j}},3) \right) & \otimes&  \zeta^{\star, \mathfrak{m}} (\cdots,3, \boldsymbol{2}^{<a_{i}},3,\boldsymbol{2}^{a_{j+1}},3, \cdots)\\
  -  \left(\zeta^{\star\star, \mathfrak{l}}_{2} (\boldsymbol{2}^{a_{j+1}},3,\cdots,3) + \zeta^{\star\star, \mathfrak{l}}_{1}(\boldsymbol{2}^{<a_{j}},3,\cdots,3) \right)& \otimes &  \zeta^{\star, \mathfrak{m}} (\cdots,3, \boldsymbol{2}^{a_{i-1}},3,\boldsymbol{2}^{< a_{j}},3, \cdots) \\
\end{array}  \right] \\
\quad \quad \begin{array}{lll}
\quad \quad+ \quad\delta_{2r+1} \quad \left( \zeta^{\star, \mathfrak{l}} (\boldsymbol{2}^{a_{0}},3,\cdots,3, \boldsymbol{2}^{\leq a_{i}})- \zeta^{\star\star, \mathfrak{l}} (\boldsymbol{2}^{\leq a_{i}},3,\cdots,3, \boldsymbol{2}^{a_{0}}) \right) & \otimes & \zeta^{\star, \mathfrak{m}} (\boldsymbol{2}^{\leq a_{i}},3, \cdots)\\
 \quad\quad +\quad \delta_{2r+1} \quad\zeta^{\star\star, \mathfrak{l}} (\boldsymbol{2}^{\leq a_{j}},3,\cdots,3, \boldsymbol{2}^{ a_{p}}) & \otimes & \zeta^{\star, \mathfrak{m}} (\cdots,3, \boldsymbol{2}^{\leq a_{j}}).
\end{array}
\end{multline}
Notably, the coaction on the Hoffman $\star$ elements is stable. Moreover, the level filtration is stable under the action of $\mathcal{G}$ since each cut (of odd length) removes at least one $3$:
\begin{equation} \label{eq:levelfiltstrable}
D_{2r+1}(\mathcal{F}^{L}_{l}\mathcal{H}^{2,3}) \subset  \mathcal{L}_{2r+1} \otimes \mathcal{F}^{L}_{l-1}\mathcal{H}_{n-2r-1}^{2,3} .
\end{equation}
Then, let consider the level graded derivation, which amounts to restrict to the cuts which remove exactly one $3$ in the right side:
\begin{equation}
gr^{L}_{l} D_{2r+1}: gr^{L}_{l}\mathcal{H}_{n}^{2,3} \rightarrow  \mathcal{L}_{2r+1} \otimes gr^{L}_{l-1}\mathcal{H}_{n-2r-1}^{2,3}.
\end{equation}
\begin{flushleft}
\hspace*{-0.5cm}$gr^{L}_{l} D_{2r+1}   \zeta^{\star, \mathfrak{m}} (\boldsymbol{2}^{a_{0}},3,\cdots,3, \boldsymbol{2}^{a_{p}}) =$
\end{flushleft}
\begin{multline}\label{eq:gdr3}
\hspace*{-1.5cm}\begin{array}{lll}
\quad - \delta_{a_{0} < r \leq a_{0}+a_{1}+2} \quad \zeta^{\star\star, \mathfrak{l}}_{2} (\boldsymbol{2}^{a_{0}}, 3, \boldsymbol{2}^{r-a_{0}-2})  &\otimes & \zeta^{\star, \mathfrak{m}} (\boldsymbol{2}^{ a_{0}+a_{1}+1-r},3, \cdots)
\end{array}\\
\hspace*{-1.3cm}\sum_{i<j} \left[ \begin{array}{l}
\delta_{r\leq a_{i}} \quad \zeta^{\star\star, \mathfrak{l}}_{1} (\boldsymbol{2}^{r}) \quad \quad \quad \quad \quad  \otimes  \left(  \zeta^{\star, \mathfrak{m}} (\cdots,3, \boldsymbol{2}^{a_{i-1}+ a_{i}-r+1},3, \cdots)  - \zeta^{\star, \mathfrak{m}} (\cdots,3, \boldsymbol{2}^{a_{i+1}+  a_{i}-r+1},3, \cdots) \right) \\
 + \left( \delta_{r=a_{i}+2} \zeta^{\star\star, \mathfrak{l}}_{2} (\boldsymbol{2}^{a_{i}},3) +  \delta_{r< a_{i}+a_{i-1}+3} \zeta^{\star\star, \mathfrak{l}}_{1} (\boldsymbol{2}^{r-a_{i}-3}, 3, \boldsymbol{2}^{a_{i}},3) \right)  \otimes  \zeta^{\star, \mathfrak{m}} (\cdots,3, \boldsymbol{2}^{a_{i}+a_{i-1}-r+1},3,\boldsymbol{2}^{a_{i+1}},3, \cdots)\\
  - \left( \delta_{r=a_{i}+2} \zeta^{\star\star, \mathfrak{l}}_{2} (\boldsymbol{2}^{a_{i}},3) + \delta_{r< a_{i}+a_{i+1}+3} \zeta^{\star\star, \mathfrak{l}}_{1}(\boldsymbol{2}^{r-a_{i}-3},3, \boldsymbol{2}^{a_{i}}, 3) \right) \otimes  \zeta^{\star, \mathfrak{m}} (\cdots,3, \boldsymbol{2}^{a_{i-1}},3,\boldsymbol{2}^{a_{i}+a_{i+1}-r+1},3, \cdots) 
\end{array} \right] \\
\hspace*{-2cm} \textsc{(D)}  \begin{array}{lll}
 +\delta_{a_{p}+1 \leq r \leq a_{p}+a_{p-1}+1}  \quad  \zeta^{\star\star, \mathfrak{l}} (\boldsymbol{2}^{r- a_{p}-1},3, \boldsymbol{2}^{ a_{p}}) &\otimes & \zeta^{\star, \mathfrak{m}} (\cdots,3, \boldsymbol{2}^{a_{p}+ a_{p-1}-r+1}). 
 \end{array}
 \end{multline}
By the antipode $\shuffle$ relation (cf. $\ref{eq:antipodeshuffle2}$):
$$\zeta^{\star\star, \mathfrak{l}}_{1} (\boldsymbol{2}^{a},3, \boldsymbol{2}^{b},3)= \zeta^{\star\star, \mathfrak{l}}_{2} (\boldsymbol{2}^{b},3, \boldsymbol{2}^{a+1})=\zeta^{\star\star, \mathfrak{l}}(\boldsymbol{2}^{b+1},3, \boldsymbol{2}^{a+1})- \zeta^{\star, \mathfrak{l}}(\boldsymbol{2}^{b+1},3, \boldsymbol{2}^{a+1}).$$
Since by Lemma $\autoref{lemmcoeff}$ all the terms appearing in the left side of $gr^{L}_{l} D_{2r+1}$ are product of simple MZV, in the coalgebra $\mathcal{L}$, it gives simply a rational multiple of $\zeta^{\mathfrak{l}}(2r+1)$:
$$gr^{L}_{l} D_{2r+1} (gr^{L}_{l}\mathcal{H}_{n}^{2,3}) \subset \mathbb{Q}\zeta^{\mathfrak{l}}(2r+1)\otimes gr^{L}_{l-1}\mathcal{H}_{n-2r-1}^{2,3}.$$
Sending $\zeta^{\mathfrak{l}}(2r+1)$ to $1$ with the projection $\pi:\mathbb{Q} \zeta^{\mathfrak{l}}(2r+1)\rightarrow\mathbb{Q}$, we can then consider:
\begin{lemm} The maps:
\begin{description}
\item[$\boldsymbol{\cdot\quad \partial^{L}_{r,l}}$] $ : gr^{L}_{l}\mathcal{H}_{n}^{2,3}\rightarrow gr^{L}_{l-1}\mathcal{H}_{n-2r-1}^{2,3},  \quad \text{ defined as the composition }$
$$\partial^{L}_{r,l}\mathrel{\mathop:}=gr_{l}^{L}\partial_{2r+1}\mathrel{\mathop:}=m\circ(\pi\otimes id)(gr^{L}_{l} D_{r}): \quad gr^{L}_{l}\mathcal{H}_{n}^{2,3} \rightarrow \mathbb{Q}\otimes_{\mathbb{Q}} gr^{L}_{l-1}\mathcal{H}_{n-2r-1}^{2,3}  \rightarrow gr^{L}_{l-1}\mathcal{H}_{n-2r-1}^{2,3} .$$
\item[$\boldsymbol{\cdot\quad \partial^{L}_{<n,l}}$] $\mathrel{\mathop:}=\oplus_{2r+1<n}\partial^{L}_{r,l} .$
\end{description}
Its explicit expression, where coefficients $A_{\bullet}, B_{\bullet}, C_{\bullet}$ are those in $\autoref{lemmcoeff}$:
\begin{flushleft}
$\partial^{L}_{r,l} (\zeta^{\star, \mathfrak{m}} (\boldsymbol{2}^{a_{0}},3,\cdots,3, \boldsymbol{2}^{a_{p}}))=$
\end{flushleft}
$$\begin{array}{l}
 \quad  - \delta_{a_{0} < r \leq a_{0}+a_{1}+2} \widetilde{B}^{a_{0}+1,r-a_{0}-2}  \zeta^{\star, \mathfrak{m}} (\boldsymbol{2}^{ a_{0}+a_{1}+1-r},3, \cdots)  \\
   \\
+ \sum_{i<j} \left[ \begin{array}{l}
 \delta_{r\leq a_{i}}C_{r}  \left(  \zeta^{\star, \mathfrak{m}} (\cdots,3, \boldsymbol{2}^{a_{i-1}+ a_{i}-r+1},3, \cdots) - \zeta^{\star, \mathfrak{m}} (\cdots,3, \boldsymbol{2}^{a_{i+1}+  a_{i}-r+1},3, \cdots) \right) \\
   \\
+\delta_{a_{i}+2\leq r \leq a_{i}+a_{i-1}+2} \widetilde{B}^{a_{i}+1,r-a_{i}-2}  \zeta^{\star, \mathfrak{m}} (\cdots,3, \boldsymbol{2}^{a_{i}+a_{i-1}-r+1},3,\boldsymbol{2}^{a_{i+1}},3, \cdots) \\
   \\
-  \delta_{a_{i}+2 \leq r\leq a_{i}+a_{i+1}+2} \widetilde{B}^{a_{i}+1,r-a_{i}-2} \zeta^{\star, \mathfrak{m}} (\cdots,3, \boldsymbol{2}^{a_{i-1}},3,\boldsymbol{2}^{a_{i}+a_{i+1}-r+1},3, \cdots) \\
\end{array} \right]  \\
 \\
  \textsc{(D)} + \delta_{a_{p}+1 \leq r \leq a_{p}+a_{p-1}+1} B^{r-a_{p}-1,a_{p}} \zeta^{\star, \mathfrak{m}} (\cdots,3, \boldsymbol{2}^{a_{p}+ a_{p-1}-r+1}) , \\
      \\
      \quad \quad \quad   \text{ with }  \widetilde{B}^{a,b}\mathrel{\mathop:}=B^{a,b}C_{a+b+1}-A^{a,b}.
\end{array}$$ 
\end{lemm}
\begin{proof}
Using Lemma $\autoref{lemmcoeff}$ for the left side of $gr^{L}_{p} D_{2r+1}$, and keeping just the coefficients of $\zeta^{2r+1}$, we obtain easily this formula. In particular:
\begin{flushleft}
$\zeta^{\star\star, \mathfrak{l}}_{2} (\boldsymbol{2}^{a},3, \boldsymbol{2}^{b})=\zeta^{\star\star, \mathfrak{l}}(\boldsymbol{2}^{a+1},3, \boldsymbol{2}^{b})- \zeta^{\star, \mathfrak{l}}(\boldsymbol{2}^{a+1},3, \boldsymbol{2}^{b}) = \widetilde{B}^{a+1,b} \zeta^{\mathfrak{l}}(\overline{2a+2b+5}).$\\
$\zeta^{\star\star, \mathfrak{l}}_{1} (\boldsymbol{2}^{a},3, \boldsymbol{2}^{b},3)= \zeta^{\star\star, \mathfrak{l}}_{2} (\boldsymbol{2}^{b},3, \boldsymbol{2}^{a+1})= \widetilde{B}^{b+1,a+1}\zeta^{\mathfrak{l}}(\overline{2a+2b+7}).$
\end{flushleft}
\end{proof}

\subsection{Proof of Theorem $5.1$}
Since the cardinal of the Hoffman $\star$ family in weight $n$ is equal to the dimension of $\mathcal{H}_{n}^{1}$, \footnote{Obviously same recursive relation: $d_{n}=d_{n-2}+d_{n-3}$} it remains to prove that they are linearly independent:
\begin{center}
\texttt{Claim 1}: The Hoffman $\star$ elements are linearly independent.
\end{center}
It fundamentally use the injectivity of the map defined above, $\partial^{L}_{<n,l}$, via a recursion on the level. Indeed, let first prove the following statement:
\begin{equation} \label{eq:bijective} \texttt{Claim 2}: \quad \partial^{L}_{<n,l}: gr^{L}_{l}\mathcal{H}_{n}^{2,3}\rightarrow \oplus_{2r+1<n} gr^{L}_{l-1}\mathcal{H}_{n-2r-1}^{2,3} \text{  is bijective}.
\end{equation}
Using Conjecture $\autoref{conjcoeff}$ (assumed here), regarding the $2$-adic valuations, with $r=a+b+1$:\footnote{The last inequality comes from the fact that $v_{2} (\binom{2r}{2b+1} )<2r $.}
\begin{equation}\label{eq:valuations}
\hspace*{-0.7cm}\left\lbrace  \begin{array}{ll}
 C_{r}=\frac{2^{2r+1}}{2r+1}  &\Rightarrow  v_{2}(C_{r})=2r+1 .\\
 \widetilde{B}^{a,b}\mathrel{\mathop:}= B^{a,b}C_{r}-A^{a,b}=2^{2r+1}\left(  \frac{1}{2r+1}-\frac{\binom{2r}{2a}}{2^{2r}-1} \right) &\Rightarrow  v_{2}(\widetilde{B}^{a,b}) \geq 2r+1.\\
B^{a,b}C_{r}=C_{r}-2\binom{2r}{2b+1} &\Rightarrow   v_{2}(B^{0,r-1}C_{r})= 2+ v_{2}(r) \leq v_{2}(B^{a,b}C_{r}) < 2r+1 .
\end{array} \right. 
\end{equation}
The deconcatenation terms in $\partial^{L}_{<n,l}$, which correspond to the terms with $B^{a,b}C_{r}$ are then the smallest 2-adically, which is crucial for the injectivity.\\
\\
Now, define a matrix $M_{n,l}$ as the matrix of $\partial^{L}_{<n,l}$ on $\zeta^{\star, \mathfrak{m}} (\boldsymbol{2}^{a_{0}},3,\cdots,3, \boldsymbol{2}^{a_{l}})$ in terms of $\zeta^{\star, \mathfrak{m}} (\boldsymbol{2}^{b_{0}},3,\cdots,3, \boldsymbol{2}^{b_{l-1}})$; even if we still do not know that these families are linearly independent, we order elements on both sides by lexicographical order on ($a_{l}, \ldots, a_{0}$), resp. ($r,b_{l-1}, \ldots, b_{0}$), such that the diagonal corresponds to $r=a_{l}$ and $b_{i}=a_{i}$ for $i<l$ and claim:
\begin{center}
\texttt{Claim 3}: The matrix $M_{n,l}$ of $\partial^{L}_{<n,l}$ on the Hoffman $\star$ elements is invertible 
\end{center}
\begin{proof}[\texttt{Proof of Claim 3}]
Indeed, let $\widetilde{M}_{n,l}$ be the matrix $M_{n,l}$ where we have multiplied each line corresponding to $D_{2r+1}$ by ($2^{-v_{2}(r)-2}$). Then modulo $2$, because of the previous computations on the $2$-adic valuations of the coefficients, only the deconcatenations terms remain. Hence, with the previous order, the matrix is, modulo $2$, triangular with $1$ on the diagonal; the diagonal being the case where $B^{0,r-1}C_{r}$ appears. This implies that $\det (\widetilde{M}_{n,l})\equiv 1   \pmod{2}$, and in particular is non zero. Consequently, the matrix $\widetilde{M}_{n,l}$ is invertible and so does $M_{n,l}$.
\end{proof}
Obviously, $\texttt{Claim 3} \Rightarrow \texttt{Claim 2} $, but it will also enables us to complete the proof:

\begin{proof}[\texttt{Proof of Claim 1}] Let first prove it for elements of a same level and weight, by recursion on level. Level $0$ is obvious: $\zeta^{\star,\mathfrak{m}}(2)^{n}$ is a rational multiple of $(\pi^{\mathfrak{m}})^{2n}$. Assuming that the Hoffman $\star$ elements of weight $\leq n$ and level $l-1$ are linearly independent, since $M_{n,l}$ is invertible, this implies that the Hoffman $\star$ elements of weight $n$ and level $l$ are linearly independent. The last step is to realize that the bijectivity of $\partial^{L}_{<n,l}$ also implies that Hoffman $\star$ elements of different levels are linearly independent. Indeed, proof can be done by contradiction: applying  $\partial^{L}_{<n,l}$ to a linear combination of Hoffman $\star$ elements, $l$ being the maximal number of $3$, we arrive at an equality between same level elements, and at a contradiction.
\end{proof}

\subsection{Analytic conjecture}

Here are the equalities needed for Theorem $5.1$, known up to some rational coefficients:
\begin{lemm} \label{lemmcoeff} With  $w$, $d$ resp.  $ht$ denoting the weight, the depth, resp. the height:
\begin{itemize}
\item[$(o)$] $\begin{array}{llll}
\zeta^{\mathfrak{m}}(\overline{r}) & = & (2^{1-r}-1) &\zeta^{\mathfrak{m}}(r).\\
\zeta^{\mathfrak{m}}(2n) & = & \frac{\mid B_{n}\mid 2^{3n-1}3^{n}}{(2n)!} &\zeta^{\mathfrak{m}}(2)^{n}.
\end{array}$
\item[$(i)$] $\zeta^{\star,\mathfrak{m}}(\boldsymbol{2}^{n})= -2 \zeta^{\mathfrak{m}}(\overline{2n}) =\frac{(2^{2n}-2)6^{n}}{(2n)!}\vert B_{2n}\vert\zeta^{\mathfrak{m}}(2)^{n}.$
\item[$(ii)$] $\zeta^{\star,\mathfrak{m}}_{1}(\boldsymbol{2}^{n})= -2 \sum_{r=1}^{n} \zeta^{\mathfrak{m}}(2r+1)\zeta^{\star,\mathfrak{m}}(\boldsymbol{2}^{n-r}).$
\item[$(iii)$] 
\begin{align}
\zeta^{\star\star,\mathfrak{m}}(\boldsymbol{2}^{n}) & = \sum_{d \leq n} \sum_{w(\textbf{m})=2n \atop ht(\textbf{m})=d(\textbf{m})=d} 2^{2n-2d}\zeta^{\mathfrak{m}}(\textbf{m})  \\
& =\sum_{2n=\sum s_{k}(2i_{k}+1)+2S \atop i_{k}\neq i_{j}}  \left( \prod_{k=1}^{p} \frac{C_{i_{k}}^{s_{k}}} {s_{k}!} \zeta^{\mathfrak{m}}(\overline{2i_{k}+1})^{s_{k}} \right) D_{S} \zeta^{\mathfrak{m}}(2)^{S}. \nonumber\\
\zeta^{\star\star,\mathfrak{m}}_{1}(\boldsymbol{2}^{n}) & =-\sum_{d \leq n} \sum_{w(\textbf{m})=2n+1 \atop ht(\textbf{m})=d(\textbf{m})=d} 2^{2n+1-2d}\zeta^{\mathfrak{m}}(\textbf{m}) \\
&=\sum_{2n+1=\sum s_{k}(2i_{k}+1)+2S \atop i_{k}\neq i_{j}}  \left( \prod_{k=1}^{p} \frac{C_{i_{k}}^{s_{k}}} {s_{k}!} \zeta^{\mathfrak{m}}(\overline{2i_{k}+1})^{s_{k}}\right) D_{S} \zeta^{\mathfrak{m}}(2)^{S}\nonumber
\end{align}

\item[$(iv)$] $\zeta^{\star,\mathfrak{m}}(\boldsymbol{2}^{a},3,\boldsymbol{2}^{b})= \sum A^{a,b}_{r} \zeta^{\mathfrak{m}}(\overline{2r+1})\zeta^{\star,\mathfrak{m}}(\boldsymbol{2}^{n-r}).$
\item[$(v)$] \begin{align}
\zeta^{\star\star,\mathfrak{m}}(\boldsymbol{2}^{a},3,\boldsymbol{2}^{b}) &= \sum_{w=\sum s_{k}(2i_{k}+1)+2S \atop i_{k}\neq i_{j}} B^{a,b}_{i_{1},\cdots, i_{p}\atop s_{1}\cdots s_{p}} \left( \prod_{k=1}^{p} \frac{C_{i_{k}}^{s_{k}}} {s_{k}!} \zeta^{\mathfrak{m}}(\overline{2i_{k}+1})^{s_{k}}\right)  D_{S} \zeta^{\mathfrak{m}}(2)^{S}.\\
\zeta^{\star\star,\mathfrak{m}}_{1}(\boldsymbol{2}^{a},3,\boldsymbol{2}^{b}) &=D^{a,b} \zeta^{\mathfrak{m}}(2)^{\frac{w}{2}}+ \sum_{w=\sum s_{k}(2i_{k}+1)+2S \atop i_{k}\neq i_{j}} B^{a,b}_{i_{1},\cdots, i_{p}\atop s_{1}\cdots s_{p}}  \left( \prod_{k=1}^{p} \frac{C_{i_{k}}^{s_{k}}} {s_{k}!} \zeta^{\mathfrak{m}}(\overline{2i_{k}+1})^{s_{k}}\right)  D_{S}\zeta^{\mathfrak{m}}(2)^{S}.
\end{align}
\end{itemize}
With  $C_{r}=\frac{2^{2r+1}}{2r+1}$ and the other rational coefficients satisfying: 
\begin{itemize}
\item[$\cdot$] \begin{equation} \label{eq:constrainta}
A^{a,b}_{r}=A_{r}^{a,r-a-1}+C_{r} \left( B^{r-b-1,b}- B^{r-a-1,a} +\delta_{r\leq b}-\delta_{r\leq a} \right). 
\end{equation}
\item[$\cdot$] The recursive formula for $B$-coefficients, where $B^{x,y}\mathrel{\mathop:}=B^{x,y}_{x+y+1 \atop 1}$ and $r<a+b+1$:
  \begin{equation} \label{eq:constraintb} 
\begin{array}{lll }
 B^{a,b}_{r \atop 1} & = & \delta_{r\leq b} -  \delta_{r< a}+ B^{r-b-1,b}+\frac{D^{a-r-1,b}}{a+b-r+1}+\delta_{r=a} \frac{2(2^{2b+1}-1)6^{b+1} \mid B_{2b+2} \mid}{(2b+2)! D_{b+1}}.\\
 B^{a,b}_{i_{1},\cdots, i_{p}\atop s_{1}\cdots s_{p}} &=& \left\{
\begin{array}{l}
 \delta_{i_{1}\leq b } - \delta_{i_{1}< a } + B^{i_{1}-b-1,b} + B^{a-i_{1}-1,b}_{i_{1}, \ldots, i_{p}\atop s_{1}-1, \ldots, s_{p}}  \quad  \text{ for } \sum s_{k} \text{ odd }  \\
 \delta_{i_{1}\leq b } - \delta_{i_{1}\leq a } + B^{i_{1}-b-1,b} +B^{a-i_{1},b}_{i_{1}, \ldots, i_{p}\atop s_{1}-1, \ldots, s_{p}}  \quad \text{ else }.
  \end{array}
  \right. 
\end{array}  
    \end{equation}
\end{itemize}

\end{lemm}
\noindent
Before giving the proof, here is the (analytic) conjecture remaining on some of these coefficients, sufficient to complete the Hoffman $\star$ basis proof (cf. Theorem $\autoref{Hoffstar}$):
\begin{conj}\label{conjcoeff}
The equalities $(v)$ are satisfied for real MZV, with:
$$B^{a,b}=1-\frac{2}{C_{a+b+1}}\binom{2a+2b+2}{2b+1}.$$
\end{conj}
\noindent
\texttt{Nota Bene:}  This conjecture is of an entirely different nature from the techniques developed in this thesis. We can expect that it can proved using analytic methods as the usual techniques of identifying hypergeometric series, as in $\cite{Za}$, or $\cite{Li}$.

\begin{theo}
If the analytic conjecture ($\ref{conjcoeff}$) holds, the equalities $(iv)$, $(v)$ are true in the motivic case, with the same values of the coefficients. In particular:
$$A_{r}^{a,b}= 2\left( -\delta_{r=a}+ \binom{2r}{2a} \right)  \frac{2^{2r}}{2^{2r}-1}-2\binom{2r}{2b+1}.$$
\end{theo}
\noindent
\texttt{Nota Bene:}  The equality $(iv)$ is already proven in the analytic case by Ohno-Zagier (cf.$\cite{IKOO}$, $\cite{Za}$), with these values for $A_{r}^{a,b}$. Nevertheless, as the proof highlights, to deduce these values for the motivic identity $(iv)$, we need the analytic version of $(v)$. 
\begin{proof}
Remind that if we know a motivic equality up to one unknown coefficient (of $\zeta(w)$), the analytic result analogue enables us to conclude on its value by Theorem $\autoref{kerdn}$.\\
Let assume now, in a recursion on $n$, that we know $\left\lbrace  B^{a,b}, D^{a,b}, B_{i_{1} \cdots i_{p} \atop s_{1} \cdots s_{p} }^{a,b} \right\rbrace _{a+b+1<n}$ and consider $(a,b)$ such that $a+b+1=n$. Then, by $(\ref{eq:constraintb})$, we are able to compute the $B_{\textbf{i}\atop \textbf{s}}^{a,b}$ with $(s,i)\neq (1,n)$. Using the analytic $(v)$ equality, and Theorem $\autoref{kerdn}$, we deduce the only remaining unknown coefficient $B^{a,b}$ resp. $D^{a,b}$ in $(v)$.\\
Lastly, by recursion on $n$ we deduce the $A_{r}^{a,b}$ coefficients: let assume they are known for $a+b+1<n$, and take $(a,b)$ with $a+b+1=n$. By 
the constraint $(\ref{eq:constrainta})$, since we already know $B$ and $C$ coefficients, we deduce $A_{r}^{a,b}$ for $r<n$. The remaining coefficient, $A_{n}^{a,b}$, is obtained using the analytic $(iv)$ equality and Theorem $\autoref{kerdn}$.
\end{proof}

\paragraph{\texttt{Proof of Lemma} $\autoref{lemmcoeff}$.}:
\begin{proof}
Computing the coaction on these elements, by a recursive procedure, we are able to prove these identities up to some rational coefficients, with Theorem $\autoref{kerdn}$. When the analytic analogue of the equality is known for MZV, we \textit{may} conclude on the value of the remaining rational coefficient of $\zeta^{\mathfrak{m}}(\omega)$ by identification (as for $(i),(ii),(iii)$). However, if the family is not stable under the coaction (as for $(iv)$), knowing the analytic case is not enough.\\
\texttt{Nota Bene:} This proof refers to the expression of $D_{2r+1}$ in Lemma $\autoref{lemmt}$: we look at cuts of length $2r+1$ among the sequence of $0, 1, $ or $\star$ (in the iterated integral writing); there are different kind of cuts (according their extremities), and each cut may bring out two terms ($T_{0,0}$ and $T_{0,\star}$ for instance). The simplifications are illustrated by the diagrams, where some arrows (term of a cut) get simplified by rules specified in Appendix $A$.\\
\begin{itemize}
\item[$(i)$] The corresponding iterated integral is:  $I^{\mathfrak{m}}(0; 1, 0, \star, 0 \cdots, \star, 0; 1)$.\\
The only possible cuts of odd length are between two $\star$ ($T_{0,\star}$ and $T_{\star,0}$) or $T_{1,0}$ from the first $1$ to a $\star$, or $T_{0,1}$ from a $\star$ to the last $1$. By \textsc{ Shift }(\ref{eq:shift}), these cuts get simplified two by two. Since $D_{2r+1}(\cdot)$, for $2r+1<2n$ are all zero, it belongs to $\mathbb{Q}\zeta^{\mathfrak{m}}(2n)$, by Theorem $\autoref{kerdn}$. Using the (known) analytic equality, we can conclude.
\item[$(ii)$] It is quite similar to $(i)$: using $\textsc{ Shift }$ $(\ref{eq:shift})$, it remains only the cut:\\
\includegraphics[]{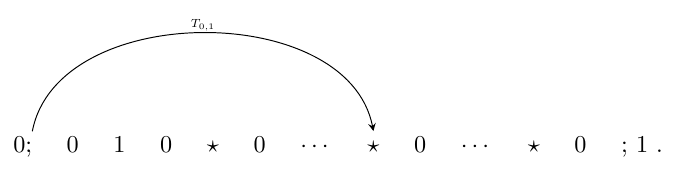}\\
$$\text{i.e.}: \quad D_{2r+1} (\zeta^{\star,\mathfrak{m}}_{1}(\boldsymbol{2}^{n}))= \zeta^{\mathfrak{l},\star}_{1}(\boldsymbol{2}^{r})\otimes \zeta^{\star,\mathfrak{m}}(\boldsymbol{2}^{n-r})=-2 \zeta^{ \mathfrak{l}}(\overline{2r+1})\otimes \zeta^{ \star,\mathfrak{m}}(\boldsymbol{2}^{n-r}).$$
The last equality is deduced from the recursive hypothesis (smaller weight). The analytic equality (coming from the Zagier-Ohno formula, and the $\shuffle$ regulation) enables us to conclude on the value of the remaining coefficient of $\zeta^{\mathfrak{m}}(2n+1)$.
\item[$(iii)$] Expressing these ES$\star\star$ as a linear combination of ES by $\shuffle$ regularisation:
$$\hspace*{-0.5cm}\zeta^{\star\star,\mathfrak{m}}(\boldsymbol{2}^{n})= \sum_{k_{i} \text{ even}} \zeta^{\mathfrak{m}}_{2n-\sum k_{i}}(k_{1},\cdots, k_{p})=\sum_{n_{i}\geq 2} \left(  \sum_{k_{i} \text{ even} \atop k_{i} \leq n_{i}}  \binom{n_{1}-1}{k_{1}-1} \cdots \binom{n_{d}-1}{k_{d}-1}  \right) \zeta^{\mathfrak{m}}(n_{1},\cdots, n_{d}) .$$
Using the multi-binomial formula:
$$2^{\sum m_{i}}=\sum_{l_{i} \leq m_{i}} \binom{m_{1}}{l_{1}}(1-(-1))^{l_{1}} \cdots \binom{m_{d}}{l_{d}}(1-(-1))^{l_{d}}= 2^{d} \sum_{l_{i} \leq m_{i}\atop l_{i} \text{ odd  }} \binom{m_{1}}{l_{1}} \cdots \binom{m_{d}}{l_{d}} .$$
Thus:
$$\zeta^{\star\star,\mathfrak{m}}(\boldsymbol{2}^{n})=\sum_{d \leq n} \sum_{w(\textbf{m})=2n \atop ht(\textbf{m})=d(\textbf{m})=d} 2^{2n-2d}\zeta^{\mathfrak{m}}(\textbf{m}).$$
Similarly for $(35)$, since:
$$\zeta^{\star\star,\mathfrak{m}}_{1}(\boldsymbol{2}^{n})=\sum_{k_{i} \text{ even}} \zeta^{\mathfrak{m}}_{2n+1-\sum k_{i}}(k_{1},\cdots, k_{p})=\sum_{d \leq n} \sum_{w(\textbf{m})=2n \atop ht(\textbf{m})=d(\textbf{m})=d} 2^{2n-2d}\zeta^{\mathfrak{m}}(\textbf{m}).$$
Now, using still only $\textsc{ Shift }$ $(\ref{eq:shift})$, it remains the following cuts:\\
\includegraphics[]{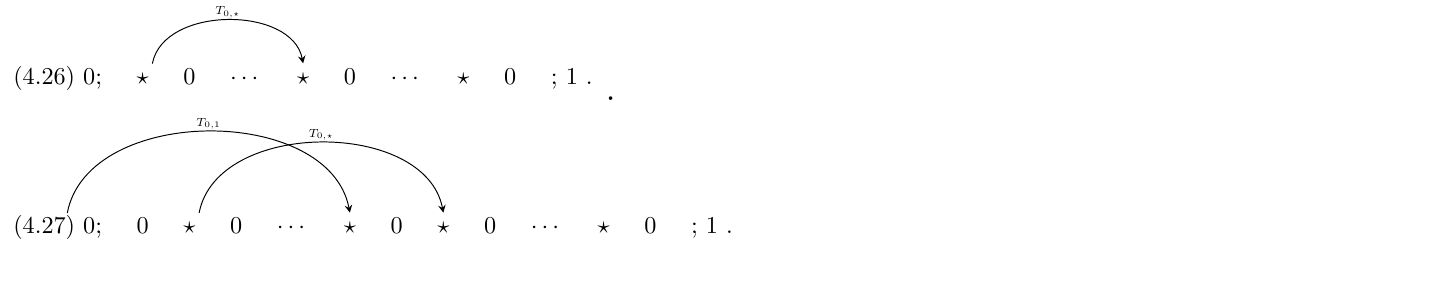}\\
With a recursion on $n$ for both $(34)$, $(35)$, we deduce:
$$D_{2r+1}(\zeta^{\star\star,\mathfrak{m}}(\boldsymbol{2}^{n}))=\zeta^{\star\star,\mathfrak{m}}_{1}(\boldsymbol{2}^{r})\otimes \zeta^{\star\star,\mathfrak{l}}_{1}(\boldsymbol{2}^{n-r})=C_{r} \zeta^{ \mathfrak{l}}(\overline{2r+1})\otimes  \zeta^{\star\star,\mathfrak{m}}_{1}(\boldsymbol{2}^{n-r-1}).$$
$$D_{2r+1}(\zeta^{\star\star,\mathfrak{m}}_{1}(\boldsymbol{2}^{n}))=\zeta^{\star\star,\mathfrak{l}}_{1}(\boldsymbol{2}^{r})\otimes \zeta^{\star\star,\mathfrak{m}}(\boldsymbol{2}^{n-r})=C_{r} \zeta^{ \mathfrak{l}}(\overline{2r+1})\otimes \zeta^{\star\star,\mathfrak{m}}(\boldsymbol{2}^{n-r}).$$
To find the remaining coefficients, we need the analytic result corresponding, which is a consequence of the sum relation for MZV of fixed weight, depth and height, by Ohno and Zagier ($\cite{OZa}$, Theorem $1$), via the hypergeometric functions.\\
Using $\cite{OZa}$, the generating series of these sums is, with $\alpha,\beta=\frac{x+y \pm \sqrt{(x+y)^{2}-4z}}{2}$:
$$\begin{array}{lll}
\phi_{0}(x,y,z)\mathrel{\mathop:} & = &  \sum_{s\leq d \atop w\geq d+s}  \left( \sum \zeta(\textbf{k}) \right) x^{w-d-s}y^{d-s}z^{s-1} \\
& = & \frac{1}{xy-z} \left( 1- \exp \left( \sum_{m=2}^{\infty} \frac{\zeta(m)}{m}(x^{m}+y^{m}-\alpha^{m}-\beta^{m}) \right)   \right) .
\end{array}$$
From this, let express the generating series of both $\zeta^{\star\star}(\boldsymbol{2}^{n})$ and $\zeta^{\star\star}_{1}(\boldsymbol{2}^{n})$:
$$\phi(x)\mathrel{\mathop:}= \sum_{w} \left( \sum_{ht(\textbf{k})=d(\textbf{k})=d\atop w\geq 2d} 2^{w-2d} \zeta(\textbf{k}) \right) x^{w-2}= \phi_{0}(2x, 0, x^{2}).$$
Using the result of Ohno and Don Zagier:
$$\phi(x)= \frac{1}{x^{2}} \left(\exp \left( \sum_{m=2}^{\infty} \frac{2^{m}-2}{m} \zeta(m) x^{m} \right)   -1\right).$$
Consequently, both $\zeta^{\star\star}(\boldsymbol{2}^{n})$ and $\zeta^{\star\star}_{1}(\boldsymbol{2}^{n})$ can be written explicitly as polynomials in simple zetas. For $\zeta^{\star\star}(\boldsymbol{2}^{n})$, by taking the coefficient of $x^{2n-2}$ in $\phi(x)$:
$$\zeta^{\star\star}(\boldsymbol{2}^{n})= \sum_{\sum m_{i} s_{i}=2n \atop m_{i}\neq m_{j}} \prod_{i=1}^{k} \left(  \frac{1}{s_{i} !}\left( \zeta(m_{i}) \frac{2^{m_{i}}-2}{m_{i}}\right)^{s_{i}} \right) .$$
Gathering the zetas at even arguments, it turns into:
$$\zeta^{\star\star}(\boldsymbol{2}^{n})= \sum_{\sum (2i_{k}+1) s_{k}+2S=2n \atop i_{k}\neq i_{j}} \prod_{i=1}^{p} \left(  \frac{1}{s_{k} !}\left( \zeta(2 i_{k}+1) \frac{2^{2 i_{k}+1}-2}{2i_{k}+1}\right)^{s_{k}} \right) d_{S} \zeta(2)^{S}, $$
\begin{equation}\label{eq:coeffds}
 \text{ where }  d_{S}\mathrel{\mathop:}=3^{S}\cdot 2^{3S}\sum_{\sum m_{i} s_{i}=S \atop m_{i}\neq m_{j}} \prod_{i=1}^{k} \left( \frac{1}{s_{i}!} \left( \frac{\mid B_{2m_{i}}\mid (2^{2m_{i}-1}-1) } {2m_{i} (2m_{i})!}\right)^{s_{i}}  \right).
\end{equation}
It remains to turn $\zeta(odd)$ into $\zeta(\overline{odd})$ by $(o)$ to fit the expression of the Lemma:
$$\zeta^{\star\star}(\boldsymbol{2}^{n})= \sum_{\sum (2i_{k}+1) s_{k}+2S=2n \atop i_{k}\neq i_{j}} \prod_{i=1}^{p} \left(  \frac{1}{s_{k} !}\left( c_{i_{k}}\zeta(\overline{2 i_{k}+1}) \right)^{s_{k}} \right) d_{S} \zeta(2)^{S}, \text{ where } c_{r}=\frac{2^{2r+1}}{2r+1}.$$
It is completely similar for $\zeta^{\star\star}_{1}(\boldsymbol{2}^{n})$: by taking the coefficient of $x^{2n-3}$ in $\phi(x)$, we obtained the analytic analogue of $(35)$, with the same coefficients $d_{S}$ and $c_{r}$.\\
Now, using these analytic results for $(34)$, $(35)$, by recursion on the weight, we can identify the coefficient $D_{S}$ and $C_{r}$ with resp. $d_{S}$ and $c_{r}$, since there is one unknown coefficient at each step of the recursion.
\item[$(iv)$] After some simplifications (antipodes, Appendix $A$), only the following cuts remain:\\
\includegraphics[]{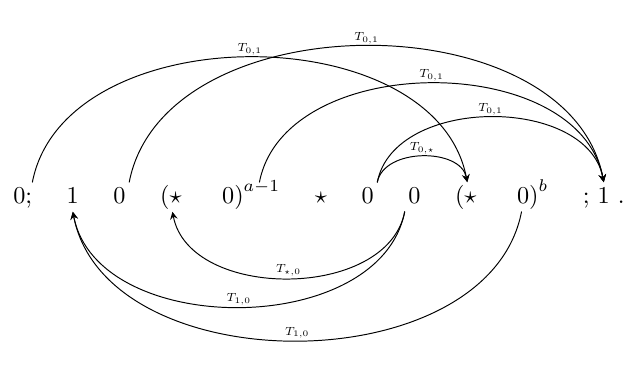}\\
This leads to the formula:\\
$$D_{2r+1} (\zeta^{\star,\mathfrak{m}}(\boldsymbol{2}^{a},3,\boldsymbol{2}^{b}))=  \left(\zeta^{\star,\mathfrak{m}}(\boldsymbol{2}^{a},3,\boldsymbol{2}^{r-a-1})+\right.$$
$$ \left.\left(  \delta_{r \leq b}-\delta_{r \leq a}\right)  \zeta^{\star\star,\mathfrak{m}}_{1}(\boldsymbol{2}^{r})  + \zeta^{\star\star , \mathfrak{m}}(\boldsymbol{2}^{r-b-1},3,\boldsymbol{2}^{b}) -\zeta^{\star\star ,\mathfrak{m}}(\boldsymbol{2}^{r-a-1},3,\boldsymbol{2}^{a})\right) \otimes \zeta^{\star,\mathfrak{m}}(\boldsymbol{2}^{n-r}).$$
Notably, the Hoffman $\star$ family is \textit{not stable} under the coaction, so we need first to prove $(v)$, and then:
$$\hspace*{-0.7cm}D_{2r+1} (\zeta^{\star ,\mathfrak{m}}(\boldsymbol{2}^{a},3,\boldsymbol{2}^{b}))= \left( A_{r}^{a,r-a-1}+C_{r} \left( B^{r-b-1,b}- B^{r-a-1,a} +\delta_{r\leq b}-\delta_{r\leq a} \right)\right)  \zeta^{ \mathfrak{l}}(\overline{2r+1})\otimes \zeta^{\star ,\mathfrak{m}}(\boldsymbol{2}^{n-r}). $$
It leads to the constraint $(\ref{eq:constrainta})$ above for coefficients $A$. To make these coefficients explicit, apart from the known analytic Ohno Zagier formula, we need the analytic analogue of $(v)$ identities, as stated in Conjecture $\autoref{conjcoeff}$.
\item[$(v)$] By Appendix rules, the following cuts get simplified (by colors, above with below):\footnote{The vertical arrows indicates a cut from the $\star$ to a $\star$ of the same group.}\\
\includegraphics[]{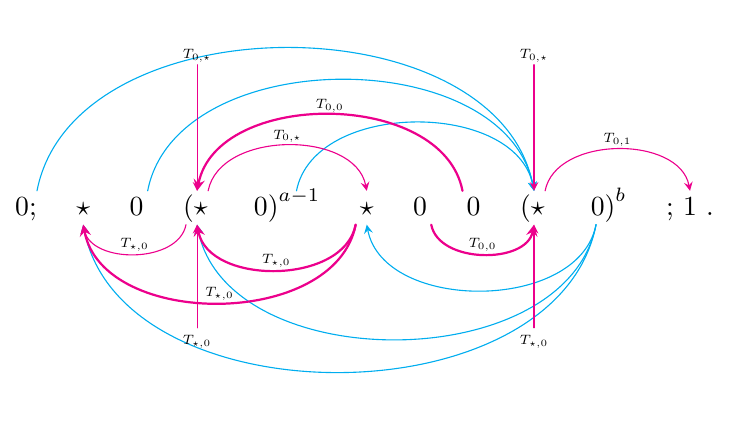}\\
Indeed, cyan arrows get simplified by \textsc{Antipode} $\shuffle$, $T_{0,0}$ resp. $T_{0, \star}$ above with $T_{0,0}$ resp. $T_{\star,0}$ below; magenta ones by $\textsc{ Shift }$ $(\ref{eq:shift})$, term above with the term below shifted by two on the left. It remains the following cuts for $(36)$:\\
\includegraphics[]{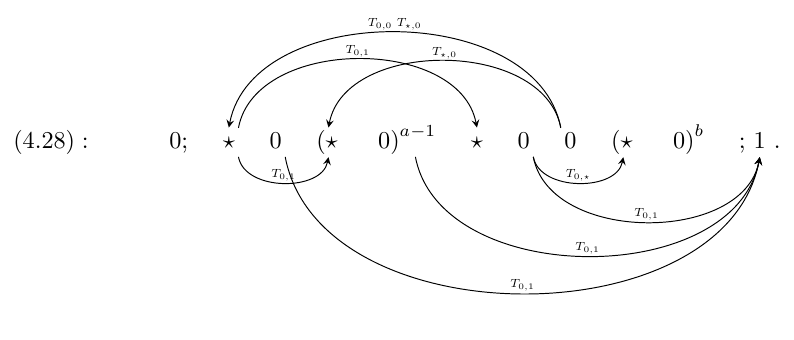}\\
In a very similar way, the simplifications lead to the following remaining terms:\\
\includegraphics[]{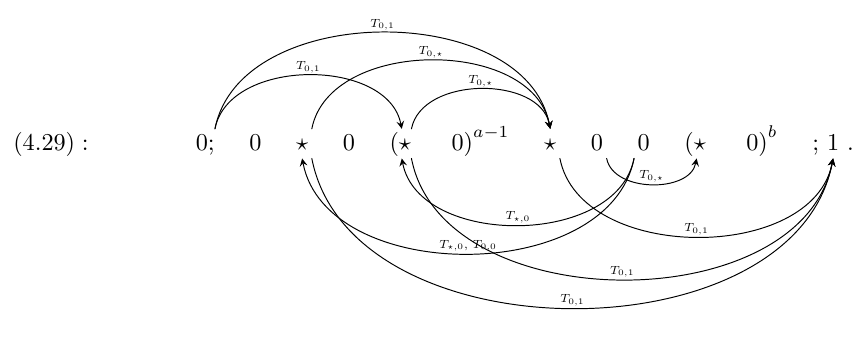}\\
Then, the derivations reduce to: 
$$\begin{array}{ll}
D_{2r+1} (\zeta^{\star\star ,\mathfrak{m}}(\boldsymbol{2}^{a},3,\boldsymbol{2}^{b})) = & 
\left( \left(  \delta_{r\leq b}-\delta_{r \leq a}\right)  \zeta^{\star\star ,\mathfrak{l}}_{1}(\boldsymbol{2}^{r}) +\delta_{r> b}\zeta^{\star\star, \mathfrak{m-l}}(\boldsymbol{2}^{r-b-1},3,\boldsymbol{2}^{b})\right) \otimes \zeta^{\star\star ,\mathfrak{m}}(\boldsymbol{2}^{n-r}) + \\
& +\delta_{r\leq a-1} \zeta^{\star\star ,\mathfrak{l}}_{1}(\boldsymbol{2}^{r}) \otimes \zeta^{\star\star ,\mathfrak{m}}_{1}(\boldsymbol{2}^{a-r-1},3,\boldsymbol{2}^{b})+ \delta_{r=a} \zeta^{\star\star ,\mathfrak{l}}_{1}(\boldsymbol{2}^{a})\otimes \zeta^{\star\star ,\mathfrak{m}}_{2}(\boldsymbol{2}^{b}) .\\
D_{2r+1} (\zeta^{\star\star ,\mathfrak{m}}_{1}(\boldsymbol{2}^{a},3,\boldsymbol{2}^{b}))= &  \left( \left(  \delta_{r\leq b}-\delta_{r \leq a}\right)  \zeta^{\star\star ,\mathfrak{l}}_{1}(\boldsymbol{2}^{r})  + \zeta^{\star\star ,\mathfrak{l}}(\boldsymbol{2}^{r-b-1},3,\boldsymbol{2}^{b})\right)\otimes \zeta^{\star\star,\mathfrak{m}}_{1}(\boldsymbol{2}^{n-r}) +\\
& + \zeta^{\star\star ,\mathfrak{l}}_{1}(\boldsymbol{2}^{r})\otimes \zeta^{\star\star,\mathfrak{m}}(\boldsymbol{2}^{a-r},3,\boldsymbol{2}^{b}) .
\end{array}$$
With a recursion on $w$ for both:
$$\hspace*{-1.4cm}\begin{array}{ll}
D_{2r+1} (\zeta^{\star\star ,\mathfrak{m}}(\boldsymbol{2}^{a},3,\boldsymbol{2}^{b})) & = C_{r} \zeta^{ \mathfrak{l}}(\overline{2r+1})\otimes\\
& \left( \left( \delta_{r\leq b}-\delta_{r < a}  + B_{r}^{r-b-1,b}\right) \zeta^{\star\star ,\mathfrak{m}}(\boldsymbol{2}^{n-r}) + \zeta^{\star\star ,\mathfrak{m}}_{1}(\boldsymbol{2^{a-r-1},}3,\boldsymbol{2}^{b})+ \delta_{r=a}\zeta^{\star ,\mathfrak{m}}(\boldsymbol{2}^{b+1}) \right) .\\
& \\
D_{2r+1} (\zeta^{\star\star ,\mathfrak{m}}_{1}(\boldsymbol{2}^{a},3,\boldsymbol{2}^{b}))& = C_{r} \zeta^{ \mathfrak{l}}(\overline{2r+1})\otimes\left( \left(  \delta_{r\leq b}-\delta_{r \leq a}   + B_{r}^{r-b-1,b}\right) \zeta^{ \star\star ,\mathfrak{m}}_{1}(\boldsymbol{2}^{n-r}) + \zeta^{\star\star ,\mathfrak{m}}(\boldsymbol{2}^{a-r},3,\boldsymbol{2}^{b}) \right).
\end{array}$$
This leads to the recursive formula $(\ref{eq:constraintb})$ for $B$.
\end{itemize}
\end{proof}


\section{Motivic Linebarger Zhao conjecture}
Conjecturally, each motivic MZV $\star$ can be expressed as a motivic ES $\sharp$ in the following way:
\begin{conj}\label{lzg}  
For $a_{i},c_{i} \in \mathbb{N}^{\ast}$, $c_{i}\neq 2$, $\gamma_{i}=c_{i}-3 + 2\delta_{c_{i}}$:
$$\zeta^{\star, \mathfrak{m}} \left( \boldsymbol{2}^{a_{0}},c_{1},\cdots,c_{p}, \boldsymbol{2}^{a_{p}}\right)  =(-1)^{1+\delta_{c_{1}}}\zeta^{\sharp,  \mathfrak{m}} \left(B_{0},\boldsymbol{1}^{\gamma_{1} },\cdots,\boldsymbol{1}^{ \gamma_{i} },B_{i}, \ldots, B_{p}\right), $$
\begin{flushright}
where $\delta_{c}=\delta_{c=1}=\left\lbrace \begin{array}{ll}
1 & \textrm{ if } c=1\\
0 & \textrm{ else }
\end{array} \right. $ and  $\left\lbrace \begin{array}{l}
B_{0}\mathrel{\mathop:}=  (-1)^{2a_{0}-\delta_{c_{1}}} (2a_{0}+1-\delta_{c_{1}})\\
B_{i}\mathrel{\mathop:}= (-1)^{2a_{i}-\delta_{c_{i}}-\delta_{c_{i+1}}} (2a_{i}+3-\delta_{c_{i}}-\delta_{c_{i+1}})\\
B_{p}\mathrel{\mathop:}=(-1)^{2a_{p}+1-\delta_{c_{p}}} ( 2 a_{p}+2-\delta_{c_{p}})
\end{array}\right.$.
\end{flushright}
\end{conj}
\texttt{Nota Bene:} The analytic version of this conjecture is proved: J. Zhao deduced it from its Theorem 1.4 in $\cite{Zh3}$.\\
\\
\textsc{Remarks}:
\begin{itemize}
\item[$\cdot$] The $B_{i}$ are then always positive odd or negative even integers.
\item[$\cdot$] Motivic Euler $\sharp$ sums appearing on the right side have already been proven to be unramified in Theorem $\autoref{ESsharphonorary}$, i.e. MMZV. 
\item[$\cdot$] This conjecture implies that the motivic Hoffman $\star$ family is a basis, since it corresponds to the motivic basis $\mathcal{B}^{\sharp}$ of Theorem $\autoref{ESsharpbasis}$: cf. ($\ref{eq:LZhoffman}$).
\item[$\cdot$] The number of sequences of consecutive $1$ in $\zeta^{\star}$, $n_{1}$ is linked with the number of even in $\zeta^{\sharp}$, $n_{e}$, here: $n_{e}=1+2n_{1}-2\delta_{c_{p}} -\delta_{c_{1}}$.
\end{itemize}
Special cases of this conjecture (references for proof for real ES indicated in the braket):
\begin{description}
\item[Two-One] [Ohno Zudilin, $\cite{OZ}$.]
\begin{equation}\label{eq:OZ21}
\zeta^{\star, \mathfrak{m}} (\boldsymbol{2}^{a_{0}},1,\cdots,1, \boldsymbol{2}^{a_{p}})= - \zeta^{\sharp,  \mathfrak{m}} \left( \overline{2a_{0}}, 2a_{1}+1, \ldots, 2a_{p-1}+1, 2 a_{p}+1\right) . 
\end{equation}
\item [Three-One] [Broadhurst et alii, $\cite{BBB}$.] 
\begin{equation}\label{eq:Z31} \zeta^{\star, \mathfrak{m}} (\boldsymbol{2}^{a_{0}},1,\boldsymbol{2}^{a_{1}},3 \cdots,1, \boldsymbol{2}^{a_{p-1}}, 3, \boldsymbol{2}^{a_{p}}) = -\zeta^{\sharp,  \mathfrak{m}} \left( \overline{2a_{0}}, \overline{2a_{1}+2}, \ldots, \overline{2a_{p-1}+2}, \overline{2 a_{p}+2} \right) .
\end{equation}
\item[Linebarger-Zhao$\star$] [Linebarger Zhao, $\cite{LZ}$]  With $c_{i}\geq 3$:
\begin{equation}\label{eq:LZ}
\zeta^{\star, \mathfrak{m}} \left( \boldsymbol{2}^{a_{0}},c_{1},\cdots,c_{p}, \boldsymbol{2}^{a_{p}}\right)  =
-\zeta^{\sharp,  \mathfrak{m}} \left( 2a_{0}+1,\boldsymbol{1}^{ c_{1}-3  },\cdots,\boldsymbol{1}^{ c_{i}-3  },2a_{i}+3, \ldots, \overline{ 2 a_{p}+2} \right) 
\end{equation}
In particular, when all $c_{i}=3$:
\begin{equation}\label{eq:LZhoffman}
\zeta^{\star, \mathfrak{m}} \left( \boldsymbol{2}^{a_{0}},3,\cdots,3, \boldsymbol{2}^{a_{p}}\right)  = - \zeta^{\sharp,  \mathfrak{m}} \left( 2a_{0}+1, 2a_{1}+3, \ldots, 2a_{p-1}+3, \overline{2 a_{p}+2}\right) . 
\end{equation}
\end{description}
\texttt{Examples}: Particular identities implied:
\begin{itemize}
\item[$\cdot$] $ \zeta^{\star, \mathfrak{m}}(1, \left\lbrace 2 \right\rbrace^{n} )=2 \zeta^{ \mathfrak{m}}(2n+1).$
\item[$\cdot$] $ \zeta^{\star, \mathfrak{m}}(1, \left\lbrace 2 \right\rbrace^{a}, 1, \left\lbrace 2 \right\rbrace^{b} )= \zeta^{ \sharp\mathfrak{m} }(2a+1,2b+1)= 4 \zeta^{ \mathfrak{m} }(2a+1,2b+1)+ 2 \zeta^{ \mathfrak{m} }(2a+2b+2). $
\item[$\cdot$] $ \zeta^{ \mathfrak{m}} (n)= - \zeta^{\sharp, \mathfrak{m}} (\lbrace 1\rbrace^{n-2}, -2)= -\sum_{ w(\boldsymbol{k})=n \atop \boldsymbol{k} \text{admissible}} \boldsymbol{2}^{p} \zeta^{\mathfrak{m}}(k_{1}, \ldots, k_{p-1}, -k_{p}).$
\item[$\cdot$] $ \zeta^{ \star, \mathfrak{m}} (\lbrace 2 \rbrace ^{n})= \sum_{\boldsymbol{k} \in \lbrace \text{ even }\rbrace^{\times} \atop w(\boldsymbol{k})= 2n} \zeta^{\mathfrak{m}} (\boldsymbol{k})=- 2 \zeta^{\mathfrak{m}} (-2n) .$
\item[$\cdot$] And some specific examples:
$$\begin{array}{ll}
\zeta^{\star, \mathfrak{m}} (2,2,3,3,2)  =-\zeta^{\sharp, \mathfrak{m}} (5,3,-4) &\zeta^{\star, \mathfrak{m}} (5,6,2)  =-\zeta^{\sharp, \mathfrak{m}} (1,1,1,3,1,1,1,-4) \\
\zeta^{\star, \mathfrak{m}} (1,6)  =\zeta^{\sharp, \mathfrak{m}} (-2,1,1,1,-2) &\zeta^{\star, \mathfrak{m}} (2,4, 1, 2,2,3)  =-\zeta^{\sharp, \mathfrak{m}} (3,1, -2, -6,-2) .
\end{array}$$
\end{itemize}

We paved the way for the proof of Conjecture $\autoref{lzg}$, bringing it back to an identity in $\mathcal{L}$:
\begin{theo}\label{lzgt}  
Let assume that in the coalgebra $\mathcal{L}$, for odd weights:
\begin{equation}\label{eq:conjid}
 \zeta^{\sharp,  \mathfrak{l}} _{B_{0}-1}(\boldsymbol{1}^{ \gamma_{1}},\cdots, \boldsymbol{1}^{\gamma_{p}  },B_{p})\equiv
  \zeta^{\star\star, \mathfrak{l}}_{2} (\boldsymbol{2}^{a_{0}-1},c_{1},\cdots,\boldsymbol{2}^{a_{p}})-\zeta^{\star\star, \mathfrak{l}}_{1} (\boldsymbol{2}^{a_{0}}, c_{1}-1, \ldots, \boldsymbol{2}^{a_{p}}) ,
\end{equation}
with $c_{1}\geq 3$, $a_{0}>0$, $\gamma_{i}, B_{i}$ as above except $B_{p}=(-1)^{2a_{p}+\delta_{c_{p}}} (2a_{p}+3-\delta_{c_{p}})$. Then:
\begin{itemize}
\item[$(i)$] Conjecture $\autoref{lzg}$ is true.
\item[$(ii)$] In the coalgebra $\mathcal{L}$, for odd weights, with $c_{1}\geq 3$ and the previous notations:
\begin{equation}\label{eq:toolid}
\zeta^{\sharp,  \mathfrak{l}} (\boldsymbol{1}^{ \gamma_{1}},\cdots, \boldsymbol{1}^{ \gamma_{p} },B_{p})\equiv  - \zeta^{\star, \mathfrak{l}}_{1} (c_{1}-1, \boldsymbol{2}^{a_{1}},c_{2},\cdots,c_{p}, \boldsymbol{2}^{a_{p}}).
\end{equation}
\end{itemize}
\end{theo}
\begin{proof}
The proof, rather tedious, is hence given in Appendix $E$. 
\end{proof}
\textsc{Remark:} This hypothesis (only valid $\mathcal{L}^{2}$, not in $\mathcal{H}^{2}$) should be proven either directly via the various relations in $\mathcal{L}$ proven in section 3 (as for $\ref{eq:toolid}$), or using the coaction, which would require the analytic identity corresponding.

\newpage

\begin{appendices}
\addtocontents{toc}{\protect\setcounter{tocdepth}{1}}
\makeatletter
\addtocontents{toc}{%
  \begingroup
  \let\protect\l@chapter\protect\l@section
  \let\protect\l@section\protect\l@subsection
}
\makeatother
\section{Coaction calculus}
The coaction formula given by Goncharov and extended by Brown for motivic iterated integrals applies to the $\star$, $\star\star$, $\sharp$ or $\sharp\sharp$ version by linearity ($\ref{eq:miistarsharp}$). Here is the version obtained for MMZV $\star,\star\star$, $\sharp$ or $\sharp\sharp$:\footnote{\textit{For purpose of stability}: if there is a $\pm 1$ at the beginning, as for $\star$  or $\sharp$ versions, the cut with this first $\pm 1$ will be let as a $T_{\pm 1, 0}$ term (and not converted into a $T_{\epsilon, 0}$ less a $T_{0,0}$), in order to still have a $\pm 1$ at the beginning; whereas, if there is no $\pm 1$ at the beginning, as for $\star\star$ or $\sharp\sharp$ version even the first cut (first line) has to be converted into a $T_{0, \epsilon}$ less a $T_{0,0}$, in order to still have a $\epsilon$ at the beginning.}
\begin{lemm}\label{lemmt}
$L$ being a sequence in $\lbrace 0, \pm\star\rbrace$ resp. $\lbrace 0,\pm\sharp\rbrace$, with possibly $1$ at the beginning,  $\epsilon \in \lbrace \pm\star \rbrace$ resp. $\in \lbrace\pm\sharp\rbrace$, and $s_{\epsilon}\mathrel{\mathop:}=\textrm{sign}(\epsilon)$.
$$D_{r} I^{\mathfrak{m}}_{s}\left(0;L;1 \right)=\delta_{ L= A \epsilon B \atop w(A)=r}   I^{\mathfrak{l}}_{k}\left(0;A ;s_{\epsilon} \right)  \otimes  I^{\mathfrak{m}}_{s-k}\left(0;s_{\epsilon}, B ;1 \right) $$
$$+ \sum_{L=A \epsilon B 0 C \atop w(B)=r}  I^{\mathfrak{l}}\left(s_{\epsilon};B ;0 \right) \otimes \left( \underbrace{I^{\mathfrak{m}}_{s}\left(0;A, \epsilon, 0, C ;1\right)}_{T_{\epsilon,0}} + \underbrace{I^{\mathfrak{m}}_{s}(0;A,0,0,C ;1)} _{T_{0,0}} \right)$$
$$+ \sum_{L=A 0 B \epsilon C \atop w(B)=r}  I^{\mathfrak{l}}\left(0;B, s_{\epsilon}\right) \otimes \left( \underbrace{I^{\mathfrak{m}}_{s}\left(0;A,0, \epsilon, C ;1 \right)}_{T_{0, \epsilon}} + \underbrace{I^{\mathfrak{m}}_{s}(0;A,0,0,C ;1)}_{T_{0,0}} \right)$$
$$+ \sum_{L=A \epsilon B \epsilon C\atop w(B)=r}  I^{\mathfrak{l}}\left(0;B, s_{\epsilon}\right) \otimes \left( \underbrace{I^{\mathfrak{m}}_{s}(0;A,\epsilon,0,C ;1)}_{T_{\epsilon,0}} - \underbrace{I^{\mathfrak{m}}_{s}\left(0;A,0, \epsilon, C ;1 \right)}_{T_{0,\epsilon}}  \right)$$

\begin{multline}\nonumber
+ \sum_{L=A  \epsilon B -\epsilon C\atop w(B)=r} \left[  I^{\mathfrak{l}}\left(0;B; -s_{\epsilon} \right) \otimes \underbrace{I^{\mathfrak{m}}_{s}(0;A,\epsilon, 0,C ;1)}_{T_{\epsilon,0}} + I^{\mathfrak{l}}\left(s_{\epsilon};B;0\right)\otimes \underbrace{I^{\mathfrak{m}}_{s}(0; A,0,-\epsilon, C ;1)}_{T_{0,\epsilon}} \right.  \\
\left. I^{\mathfrak{l}}\left(s_{\epsilon};B; -s_{\epsilon} \right) \otimes  \underbrace{I^{\mathfrak{m}}_{s}\left(0;A, \epsilon, - \epsilon, C ;1 \right)}_{T_{\epsilon, -\epsilon}}  \right] .
\end{multline}
\end{lemm}
\begin{proof}
The proof is straightforward from $\eqref{eq:Der}$, using the linearity (by $\ref{eq:miistarsharp}$) in both ways:
\begin{itemize}
\item[$(i)$] First, to turn $\epsilon$ into a difference of $\pm 1$ minus $0$ in order to use $\eqref{eq:Der}$.
\item[$(ii)$] Then, in the right side, a $\pm 1$ appeared inside the iterated integral when looking at the usual coaction formula which is turned into a sum of a term with $\epsilon$ (denoted $T_{\epsilon,0}$ or $T_{0,\epsilon}$) and a term with $0$ (denoted $T_{0,0}$) by linearity of the iterated integrals and in order to end up only with $0,\epsilon$ in the right side.
\end{itemize}
Once listing the different cuts, we found the expression of the lemma, since:
\begin{itemize}
\item[$\cdot$] The first line corresponds to the initial cut (from the $s+1$ first $0$).
\item[$\cdot$] The second line corresponds to a cut either from $\pm\epsilon$ to $0$; the $\pm\epsilon$ being $\pm 1$.
\item[$\cdot$] The third line corresponds to a cut from $0$ to $\pm\epsilon$.
\item[$\cdot$] The fourth line corresponds to cut from $\epsilon$ to $\epsilon$, with two choices: a $\epsilon$ being fixed to $0$, the other one fixed to $1$. Replacing $1$ by $(\epsilon)+(0)$, this leads to a $T_{0,\epsilon}$, a $T_{\epsilon,0}$ and two $T_{0,0}$ terms which get simplified together.
\item[$\cdot$] The last lines correspond to cuts from $\epsilon$ to $-\epsilon$, with three possibilities: one being fixed to $0$, the other one fixed to $\pm 1$, or the first being $1$, the second $-1$. This leads to a $T_{\epsilon,0}$, a $T_{0, -\epsilon}$ and a $T_{\epsilon,-\epsilon}$, since the $T_{0,0}$ terms get simplified.
\end{itemize}
\end{proof}

\subsection{Simplification rules}
Thanks to some relations ($\S 3$) between MES in the coalgebra $\mathcal{L}$ applied to its left side, the coaction for some MES may be simplified. \\
\\
\texttt{Notations:} We use the notation of the iterated integrals inner sequences and represent a term of a cut in $D_{2r+1}$ (cf. Lemma \autoref{lemmt}) by arrows between two elements of this sequence. The weight of the cut (which is the length of the subsequence in the iterated integral) would always be considered odd here. We gather here terms in $D_{2r+1}$ according to their right side and the diagrams show which terms get simplified together: i.e. terms which have same right side, but opposite left side, by antipodal and hybrid relations. \\
\begin{description}
\item[\textsc{ Composition }:]  The composition rule ($\S. \ref{propii}$) in the coalgebra $\mathcal{L}$ boils down to:
\begin{equation} 
I^{\mathfrak{l}}(a; X; b)\equiv -I^{\mathfrak{l}}(b; X; a), \quad  \text{ with $X$ any sequence of  }  0, \pm 1, \pm \star, \pm \sharp.
\end{equation}
It allows us to switch the two extremities of the integral if we multiply by $-1$ the integral: this exchange is considerably used below, without mentioning.
\item[\textsc{ Antipode }  $\shuffle$:] It corresponds to a reversal of path for iterated integrals (cf. $\ref{eq:antipodeshuffle2}$):
\begin{center}
$I^{\mathfrak{l}}(a; X; b)\equiv(-1)^{w}I^{\mathfrak{l}}(b; \widetilde{X}; a)$ for any X sequence of $0, \pm 1, \pm \star, \pm \sharp$. Hence:
\end{center}
\begin{itemize}
\item[$\cdot$] If X symmetric, i.e.  $\widetilde{X}=X$, these two cuts get simplified,  \includegraphics[]{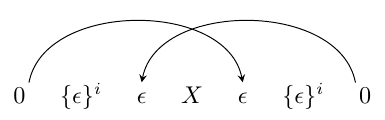}
since  $ I^{\mathfrak{l}}(\epsilon; X \epsilon^{i+1}; 0) \equiv - I^{\mathfrak{l}}(0;  \epsilon^{i+1}\widetilde{X} ; \epsilon) \equiv - I^{\mathfrak{l}}(0;  \epsilon^{i+1} X ; \epsilon)$.
\item[$\cdot$] If X antisymmetric, i.e. $\widetilde{X}=-X$,the cut \includegraphics[]{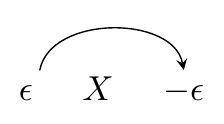}  is zero since:\\
$ \begin{array}{ll}
I^{\mathfrak{l}}(\epsilon; X; -\epsilon) & \equiv I^{\mathfrak{l}}(\epsilon; X; 0)+ I^{\mathfrak{l}}(0; X; -\epsilon)\\
&\equiv I^{\mathfrak{l}}(\epsilon; X; 0)- I^{\mathfrak{l}}(-\epsilon; \widetilde{X}; 0)  \\
&\equiv I^{\mathfrak{l}}(\epsilon; X; 0)- I^{\mathfrak{l}}(-\epsilon; -X; 0)\\
&\equiv 0
\end{array}$.
\end{itemize}
\item[\textsc{ Shift }] For MES $\star\star$ and, when weight and depth odd for Euler $\sharp\sharp$ sums:
\begin{equation}\label{eq:shift} \textsc{(Shift) } \zeta^{\bullet}_{n-1} (n_{1},\cdots, n_{p})= \zeta^{\bullet}_{n_{1}-1} (n_{2},\cdots, n_{p},n)
\end{equation}
\includegraphics[]{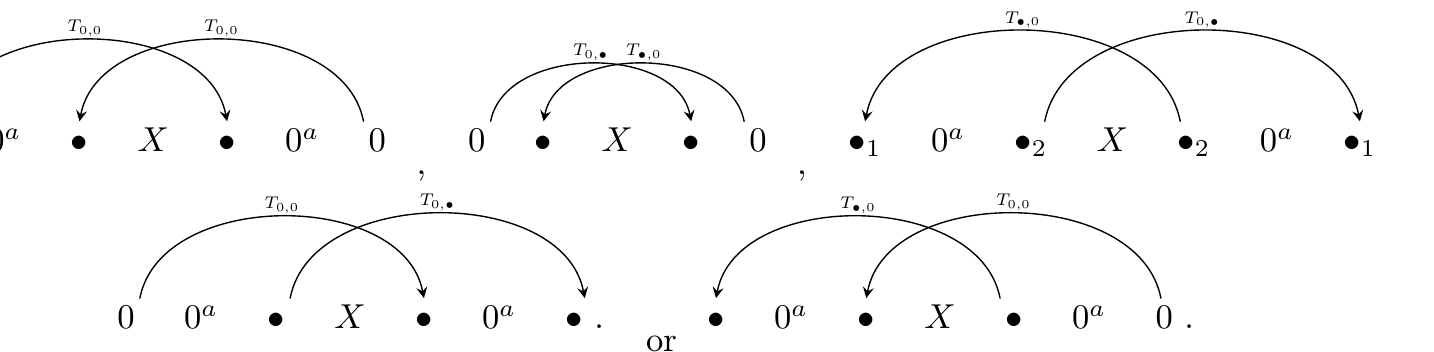}
A dot belongs to $\lbrace \pm\star,\pm\sharp\rbrace$, and two dots with a same index, $\bullet_{i}$ shall be identical.
\item[\textsc{ Cut }] For ES $\sharp\sharp$, with even depth\footnotemark[1], odd weight:\footnotetext[1]{Note that the depth considered here needed to be even is the depth of the bigger cut.}\\
\begin{equation}\label{eq:cut} 
\includegraphics[]{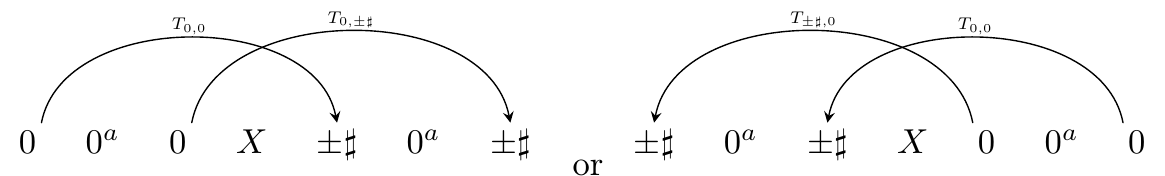}
\end{equation}
\item[\textsc{ Cut Shifted }]: For ES $\sharp\sharp$, with even depth\footnotemark[1], odd weight, composing Cut with Shift: 
\begin{equation}\label{eq:cutshifted}  
\includegraphics[]{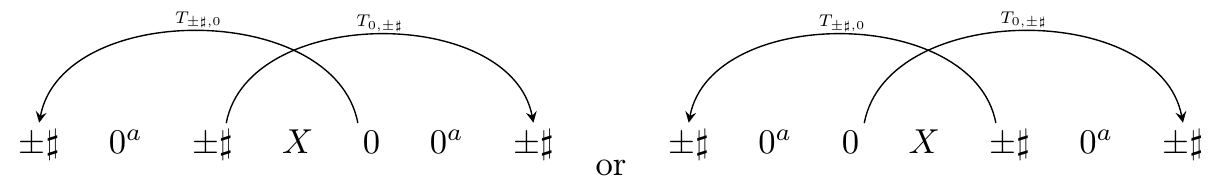}
\end{equation}
\item[\textsc{ Minus }] For ES $\sharp\sharp$, with even depth, odd weight:
\begin{equation}\label{eq:minus} 
\includegraphics[]{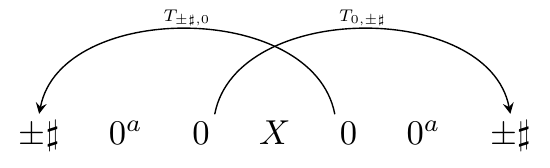}
\end{equation}
\item[\textsc{ Sign }] For ES $\sharp\sharp$ with even depth, odd weight, i.e. $X\in \lbrace 0, \pm \sharp\rbrace^{\times}$:
\begin{equation}\label{eq:sign} 
\includegraphics[]{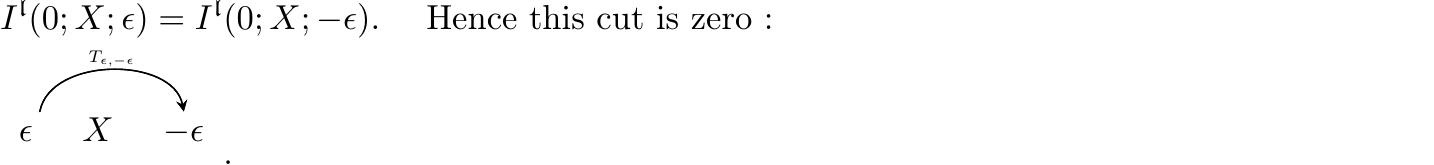}
\end{equation}
\textsc{ Sign } hence also means that the $\pm$ sign at one end of a cut does not matter.
\end{description}

\subsection{MMZV $\star$}

Let express each MMZV$^{\star}$ as: $\boldsymbol{\zeta^{\star, \mathfrak{m}} (2^{a_{0}},c_{1},\cdots,c_{p}, 2^{a_{p}})}$, which corresponds to the II:
\begin{equation}\label{eq:iistar}
I\left( 0; 1, 0, \left(\star, 0\right)^{a_{0}-1}, \star \cdots 0^{c_{i}-1} \left( \star 0 \right)^{a_{i}} \star , \ldots, 0^{c_{j}-1}  \left(\star 0 \right)^{a_{j}} \star, \ldots, 0^{c_{p}-1}  \left( \star 0 \right)^{a_{p}} ; 1\right) .
\end{equation}
Considering $D_{2r+1}$ after some simplifications:\footnotemark[2] \footnotetext{Here $\delta_{r}$ imposes that the left side must have a weight equal to $2r+1$.} 
\begin{lemm}
\begin{equation} \label{eq:DerivStar} D_{2r+1} \left(   \zeta^{\star, \mathfrak{m}} (2^{a_{0}},c_{1},\cdots,c_{p}, 2^{a_{p}})\right)  =  
\end{equation}
$$\delta_{r} \sum_{i<j} \left[    \delta_{3\leq \alpha \leq c_{i+1}-1 \atop 0\leq \beta \leq a_{j}} \zeta^{\star, \mathfrak{l}}_{c_{i+1}-\alpha} (2^{a_{j}-\beta}, \ldots, 2^{ a_{i+1}}) \otimes  \zeta^{\star, \mathfrak{m}} (\cdots,2^{a_{i}}, \alpha, 2^{\beta}, c_{j+1}, \cdots) \right.$$
$$\left( -\delta_{c_{i+1}>3} \zeta^{\star\star, \mathfrak{l}}_{2} (2^{a_{j}-\beta-1}, \ldots, 2^{ a_{i+1}})  + \delta_{c_{j+1}>3} \zeta^{\star\star, \mathfrak{l}}_{2} (2^{a_{j}}, \ldots, 2^{ a_{i+1}-\beta -1}) + \right. $$
$$ - \delta_{c_{i+1}=1} \zeta^{\star\star, \mathfrak{l}} (2^{a_{j}-\beta}, \ldots, 2^{ a_{i+1}})  + \delta_{c_{j+1}=1} \zeta^{\star\star, \mathfrak{l}} (2^{a_{i+1}-\beta}, \ldots, 2^{ a_{j}}) $$
$$+ \delta_{c_{i+2}=1 \atop \beta>a_{i+1}} \zeta^{\star\star, \mathfrak{l}}_{1} (2^{a_{j}+a_{i+1}-\beta}, \ldots, 2^{ a_{i+2}})  - \delta_{c_{j}=1 \atop \beta>a_{j}} \zeta^{\star\star, \mathfrak{l}}_{1} (2^{a_{i+1}+a_{j}-\beta}, \ldots, 2^{ a_{j-1}}) .$$
$$ \left. +\delta_{\beta > a_{i+1}}\zeta^{\star\star, \mathfrak{l}}_{c_{i+2}-2} (2^{a_{i+1}+a_{j}-\beta+1}, \ldots, 2^{ a_{i+2}})  - \delta_{\beta > a_{j}} \zeta^{\star\star, \mathfrak{l}}_{c_{j}-2} (2^{a_{i+1}+a_{j}-\beta+1}, \ldots, 2^{ a_{j-1}}) \right) $$
$$\otimes  \zeta^{\star, \mathfrak{m}} (\cdots,2^{a_{i}}, c_{i+1}, 2^{\beta}, c_{j+1}, \cdots)$$
$$\left. - \delta_{3\leq \alpha \leq c_{j}-1 \atop 0\leq \beta \leq a_{i}} \zeta^{\star, \mathfrak{l}}_{c_{j}-\alpha} (2^{a_{i}-\beta}, \ldots, 2^{ a_{j-1}}) \otimes  \zeta^{\star, \mathfrak{m}} (\cdots, c_{i}, 2^{\beta}, \alpha, 2^{a_{j}}, \cdots)\right] . $$
\end{lemm} 

\begin{proof}
We look at cuts of odd interior length between two elements of the sequence inside $\ref{eq:iistar}$. By \textsc{Shift}, the following cuts of same colors get simplified, above with below:\footnote{More precisely, in the first diagram: $T_{0,0}$ resp. $T_{0,\star}$ above get simplified with $T_{0,0}$ resp. $T_{\star,0}$ below (shifted by one at the right). The dotted arrows mean that if $c_{i}=1$ resp. $c_{j}=1$, only $T_{0,0}$ get simplified.}\\
\includegraphics[]{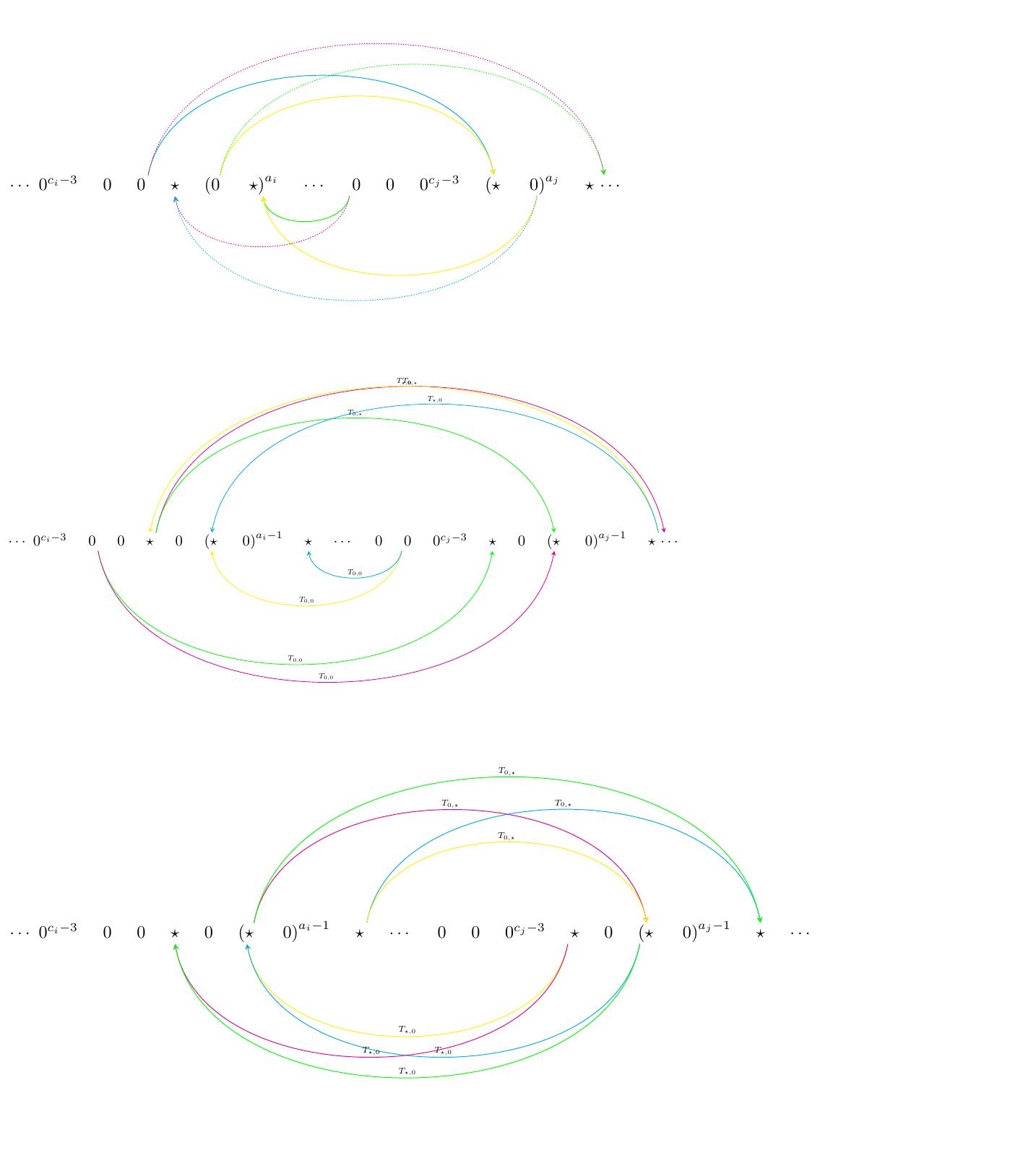}\\
It remains\footnote{Cyan arrows above resp. below are $T_{0,\star}$ resp. $T_{0,\star}$ terms; magenta ones above resp. below stand for $T_{0,0}$ and $T_{0,\star}$ resp. $T_{0,0}$ and $T_{\star,0}$ terms.}:
\includegraphics[]{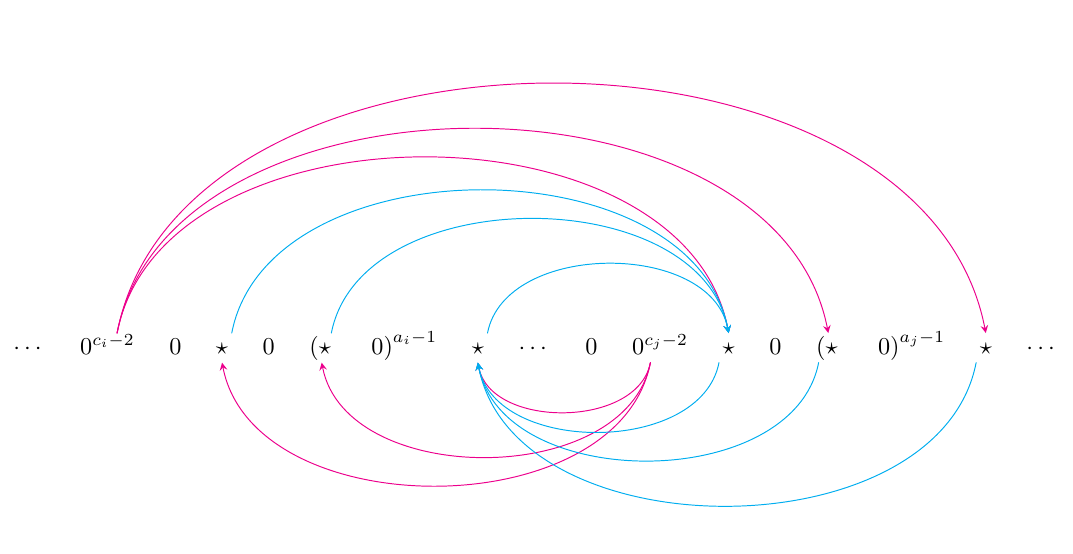}\\
More, if $c_{i}=1$:\footnote{Black, resp. cyan arrows are $T_{\star,0}$ resp. $T_{0,\star}$ terms. The case $c_{j}=1$, antisymmetric, is omitted here.}
\includegraphics[]{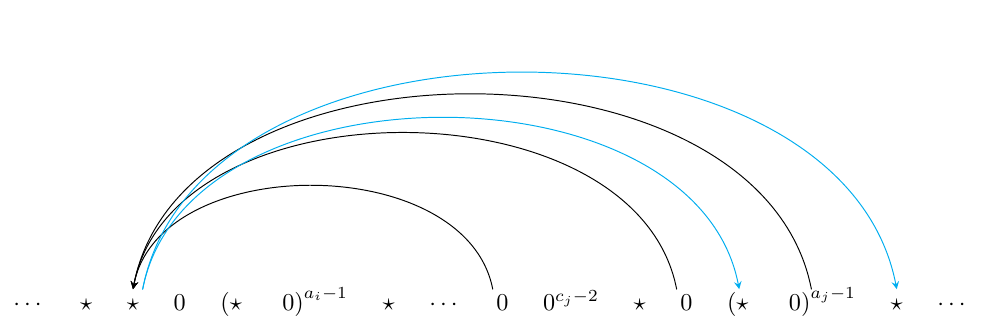}\\
Gathering the remaining cuts in this diagram, according the right side:\\
\begin{enumerate}
\item For $\zeta^{\star, \mathfrak{m}} (\cdots,2^{a_{i}}, \boldsymbol{\alpha, 2^{\beta}}, c_{j+1}, \cdots)$:
$$ \left( \textcolor{magenta}{\delta_{3\leq \alpha < c_{i+1} \atop 0\leq \beta < a_{j}}} \zeta^{\star\star, \mathfrak{l}}_{c_{i+1}-\alpha} (2^{a_{i+1}}, \ldots, 2^{ a_{j}-\beta}) - \left( \textcolor{magenta}{\delta_{4\leq \alpha < c_{i+1} \atop 0\leq \beta < a_{j}}} + \textcolor{cyan}{\delta_{\alpha=3 \atop 0\leq \beta < a_{j}}} \right) \zeta^{\star\star, \mathfrak{l}}_{c_{i+1}-\alpha+2} (2^{a_{i+1}}, \ldots, 2^{ a_{j}-\beta-1}) \right).$$
Using \textsc{Shift} $(\ref{eq:shift})$ for the first term and then the definition of $\zeta^{\star}$ it turns into:
$$ \delta_{3\leq \alpha < c_{i+1} \atop 0\leq \beta <a_{j}} \left(  \zeta^{\star\star, \mathfrak{l}}_{1} (c_{i+1}-\alpha+1, 2^{a_{i+1}}, \ldots, 2^{ a_{j}-\beta-1}) -\zeta^{\star\star, \mathfrak{l}}_{c_{i+1}-\alpha+2} (2^{a_{i+1}}, \ldots, 2^{ a_{j}-\beta-1}) \right) $$
$$=  \delta_{3\leq \alpha < c_{i+1} \atop 0\leq \beta < a_{j}}  \zeta^{\star, \mathfrak{l}}_{1} (c_{i+1}-\alpha+1, 2^{a_{i+1}}, \ldots, 2^{ a_{j}-\beta-1}). $$
Applying antipodes $A_{\shuffle} \circ A_{\ast} \circ A_{\shuffle}$, it gives $ \delta_{3\leq \alpha < c_{i+1} \atop 0\leq \beta < a_{j}}  \zeta^{\star, \mathfrak{l}}_{c_{i+1}-\alpha} (2^{ a_{j}-\beta}, \ldots, 2^{a_{i+1}}),$ and consequently the first line in $\eqref{eq:DerivStar}$.

\item For $\zeta^{\star, \mathfrak{m}} (\cdots,2^{a_{i}}, \boldsymbol{\alpha}, 2^{a_{j}}, c_{j+1}, \cdots)$, the corresponding left sides are:
$$\begin{array}{l}
 -\left( \textcolor{magenta}{\delta_{c_{j}+2\leq \alpha \leq c_{i+1}}} + \textcolor{cyan}{\delta_{\alpha=c_{j}+1\atop c_{i+1}>c_{j}} } \right)  \zeta^{\star\star, \mathfrak{l}}_{c_{i+1}+ c_{j}-\alpha} (2^{a_{i+1}}, \ldots, 2^{a_{j-1}}) \\
 +\left( \textcolor{magenta}{ \delta_{c_{i+1}+2 \leq \alpha \leq c_{j}}} + \textcolor{cyan}{ \delta_{\alpha = c_{i+1}+1 \atop c_{j}> c_{i+1}}} \right)  \zeta^{\star\star, \mathfrak{l}}_{c_{i+1}+c_{j}-\alpha} (2^{a_{j-1}}, \ldots, 2^{a_{i+1}})\\
 +\delta_{3\leq \alpha < c_{i+1}}  \zeta^{\star\star, \mathfrak{l}}_{c_{i+1}-\alpha} (2^{a_{i+1}}, \ldots, c_{j}) -\delta_{3\leq \alpha <c_{j}}\zeta^{\star\star, \mathfrak{l}}_{c_{j}-\alpha} (2^{a_{j-1}}, \ldots,c_{i+1})
\end{array}$$
Using Antipode $\ast$ and turning some $\epsilon$ into $'1+0'$:
$$=+\delta_{3\leq \alpha < c_{i+1}}  \zeta^{\star\star, \mathfrak{l}}_{c_{i+1}-\alpha} (2^{a_{i+1}}, \ldots, c_{j}) -\delta_{3\leq \alpha < c_{j}}\zeta^{\star\star, \mathfrak{l}}_{c_{j}-\alpha} (2^{a_{j-1}}, \ldots,c_{i+1})$$
$$+  (-1)^{{c_{j}<c_{i+1}}} \delta_{\min(c_{j},c_{i+1}) < \alpha \leq \max(c_{j},c_{i+1})}  \zeta^{\star\star, \mathfrak{l}}_{c_{i+1}+ c_{j}-\alpha} (2^{a_{i+1}}, \ldots, 2^{a_{j-1}}) $$
$$= \delta_{3\leq \alpha < c_{i+1}} \zeta^{\star, \mathfrak{l}}_{c_{i+1}-\alpha} (c_{j}, \ldots, 2^{ a_{i+1}}) - \delta_{3\leq \alpha < c_{j} } \zeta^{\star, \mathfrak{l}}_{c_{j}-\alpha} (c_{i+1}, \ldots, 2^{ a_{j-1}})$$
This gives exactly the same expression than the first and fourth cases for $\beta=a_{i}$ or $a_{j}$, and are integrated to them in $\eqref{eq:DerivStar}$.

\item For $\zeta^{\star, \mathfrak{m}} (\cdots, c_{i+1},\boldsymbol{2^{\beta}}, c_{j+1}, \cdots)$,  including the case $\alpha=2$:
$$\begin{array}{llll}
-\delta_{c_{i+1}>3 \atop 0 \leq \beta < a_{j}}& \zeta^{\star\star, \mathfrak{l}}_{2} (2^{a_{j}-\beta-1}, \ldots, 2^{ a_{i+1}}) & + \delta_{c_{j+1}>3 \atop 0 \leq \beta < a_{i+1}}  & \zeta^{\star\star, \mathfrak{l}}_{2} (2^{a_{j}}, \ldots, 2^{ a_{i+1}-\beta -1})\\
+\delta_{\beta > a_{i+1} \atop c_{i+1}>3 }& \zeta^{\star\star, \mathfrak{l}}_{c_{i+2}-2} (2^{a_{i+1}+a_{j}-\beta+1}, \ldots, 2^{ a_{i+2}})  &- \delta_{\beta > a_{j}} &\zeta^{\star\star, \mathfrak{l}}_{c_{j}-2} (2^{a_{i+1}+a_{j}-\beta+1}, \ldots, 2^{ a_{j-1}}) \\
- \delta_{c_{i+1}=1 \atop 1\leq \beta < a_{j}} &\zeta^{\star\star, \mathfrak{l}} (2^{a_{j}-\beta}, \ldots, 2^{ a_{i+1}})  &+ \delta_{c_{j+1}=1 \atop 1 \leq \beta <a_{i+1}} &\zeta^{\star\star, \mathfrak{l}} (2^{a_{i+1}-\beta}, \ldots, 2^{ a_{j}}) \\
+ \delta_{c_{i+2}=1 \atop \beta>a_{i+1}} &\zeta^{\star\star, \mathfrak{l}}_{1} (2^{a_{j}+a_{i+1}-\beta}, \ldots, 2^{ a_{i+2}}) & - \delta_{c_{j}=1 \atop \beta>a_{j}} &\zeta^{\star\star, \mathfrak{l}}_{1} (2^{a_{i+1}+a_{j}-\beta}, \ldots, 2^{ a_{j-1}}) .\\
\end{array}$$
\item For $ \zeta^{\star, \mathfrak{m}} (\cdots, c_{i}, \boldsymbol{2^{\beta}, \alpha}, 2^{a_{j}}, \cdots)$, antisymmetric to 1:
$$ \left( - \textcolor{magenta}{\delta_{2\leq \alpha \leq c_{j}-1 \atop 0\leq \beta \leq a_{i}}} \zeta^{\star\star, \mathfrak{l}}_{c_{j}-\alpha} (2^{a_{j-1}}, \ldots, 2^{ a_{i}-\beta}) + \left(\textcolor{magenta}{\delta_{4\leq \alpha \leq c_{j}+1 \atop 0\leq \beta \leq a_{i}-1}} + \textcolor{cyan}{\delta_{\alpha=3 \atop 0\leq \beta \leq a_{i}-1}} \right) \zeta^{\star\star, \mathfrak{l}}_{c_{j}-\alpha+2} (2^{a_{j-1}}, \ldots, 2^{ a_{i}-\beta-1}) \right) $$
\end{enumerate}
This leads to the lemma, with the second case incorporated in the first and last line.
\end{proof}

\subsection{Euler $\sharp$ sums with $\boldsymbol{\overline{even}}, \boldsymbol{odd}$}

Let us consider the following family:
$$\zeta^{\sharp, \mathfrak{m}}\left( \lbrace\boldsymbol{\overline{even}}, \boldsymbol{odd}\rbrace^{\times} \right) , \text{i.e.  negative even and positive odd integers},$$
which, in terms of iterated integrals corresponds to, with $\epsilon\in \lbrace\pm \sharp\rbrace$:
\begin{equation}\label{eq:iisharp}
I^{ \mathfrak{m}} \left( 0; \left\lbrace \begin{array}{l}
1,  \boldsymbol{0}^{odd} ,-\sharp \\
1,  \boldsymbol{0}^{even} , \sharp
\end{array}\right\rbrace  ,  \cdots, \quad \left\lbrace 
\begin{array}{l}
\epsilon, \boldsymbol{0}^{odd}, -\epsilon \\
\epsilon, \boldsymbol{0}^{even}, \epsilon 
\end{array}\right\rbrace \quad, \cdots; 1 \right) .
\end{equation}

\begin{lemm}
The family $\zeta^{\sharp \mathfrak{m}}\left( \lbrace\overline{even}, odd\rbrace^{\times} \right) $ is stable under the coaction.
\end{lemm}
\begin{proof}
Looking at the possible cuts, and gathering them according the right side:\\
\includegraphics[]{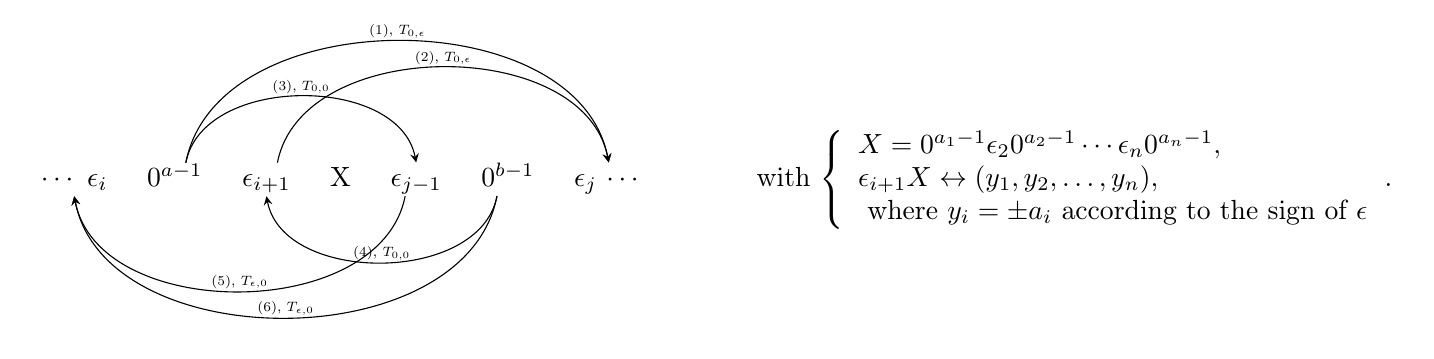}\\
These cuts have the same form for the right side in the coaction:
$$I^{\mathfrak{m}}(0; \cdots, \boldsymbol{\epsilon_{i} 0^{\alpha}  \epsilon_{j}}, \cdots ;1).$$
Notice there would be no term $T_{\epsilon, -\epsilon}$ in a cut from $\epsilon$ to $-\epsilon$ thanks to \textsc{Sign} $(\ref{eq:sign})$ identity, therefore you have there all the possible cuts pictured.\\
A priori, cuts can create in the right side a sequence $\left\lbrace  \epsilon, \boldsymbol{0}^{even}, -\epsilon \right\rbrace $ or $\left\lbrace \epsilon, \boldsymbol{0}^{odd}, \epsilon\right\rbrace $ inside the iterated integral; these cuts are the \textit{unstable} ones, since they are out of the considered family. However, by coupling these cuts two by two, and using the rules listed at the beginning of the Appendix, the unstable cuts would all get simplified.\\
Indeed, let examine each of the terms $(1-6)$\footnote{There is no remaining cuts between $\epsilon$ and $\epsilon$. Notice also that the left sides of the remaining terms have an even depth.}:\\
\\
\begin{tabular}{c | c | l | l}
Term & Left side & Unstable if & Simplified with \\
\hline
\multirow{4}{*}{$(1)$}   & $\zeta^{ \sharp\sharp,\mathfrak{l}}_{a-1-\alpha} (a_{1}, \ldots, a_{n},b)$, & $n$ even & the previous cut: \\
 &  with $\alpha<a$.  & & either $(6)$ by \textsc{Minus}   \\
 & & & or $(5)$ by \textsc{Cut}  \\
 & & &  or $(3)$ by \textsc{Cut} \footnotemark[1]\\ \hline
\multirow{4}{*}{$(2)$} & $\zeta^{ \sharp\sharp,\mathfrak{l}}_{b-1} (a_{n}, \ldots, a_{1})$ & $\epsilon_{i+1}=\epsilon_{j}$ & the previous cut: \\
 & with $\alpha=a$. & &   either with $(5)$ by \textsc{Shift} \\
  & & &  or with $(6)$ by \textsc{Cut Shifted}  \\
   & & & or with $(3)$ by \textsc{Shift}.\\ \hline
\multirow{4}{*}{$(3)$}  & $-\zeta^{ \sharp\sharp,\mathfrak{l}}_{a+b-\alpha-1} (a_{1}, \ldots, a_{n})$ &  $n$ odd & the following cut:\\ 
 & with $\alpha>b$. & &  either $(1)$ by \textsc{Cut} \\
  & & &  or $(2)$ by \textsc{Shift}  \\
   & & & or $(4)$ by \textsc{Shift}. \\ \hline
   \multirow{4}{*}{$(4)$} & $\zeta^{ \sharp\sharp,\mathfrak{l}}_{a+b-\alpha-1}(a_{n}, \ldots, a_{1})$ &  $n$ odd & the previous cut:\\
 & with $\alpha>a$ .& & either $(6)$ by \textsc{Cut Shifted}  \\
  & & &  or with $(5)$ by \textsc{Shift}  \\
   & & & or with $(3)$ by \textsc{Shift}. \\ \hline
   \multirow{4}{*}{$(5)$}& $-\zeta^{ \sharp\sharp,\mathfrak{l}}_{a-1} (a_{1}, \ldots, a_{n})$ & $\epsilon_{i}=\epsilon_{j-1}$ & the following cut:\\
 & with $\alpha=b$. & &  either  with $(1)$ by \textsc{Cut} \\
  & & &   or with $(2)$ by \textsc{Shift} \\
   & & & or with $(4)$ by \textsc{Shift}. \\ \hline
   \multirow{4}{*}{$(6)$} & $-\zeta^{ \sharp\sharp,\mathfrak{l}}_{b-1-\alpha} (a_{n}, \cdots a_{1}, a)$ & $n$ even & the following cut:\\
 & & &  either $(1)$ by \textsc{Minus} \\
  & & &   or with $(2)$ by \textsc{Cut Shifted}  \\
   & & & or with $(4)$ by \textsc{Cut Shifted} \\ \hline
\end{tabular} \\
\footnotetext[1]{It depends on the sign of $b+1-\alpha$ here for instance.}
\end{proof}
\noindent
\paragraph{Derivations.}
Let use the writing of the Conjecture $\autoref{conjcoeff}$:
\begin{equation}\label{eq:essharpgather}
 \zeta^{\sharp,\mathfrak{m}}(B_{0}, 1^{\gamma_{1}}, \ldots, 1^{\gamma_{p}}, B_{p}) \text{ with } B_{i}<0 \text{ if and only if } B_{i} \text{ even }.
\end{equation}
\texttt{Nota Bene:} Beware, for instance $B_{i}$ may be equal to $1$, which implies that $\gamma_{i}= \gamma_{i+1}=0$. Indeed, we look at the indices corresponding to a sequence $(2^{a_{0}}, c_{1}, \ldots, c_{p}, 2^{a_{p}})$ as in the Conjecture $\autoref{conjcoeff}$:
$$\begin{array}{l}
B_{i}= 2a_{i}+3 - \delta_{c_{i}}- \delta_{c_{i+1}}\\
B_{0}= 2a_{0}+1 - \delta_{c_{1}}\\
B_{p}= 2a_{p}+2 - \delta_{c_{p}}\\
\end{array}, \gamma_{i}\mathrel{\mathop:}= c_{i}-3 +2  \delta_{c_{i}}, \quad  \text{ where }   \left\lbrace  \begin{array}{l} a_{i} \geq 0 \\ c_{i}>0,c_{i}\neq 2 \\
\delta_{c}\mathrel{\mathop:}= \left\lbrace \begin{array}{ll}
1 & \text{ if } c=1\\
0 & \text{ else }.
\end{array}\right.
\end{array}.  \right.  $$
 
\begin{lemm}
\begin{equation}
D_{2r+1}\left( \zeta^{\sharp,\mathfrak{m}}(B_{0}, 1^{\gamma_{1}}, \ldots, 1^{\gamma_{p}}, B_{p})\right) =\footnotemark[2] 
\end{equation}
$$\delta_{r} \left[  -\delta_{{2 \leq B \leq B_{j}+1 \atop 0\leq\gamma\leq\gamma_{i+1}-1 }}\zeta^{\sharp,\mathfrak{l}}(B_{j}-B+1, 1^{\gamma_{j}}, \ldots, 1^{\gamma_{i+1}-\gamma-1})\otimes\zeta^{\sharp,\mathfrak{m}}(B_{0} \cdots, B_{i}, \textcolor{magenta}{1^{\gamma}, B}, 1^{\gamma_{j+1}}, \ldots, B_{p}) \right. $$

$$\left[  
\begin{array}{l}
+ \delta_{B_{i+1}< B}\zeta^{\sharp\sharp,\mathfrak{l}}_{B_{i+1}+B_{j}-B}(1^{\gamma_{j}}, \ldots, 1^{\gamma_{i+2}})\\
- \delta_{B_{j}< B}\zeta^{\sharp\sharp,\mathfrak{l}}_{B_{i+1}+B_{j}-B}(1^{\gamma_{i+2}}, \ldots, 1^{\gamma_{j}})\\
 + \zeta^{\sharp\sharp,\mathfrak{l}}_{B_{i+1}-B}(1^{\gamma_{i+2}}, \ldots, B_{j}) - \zeta^{\sharp\sharp,\mathfrak{l}}_{B_{j}-B}(1^{\gamma_{j}}, \ldots, B_{i+1})
\end{array} \right] \otimes\zeta^{\sharp,\mathfrak{m}}(B_{0} \cdots, B_{i}, 1^{\gamma_{i+1}}, \textcolor{green}{B}, 1^{\gamma_{j+1}}, \ldots, B_{p}) $$

$$\left.  \delta_{{1 \leq B \leq B_{i+1}+1 \atop 0\leq\gamma\leq\gamma_{j+1}-1}}\zeta^{\sharp,\mathfrak{l}}(B_{i+1}-B+1, 1^{\gamma_{i+2}}, \ldots, 1^{\gamma_{j+1}-\gamma-1})\otimes\zeta^{\sharp,\mathfrak{m}}(B_{0} \cdots, 1^{\gamma_{i+1}},\textcolor{cyan}{ B, 1^{\gamma}}, B_{j+1}, \ldots, B_{p}) \right],$$
where $\delta_{r}$ imposes $(\text{weight of left side}=2r+1)$ and $B$ is positive if odd, negative if even.
\end{lemm}
\begin{proof}
\texttt{Nota Bene}: For the left side, we only look at odd weight $w$, and the parity of the depth $d$ is fundamental since the relations stated above depend on the parity of $w-d$. For instance, for such a sequence $(\boldsymbol{1}^{\gamma_{i}},B_{i}, \ldots,B_{j-1},\boldsymbol{1}^{\gamma_{j}})$ (with the previous notations), $weight- depth$ has the same parity than $\delta_{c_{i}}+\delta_{c_{j}}$. \\
\\
The following cuts get simplified, with \textsc{Shift}, since depth is odd ($B_{i}$ odd if $c_{i},c_{i+1}\neq 1$):\\
 \includegraphics[]{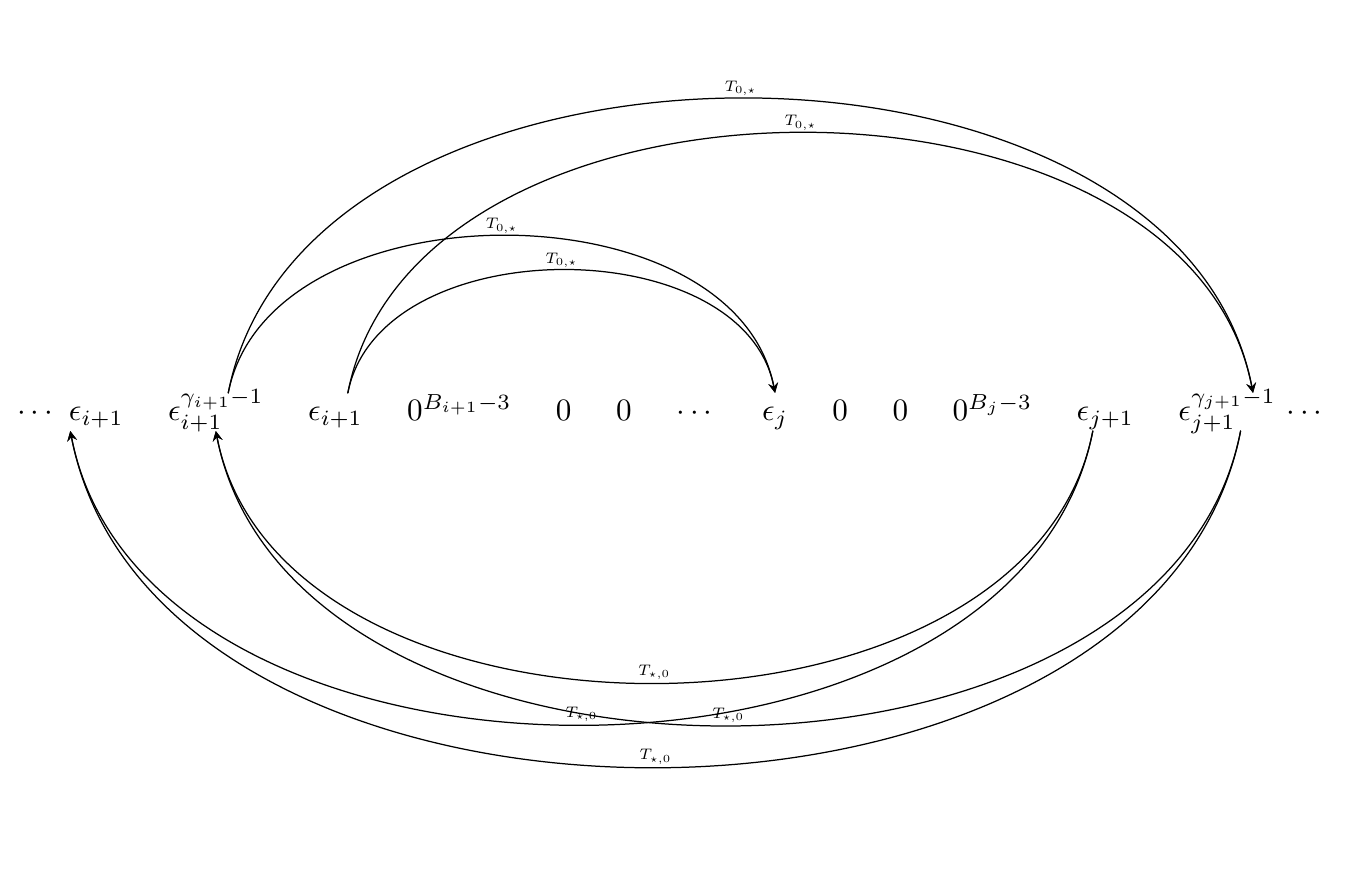}\\ 
It remains, where all the unstable cuts are simplified by the Lemma $A.4$, cuts that we can gather into four groups, according to the right side of the coaction:
\begin{itemize}
\item[$(i)$]  $\zeta^{\sharp,\mathfrak{m}}(B_{0} \cdots, B_{i}, \textcolor{magenta}{1^{\gamma}, B}, 1^{\gamma_{j+1}}, \ldots, B_{p}) $.
\item[$(ii)$]  $\zeta^{\sharp,\mathfrak{m}}(B_{0} \cdots, B_{i}, \textcolor{yellow}{1^{\gamma}}, B_{j}, \ldots, B_{p}) $.
\item[$(iii)$]  $\zeta^{\sharp,\mathfrak{m}}(B_{0} \cdots, B_{i}, 1^{\gamma_{i+1}}, \textcolor{green}{B}, 1^{\gamma_{j+1}}, \ldots, B_{p}) $.
\item[$(iv)$]  $\zeta^{\sharp,\mathfrak{m}}(B_{0} \cdots, 1^{\gamma_{i+1}},\textcolor{cyan}{ B, 1^{\gamma}}, B_{j+1}, \ldots, B_{p}) $.
\end{itemize}
It remains, where $(iv)$ terms, antisymmetric of $(i)$, are omitted to lighten the diagrams: \\
\includegraphics[]{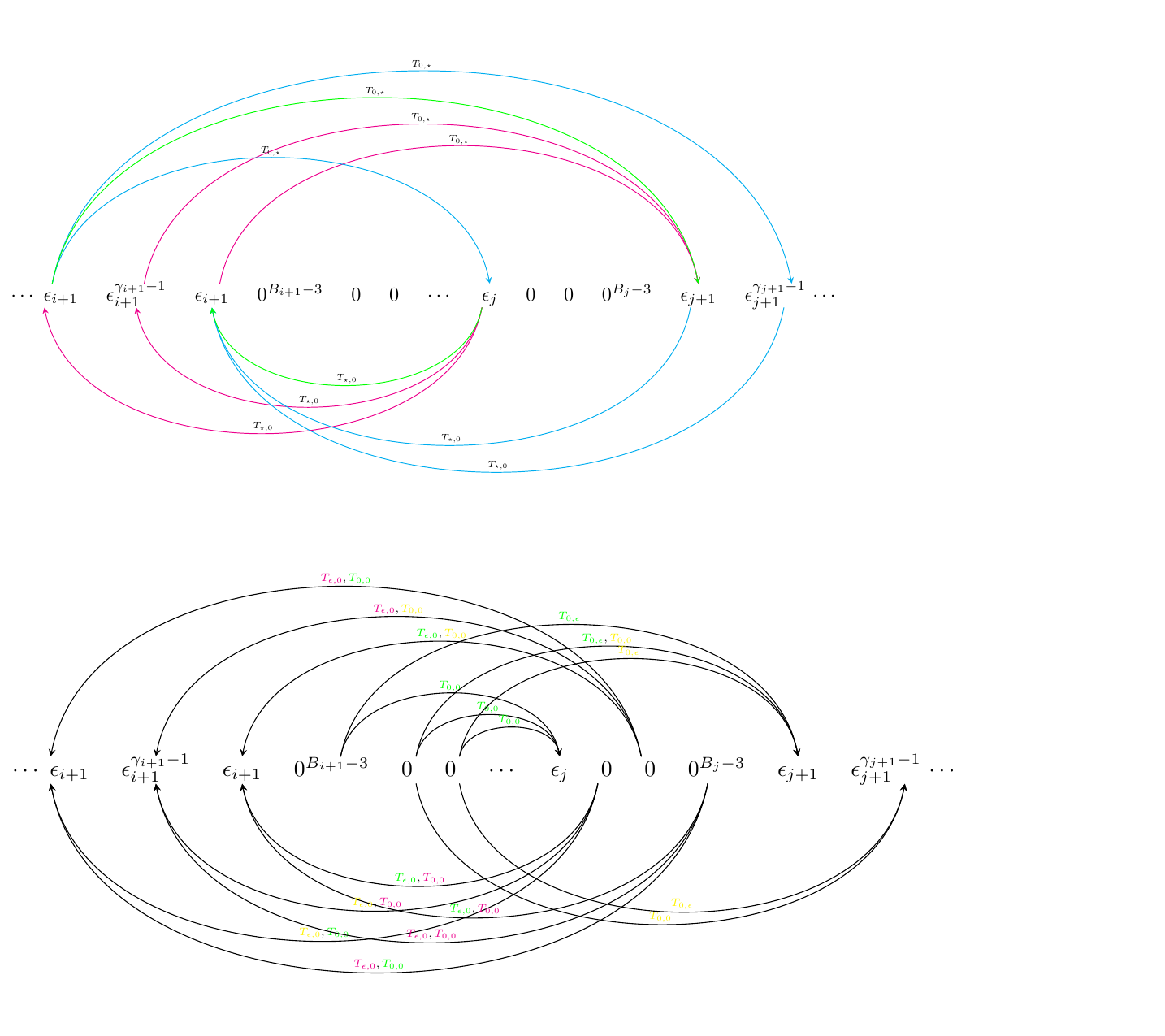}\\
Now, let list these remaining terms, gathered according to their right side as above:
\begin{itemize}
\item[$(i)$] Looking at the magenta terms, with $2 \leq B \leq B_{j}-1$ or $B=B_{j}+1$ and $0\leq\gamma\leq\gamma_{i+1}-1$:
$$\zeta^{\sharp\sharp,\mathfrak{l}}_{B_{j}-B+1}(1^{\gamma_{j}}, \ldots, 1^{\gamma_{i+1}-\gamma-1})-\zeta^{\sharp\sharp,\mathfrak{l}}_{B_{j}-B}(1^{\gamma_{j}}, \ldots, 1^{\gamma_{i+1}-\gamma}) =-\zeta^{\sharp,\mathfrak{l}}(B_{j}-B+1, 1^{\gamma_{j}}, \ldots, 1^{\gamma_{i+1}-\gamma-1})$$
With $even$ depth for the first term and $odd$ for the second since otherwise the cuts would be unstable and simplified by $\textsc{Cut}$; here also $c_{i+1}\neq 1$. 
\item[$(ii)$]  This matches exactly with the left side of $(i)$ for $B=B_{j}$ and $(iv)$ terms for $B=B_{i}$.
\item[$(iii)$]  The following cuts:
$$\delta_{B_{i+1}\geq B}\zeta^{\sharp\sharp,\mathfrak{l}}_{B_{i+1}-B}(1^{\gamma_{i+2}}, \ldots, B_{j}) -\delta_{B_{j}\geq B} \zeta^{\sharp\sharp,\mathfrak{l}}_{B_{j}-B}(1^{\gamma_{j}}, \ldots, B_{i+1})+$$ 
$$\delta_{B_{i+1}< B}\zeta^{\sharp\sharp,\mathfrak{l}}_{B_{i+1}+B_{j}-B}(1^{\gamma_{j}}, \ldots, 1^{\gamma_{i+2}}) - \delta_{B_{j}< B}\zeta^{\sharp\sharp,\mathfrak{l}}_{B_{i+1}+B_{j}-B}(1^{\gamma_{i+2}}, \ldots, 1^{\gamma_{j}}) .$$
The parity of $weight-depth$ for the first line is equal to the parity of $\delta_{c_{i+1}}+ \delta_{c_{j+1}}+B$. Notice that if this is even, the first line has odd depth whereas the second line has even depth, and by $\textsc{Cut}$ and $\textsc{Antipode} \ast$, all terms got simplified. Hence, we can restrict to $B$ written as $2\beta+3- \delta_{c_{i+1}}- \delta_{c_{j+1}}$, the first line being of even depth, the second line of odd depth.

\item[$(iv)$]  Antisymmetric of $(i)$.
\end{itemize}
\end{proof}

\section{Homographies of $\boldsymbol{\mathbb{P}^{1}\diagdown \lbrace 0, \mu_{N}, \infty\rbrace}$}

The homographies of the projective line $\mathbb{P}^{1}$ which permutes $\lbrace 0, \mu_{N}, \infty \rbrace$, induce automorphisms $\mathbb{P}^{1}\diagdown \lbrace 0, \mu_{N}, \infty\rbrace \rightarrow \mathbb{P}^{1}\diagdown \lbrace 0, \mu_{N}, \infty\rbrace$. The projective space $\mathbb{P}^{1} \diagdown \lbrace 0, \mu_{N}, \infty \rbrace$ has a dihedral symmetry, the dihedral group $Di_{N}= \mathbb{Z}\diagup 2 \mathbb{Z} \ltimes \mu_{N}$ acting with $x \mapsto x^{-1}$ and $x\mapsto \eta x$. In the special case of $N=1,2,4$, and for these only, the group of homographies is bigger than the dihedral group, due to particular symmetries of the points $\mu_{N}\cup \lbrace 0, \infty\rbrace$ on the Riemann sphere. Let precise these cases:
\begin{itemize}
\item[For $N=1:$] The homography group is the anharmonic group generated by $z\mapsto\frac{1}{z}$ and $z \mapsto 1-z$, and corresponds to the permutation group $\mathfrak{S}_{3}$. Precisely, projective transformations of $\mathbb{P}^{1}\diagdown \lbrace 0, 1, \infty\rbrace$ are:
$$\begin{array}{lll}\label{homography1}
\phi_{\tau}: & t \mapsto 1-t : &  \left\lbrace \begin{array}{l} 
(0,1,\infty)\mapsto (1,0,\infty)\\
(\omega_{0},\omega_{1}, \omega_{\star}, \omega_{\sharp}) \mapsto (\omega_{1},\omega_{0}, -\omega_{\star}, \omega_{0}-\omega_{\star}).
\end{array} \right. \\
\phi_{c}: &  t \mapsto \frac{1}{1-t}  :&  \left\lbrace \begin{array}{l} 
0\mapsto 1 \mapsto \infty \mapsto 0\\
(\omega_{0},\omega_{1}, \omega_{\star}, \omega_{\sharp}) \mapsto (\omega_{\star},-\omega_{0}, -\omega_{1}, -\omega_{0}-\omega_{1})
\end{array} \right. \\
\phi_{\tau c} : &  t \mapsto \frac{t}{t-1} :& \left\lbrace \begin{array}{l}
 (0,1,\infty)\mapsto (0,\infty,1)\\
 (\omega_{0},\omega_{1}, \omega_{\star}) \mapsto (-\omega_{\star},-\omega_{1}, -\omega_{0})
\end{array} \right.\\
\phi_{c\tau}: &  t \mapsto \frac{1}{t} : & \left\lbrace \begin{array}{l} 
(0,1,\infty)\mapsto (\infty,1,0)\\
(\omega_{0},\omega_{1}, \omega_{\star},\omega_{\sharp}) \mapsto (-\omega_{0},\omega_{\star}, \omega_{1}, \omega_{\sharp})
\end{array} \right. \\
 \phi_{c^{2}}: &  t \mapsto \frac{t-1}{t}  : & \left\lbrace \begin{array}{l} 
0\mapsto \infty \mapsto 1 \mapsto 0\\
(\omega_{0},\omega_{1}, \omega_{\star}) \mapsto (-\omega_{1},-\omega_{\star}, \omega_{0})
\end{array} \right. \\
\end{array}$$
Remark that hexagon relation corresponds to a cycle $c$ whereas the reflection relation corresponds to a transposition $\tau$, and :
$$\mathfrak{S}_{3}= \langle c, \tau \mid c^{3}=id, \tau^{2},c\tau c =\tau \rangle= \lbrace 1, c, c^{2}, \tau, \tau c, c\tau\rbrace.$$
\item[For $N=2:$] Here, $(0,\infty, 1,-1)$ has a cross ratio $-1$ (harmonic conjugates) and there are $8$ permutations of $(0,\infty, 1,-1)$ preserving its cross ratio. The homography group corresponds indeed to the group of automorphisms of a square with consecutive vertices $(0, 1, \infty, -1)$, i.e. the dihedral group of degree four $Di_{4}$ defined by the presentation $\langle \sigma, \tau \mid \sigma^{4}= \tau^{2}=id, \sigma\tau \sigma= \tau \rangle$:
$$\begin{array}{lll}\label{homography2}
\phi_{\tau}: & t \mapsto \frac{1}{t} :& \left\lbrace \begin{array}{l} 
\pm 1\mapsto \pm 1  \quad 0 \leftrightarrow \infty\\
(\omega_{0},\omega_{1}, \omega_{\star},\omega_{-1}, \omega_{-\star}, \omega_{\pm\sharp}) \mapsto (-\omega_{0},\omega_{\star}, \omega_{1},\omega_{-\star},\omega_{-1}, \omega_{\pm\sharp})
\end{array} \right.\\
\\
 \phi_{\sigma}: &  t \mapsto \frac{1+t}{1-t}  : & \left\lbrace \begin{array}{l} 
-1\mapsto 0\mapsto 1\mapsto \infty\mapsto -1\\
(\omega_{0},\omega_{1},\omega_{\star},\omega_{-1}, \omega_{-\star}) \mapsto (\omega_{-1}- \omega_{1}, -\omega_{-1}, - \omega_{1}, - \omega_{-\star}, - \omega_{\star})\\
(\omega_{\sharp}, \omega_{-\sharp})  \mapsto  (-\omega_{1}-\omega_{-1}, -\omega_{\star}-\omega_{-\star})
\end{array} \right. \\
\\
 \phi_{\sigma^{2}\tau}: &  t \mapsto -t:& \left\lbrace \begin{array}{l} 
-1 \leftrightarrow 1 \\
(\omega_{0},\omega_{ 1}, \omega_{-1}, \omega_{ \pm \ast}, \omega_{\pm \sharp}) \mapsto (\omega_{0},\omega_{-1}, \omega_{1},\omega_{\mp \ast}, \omega_{\mp \sharp})
\end{array} \right. \\
\\
\phi_{\sigma^{2}}: & t \mapsto \frac{-1}{t} : &  \left\lbrace \begin{array}{l}
0 \leftrightarrow \infty \quad -1 \leftrightarrow 1\\
(\omega_{0},\omega_{1},\omega_{\star},\omega_{-1}, \omega_{-\star}, \omega_{\pm \sharp}) \mapsto (-\omega_{0}, \omega_{-\star}, - \omega_{-1}, \omega_{\star}, \omega_{1}, \omega_{\mp\sharp})
\end{array} \right.\\
\\
\phi_{\sigma^{-1}}: & t \mapsto \frac{t-1}{1+t} : & \left\lbrace \begin{array}{l}
0 \mapsto -1 \mapsto \infty \mapsto 1 \mapsto 0 \\
(\omega_{0}, \omega_{1}, \omega_{-1}, \omega_{\star}, \omega_{-\star}) \mapsto (\omega_{-1}-\omega_{1}, - \omega_{\star}, -\omega_{1}, -\omega_{-\star}, -\omega_{-1}) \\
(\omega_{\sharp}, \omega_{-\sharp})  \mapsto  ( -\omega_{\star}-\omega_{-\star}, -\omega_{1}-\omega_{-1})
\end{array} \right.\\
\\
\phi_{\tau\sigma}: &  t \mapsto \frac{1-t}{1+t} : & \left\lbrace \begin{array}{l}
-1 \leftrightarrow \infty \quad 0 \leftrightarrow 1 \\
(\omega_{0},\omega_{1},\omega_{\star},\omega_{-1}, \omega_{-\star}) \mapsto (\omega_{1}-\omega_{-1},-\omega_{-\star},-\omega_{\star},-\omega_{-1}, -\omega_{1}) \\
(\omega_{\sharp}, \omega_{-\sharp})  \mapsto  ( -\omega_{\star}-\omega_{-\star}, -\omega_{1}-\omega_{-1})
\end{array} \right.\\
\\
\phi_{\sigma \tau}: & t \mapsto \frac{1+t}{t-1} : & \left\lbrace \begin{array}{l}
-1 \leftrightarrow 0  \quad 1 \leftrightarrow \infty\\
(\omega_{0},\omega_{1},\omega_{\star},\omega_{-1}, \omega_{-\star}) \mapsto  (\omega_{1}-\omega_{-1},-\omega_{1},-\omega_{-1},-\omega_{\star}, -\omega_{-\star}) \\
(\omega_{\sharp}, \omega_{-\sharp})  \mapsto  ( -\omega_{1}-\omega_{-1}, -\omega_{\star}-\omega_{-\star})
\end{array} \right.
\end{array} $$
Remark that the octagon relation ($\ref{eq:octagon}$) comes from the cycle $\sigma$ of order $4$; the other permutations above could also leads to relations.
\end{itemize}

\section{From the linearized octagon relation}
The identities in the coalgebra $\mathcal{L}$ obtained from the linearized octagon relation $\ref{eq:octagonlin}$:

\begin{lemm}\label{lemmlor}
In the coalgebra $\mathcal{L}$, $n_{i}\in\mathbb{Z}^{\ast}$:\footnote{Here, $\mlq + \mrq$ still denotes the operation where absolute values are summed and signs multiplied.}
\begin{itemize}
\item[$(i)$] $\zeta^{\star\star, \mathfrak{l}}(n_{0},\cdots, n_{p})= (-1)^{w+1} \zeta^{\star\star,\mathfrak{l}}(n_{p},\cdots, n_{0})$.
\item[$(ii)$] $\zeta^{\mathfrak{l}}(n_{0},\cdots, n_{p})+(-1)^{w+p} \zeta^{\star\star, \mathfrak{l}}(n_{0},\cdots, n_{p})+(-1)^{p} \zeta^{\star\star, \mathfrak{l}}_{\mid n_{p}\mid}(n_{p-1},\cdots,n_{1},n_{0})=0$.
\item[$(iii)$] 
$$\hspace*{-1cm} \zeta^{\mathfrak{l}}_{n_{0}-1}(n_{1},\cdots, n_{p})-  \zeta^{\mathfrak{l}}_{n_{0}}(n_{1},\cdots,n_{p-1}, n_{p}\mlq + \mrq 1 )=(-1)^{w} \left[ \zeta^{\star\star,\mathfrak{l}}_{n_{0}-1}(n_{1},\cdots, n_{p})-  \zeta^{\star\star,\mathfrak{l}}_{n_{0}}(n_{1},\cdots,n_{p-1}, n_{p}\mlq + \mrq 1)\right].$$
\end{itemize}
\end{lemm}
\begin{proof}
The sign of $n_{i}$ is denoted $\epsilon_{i}$ as usual. First, we remark that, with $\eta_{i}=\pm 1$, $n_{i}=\epsilon_{i} (a_{i}+1)$, and $\epsilon_{i}= \eta_{i}\eta_{i+1}$:
$$\hspace*{-0.5cm}\begin{array}{ll}
\Phi^{\mathfrak{m}}(e_{\infty}, e_{-1},e_{1}) & = \sum  I^{\mathfrak{m}} \left(0;  (-\omega_{0})^{a_{0}} (-\omega_{-\eta_{1}\star})  (-\omega_{0})^{a_{1}} \cdots (-\omega_{-\eta_{p}\star})  (-\omega_{0})^{a_{p}} ;1 \right)  e_{0}^{a_{0}}e_{\eta_{1}} e_{0}^{a_{1}} \cdots e_{\eta_{p}} e_{0}^{a_{p}}\\
& \\
& = \sum  (-1)^{n+p}\zeta^{\star\star,\mathfrak{m}}_{n_{0}-1} \left( n_{1}, \cdots, n_{p-1}, -n_{p}\right)  e_{0}^{a_{0}}e_{\eta_{1}} e_{0}^{a_{1}} \cdots e_{\eta_{p}} e_{0}^{a_{p}}. \\
\end{array}$$
Similarly, with $ \mu_{i}\mathrel{\mathop:}= \left\lbrace \begin{array}{ll}
\star & \texttt{if } \eta_{i}=1\\
1 & \texttt{if } \eta_{i}=-1
\end{array}  \right. $, applying $\phi_{\tau\sigma}$ to get the second line:
$$\hspace*{-0.5cm}\begin{array}{ll}
\Phi^{\mathfrak{m}}(e_{-1}, e_{0},e_{\infty}) & = \sum I^{\mathfrak{m}} \left(0;  (\omega_{1}-\omega_{-1})^{a_{0}} \omega_{\mu_{1}}  (\omega_{1}-\omega_{-1})^{a_{1}} \cdots \omega_{\mu_{p}} (\omega_{1}-\omega_{-1})^{a_{p}} ;1 \right)  e_{0}^{a_{0}}e_{\eta_{1}} e_{0}^{a_{1}} \cdots e_{\eta_{p}} e_{0}^{a_{p}}\\
& \\
\Phi^{\mathfrak{l}}(e_{-1}, e_{0},e_{\infty}) & = \sum (-1)^{p} I^{\mathfrak{m}} \left(0;  0^{a_{0}} \omega_{-\eta_{1}}  0^{a_{1}} \cdots \omega_{-\eta_{p}} 0^{a_{p}} ;1 \right)  e_{0}^{a_{0}}e_{\eta_{1}} e_{0}^{a_{1}} \cdots e_{\eta_{p}} e_{0}^{a_{p}}\\
& \\
& = \sum  \zeta^{\mathfrak{m}}_{n_{0}-1} \left( n_{1}, \cdots, n_{p-1}, -n_{p}\right)  e_{0}^{a_{0}}e_{\eta_{1}} e_{0}^{a_{1}} \cdots e_{\eta_{p}} e_{0}^{a_{p}}. \\
\end{array}$$
Lastly, still using  $\phi_{\tau\sigma}$, with here $\mu_{i}\mathrel{\mathop:}= \left\lbrace \begin{array}{ll}
\star & \texttt{if } \eta_{i}=1\\
1 & \texttt{if } \eta_{i}=-1
\end{array}  \right. $:
$$\hspace*{-0.5cm}\begin{array}{ll}
\Phi^{\mathfrak{m}}(e_{1}, e_{\infty},e_{0}) & = \sum  I^{\mathfrak{m}} \left(0;  (\omega_{-1}-\omega_{1})^{a_{0}} \omega_{\mu_{1}}  (\omega_{-1}-\omega_{1})^{a_{1}} \cdots \omega_{\mu_{p}} (\omega_{-1}-\omega_{1})^{a_{p}} ;1 \right)  e_{0}^{a_{0}}e_{\eta_{1}} e_{0}^{a_{1}} \cdots e_{\eta_{p}} e_{0}^{a_{p}}\\
& \\
\Phi^{\mathfrak{l}}(e_{1}, e_{\infty},e_{0}) & = \sum  (-1)^{w+1} I^{\mathfrak{m}} \left(0;  0^{a_{0}} \omega_{\eta_{1}\star}  0^{a_{1}} \cdots \omega_{\eta_{p}\star} 0^{a_{p}} ;1 \right)  e_{0}^{a_{0}}e_{\eta_{1}} e_{0}^{a_{1}} \cdots e_{\eta_{p}} e_{0}^{a_{p}}\\
& \\
& = \sum  (-1)^{n+p+1}\zeta^{\star\star,\mathfrak{m}}_{n_{0}-1} \left( n_{1}, \cdots, n_{p-1}, n_{p}\right)  e_{0}^{a_{0}}e_{\eta_{1}} e_{0}^{a_{1}} \cdots e_{\eta_{p}} e_{0}^{a_{p}}. \\
\end{array}$$

\begin{itemize}
\item[$(i)$] 
This case is the one used in Theorem $\ref{hybrid}$. This identity is equivalent to, in terms of iterated integrals, for $X$ any sequence of $\left\lbrace  0, \pm 1 \right\rbrace $ or of $\left\lbrace  0, \pm \star \right\rbrace $:
$$\left\lbrace  \begin{array}{llll}
I^{\mathfrak{l}}(0;0^{k}, \star, X ;1) & = & I^{\mathfrak{l}}(0; X, \star, 0^{k}; 1)  & \text{ if } \prod_{i=0}^{p} \epsilon_{i}=1 \Leftrightarrow \eta_{0}=1\\
I^{\mathfrak{l}}(0;0^{k}, -\star, X ;1) & = & I^{\mathfrak{l}}(0; -X, -\star, 0^{k}; 1)  & \text{ if } \prod_{i=0}^{p} \epsilon_{i}=-1 \Leftrightarrow \eta_{0}=-1\\
\end{array} \right. $$
The first case is deduced from $\ref{eq:octagonlin}$ when looking at the coefficient of a word beginning and ending by $e_{-1}$ (or beginning and ending by $e_{1}$), whereas the second case is obtained from the coefficient of a word beginning by $e_{-1}$ and ending by $e_{1}$, or beginning by $e_{1}$ and ending by $e_{-1}$.

\item[$(ii)$] Let split into two cases, according to the sign of $\prod \epsilon_{i}$:
\begin{itemize}
\item[$\cdot$]  In $\ref{eq:octagonlin}$, when looking at the coefficient of a word beginning by $e_{1}$ and ending by $e_{0}$, only these three terms contribute:
$$ \Phi^{\mathfrak{l}}(e_{-1}, e_{0},e_{\infty})e_{0}- \Phi^{\mathfrak{l}}(e_{\infty}, e_{-1},e_{1})e_{0}- e_{1} \Phi^{\mathfrak{l}}(e_{1}, e_{\infty},e_{0}) .$$
Moreover, the coefficient of $e_{-1} e_{0}^{a_{0}} e_{\eta_{1}} \cdots  e_{\eta_{p}} e_{0}^{a_{p}+1}$ is, using the expressions above for $\Phi^{\mathfrak{l}}(\cdot)$:
\begin{multline}\nonumber
 (-1)^{p} I^{\mathfrak{l}}(0; -1, -X; 1)+ (-1)^{w+1} I^{\mathfrak{l}}(0; -\star, -X_{\star}; 1)+ (-1)^{w}I^{\mathfrak{l}}(0; X_{\star}, 0; 1)=0, \\
\text{where }\begin{array}{l}
X:= \omega_{0}^{a_{0}} \omega_{\eta_{1}} \cdots  \omega_{\eta_{p}} \omega_{0}^{a_{p}}\\
X_{\star}:= \omega_{0}^{a_{0}} \omega_{\eta_{1}\star} \cdots  \omega_{\eta_{p}\star} \omega_{0}^{a_{p}}
\end{array}.
\end{multline}
In terms of motivic Euler sums, it is, with $\prod \epsilon_{i}=1$:
$$ \zeta^{\mathfrak{l}} (n_{0},\cdots, -n_{p}) +(-1)^{w+p} \zeta^{\star\star\mathfrak{l}}(n_{0},\cdots, -n_{p})+(-1)^{w+p} \zeta^{\star\star\mathfrak{l}}_{n_{0}-1}(n_{1},\cdots,n_{p-1}, n_{p} \mlq + \mrq 1)=0.$$
Changing $n_{p}$ into $-n_{p}$, and applying Antipode $\shuffle$ to the last term, it gives, with now $\prod \epsilon_{i}=-1$: 
$$ \zeta^{\mathfrak{l}} (n_{0},\cdots, n_{p}) +(-1)^{w+p} \zeta^{\star\star\mathfrak{l}}(n_{0},\cdots, n_{p})+(-1)^{p} \zeta^{\star\star\mathfrak{l}}_{n_{p}}(n_{p-1},\cdots,n_{1},n_{0})=0.$$ 
\item[$\cdot$] Similarly, for the coefficient of a word beginning by $e_{-1}$ and ending by $e_{0}$, only these three terms contribute:
$$ \Phi^{\mathfrak{l}}(e_{-1}, e_{0},e_{\infty})e_{0}- \Phi^{\mathfrak{l}}(e_{\infty}, e_{-1},e_{1})e_{0}+ e_{-1} \Phi^{\mathfrak{l}}(e_{\infty}, e_{-1},e_{1}) .$$
Similarly than above, it leads to the identity, with $\prod \epsilon_{i}=-1$:
$$ \zeta^{\mathfrak{l}} (n_{0},\cdots, -n_{p}) +(-1)^{w+p} \zeta^{\star\star\mathfrak{l}}(n_{0},\cdots, -n_{p})+(-1)^{w+p} \zeta^{\star\star\mathfrak{l}}_{n_{0}-1}(n_{1},\cdots,n_{p-1},- (n_{p} \mlq + \mrq 1) )=0.$$
Changing $n_{p}$ into $-n_{p}$, and applying Antipode $\shuffle$ to the last term, it gives, with now $\prod \epsilon_{i}=1$: 
$$ \zeta^{\mathfrak{l}} (n_{0},\cdots, n_{p}) +(-1)^{w+p+1} \zeta^{\star\star\mathfrak{l}}(n_{0},\cdots, n_{p})+(-1)^{p} \zeta^{\star\star\mathfrak{l}}_{n_{p}}(n_{p-1},\cdots,n_{1},n_{0})=0.$$ 
\end{itemize}

\item[$(iii)$] When looking at the coefficient of a word beginning by $e_{0}$ and ending by $e_{0}$ in $\ref{eq:octagonlin}$, only these three terms contribute:
$$ -e_{0} \Phi^{\mathfrak{l}}(e_{-1}, e_{0},e_{\infty})+ \Phi^{\mathfrak{l}}(e_{-1}, e_{0},e_{\infty})e_{0}+ e_{0} \Phi^{\mathfrak{l}}(e_{\infty}, e_{-1},e_{1})- \Phi^{\mathfrak{l}}(e_{\infty}, e_{-1},e_{1})e_{0}.$$
If we identify the coefficient of the word $ e_{0}^{a_{0}+1} e_{-\eta_{1}} \cdots  e_{-\eta_{p}} e_{0}^{a_{p}+1}$, it leads straight to the identity $(iii)$.
\end{itemize}
\textsc{Remark}: Looking at the coefficient of words beginning by $e_{0}$ and ending by $e_{1}$ or $e_{-1}$ in $\ref{eq:octagonlin}$ would lead to the same identity than the second case.
\end{proof}

\section{Some unramified ES up to depth $5$.}

The proof relies upon the criterion $\ref{criterehonoraire}$, which enables us to construct infinite families of unramified Euler sums with parity patterns by iteration on the depth, up to depth $5$.\\
\\
\texttt{Notations:} The occurrences of the symbols $E$ or $O$ can denote arbitrary even or odd integers, whereas every repeated occurrence of symbols $E_{i}$  (respectively $O_{i}$) denotes the same positive even (resp. odd) integer. The bracket $\left\{\cdot, \ldots, \cdot \right\}$ means that every permutation of the listed elements is allowed.
\begin{theo}
The following motivic Euler sums are unramified, i.e. motivic MZV:\footnote{Beware, here, $\overline{O}$ and $\overline{n}$ must be different from $\overline{1}$, whereas $O$ and $n$ may be 1. There is no $\overline{1}$ allowed in these terms if not explicitly written.} \\
\\
\hspace*{-0.5cm} \begin{tabular}{| c | l | l | }
           \hline  
     &  \textsc{Even Weight} & \textsc{Odd Weight}  \\       \hline
\textsc{Depth } 1  &  \text{ All } &  \text{ All }\footnotemark[1] \\    
         \hline    
\textsc{Depth } 2   & $\zeta^{\mathfrak{m}}(\overline{O},\overline{O}),  \zeta^{\mathfrak{m}}(\overline{E},\overline{E})$ & \text{ All }   \\        
         \hline              
\multirow{2}{*}{ \textsc{Depth } 3 } & $\zeta^{\mathfrak{m}}(\left\{E,\overline{O},\overline{O}\right\}), \zeta^{\mathfrak{m}}(O,\overline{E},\overline{O}), \zeta^{\mathfrak{m}}(\overline{O},\overline{E}, O)$  &  $\zeta^{\mathfrak{m}}(\left\{\overline{E},\overline{E},O\right\}), \zeta^{\mathfrak{m}}(\overline{E},\overline{O},E), \zeta^{\mathfrak{m}}(E,\overline{O},\overline{E})$ \\  
 & $ \zeta^{\mathfrak{m}}(\overline{O_{1}}, \overline{E},\overline{O_{1}}), \zeta^{\mathfrak{m}}(O_{1}, \overline{E},O_{1}), \zeta^{\mathfrak{m}}(\overline{E_{1}}, \overline{E},\overline{E_{1}}) .$ & \\
         \hline              
\multirow{2}{*}{ \textsc{Depth } 4 }   &  $\zeta^{\mathfrak{m}}(E,\overline{O},\overline{O},E),\zeta^{\mathfrak{m}}(O,\overline{E},\overline{O},E), $ &   \\   
& $\zeta^{\mathfrak{m}}(O,\overline{E},\overline{E},O), \zeta^{\mathfrak{m}}(E,\overline{O},\overline{E},O)$ & \\
         \hline              
 \textsc{Depth } 5     &   & $\zeta^{\mathfrak{m}}(O_{1}, \overline{E_{1}},O_{1},\overline{E_{1}}, O_{1}).$  \\     
    \hline
  \end{tabular}
Similarly for these linear combinations, in depth $2$ or $3$:
  $$\zeta^{\mathfrak{m}}(n_{1},\overline{n_{2}}) +  \zeta^{\mathfrak{m}}(\overline{n_{2}},n_{1}) , \zeta^{\mathfrak{m}}(n_{1},\overline{n_{2}}) +  \zeta^{\mathfrak{m}}(\overline{n_{1}},n_{2}), \zeta^{\mathfrak{m}}(n_{1},\overline{n_{2}}) -  \zeta^{\mathfrak{m}}(n_{2}, \overline{n_{1}}) .$$
		$$(2^{n_{1}}-1) \zeta^{\mathfrak{m}}(n_{1},\overline{1}) +  (2^{n_{1}-1}-1) \zeta^{\mathfrak{m}}(\overline{1},n_{1}).$$
				$$ \zeta^{\mathfrak{m}}(n_{1},n_{2},\overline{n_{3}}) + (-1)^{n_{1}-1}  \zeta^{\mathfrak{m}}(\overline{n_{3}},n_{2},n_{1}) \text{ with } n_{2}+n_{3} \text{ odd }.$$
\end{theo}
\texttt{Examples}:
These motivic Euler sums are motivic multiple zeta values: 
$$\zeta^{\mathfrak{m}}(\overline{25}, 14,\overline{17}),\zeta^{\mathfrak{m}}(17, \overline{14},17), \zeta^{\mathfrak{m}}(\overline{24}, \overline{14},\overline{24}), \zeta^{\mathfrak{m}}(6, \overline{23}, \overline{17}, 10) , \zeta^{\mathfrak{m}}(13, \overline{24}, 13,\overline{24}, 13).$$
\textsc{Remarks:}
\begin{itemize}
\item[$\cdot$] This result for motivic ES implies the analogue statement for ES.
\item[$\cdot$] Notice that for each honorary MZV above, the reverse sum is honorary too, which was not obvious a priori, since the condition $\textsc{c}1$ below is not symmetric. 
\end{itemize}
\begin{proof} The proof amounts to the straight-forward control that $D_{1}(\cdot)=0$ (immediate) and that all the elements appearing in the right side of $D_{2r+1}$ are unramified, by recursion on depth: here, these elements satisfy the sufficient criteria given below. Let only point out a few things, referring to the expression $(\ref{eq:derhonorary})$:
\begin{description}
\item[\texttt{Terms} $\textsc{c}$:] The symmetry condition $(\textsc{c}4)$, obviously true for these single elements above, get rid of these terms. For the few linear combinations of MES given, the cuts of type (\textsc{c}) get simplified together.
\item[\texttt{Terms} $\textsc{a,b}$:] Checking that the right sides are unramified is straightforward by depth-recursion hypothesis, since only the (previously proven) unramified elements of lower depth emerge. For example, the possible right sides (not canceled by a symmetric cut and up to reversal symmetry) are listed below, for some elements from depth 3. \\
 \hspace*{-1cm} \begin{tabular}{| c | l | l | }
           \hline  
     &  Terms \textsc{a0} & Terms  \textsc{a,b}  \\       \hline
  $(O,\overline{E},\overline{E}) $  &   & $(\overline{E},\overline{E})$ ,$(O,O)$  \\   
  $(\overline{E},O,\overline{E}) $  &  / &   $(\overline{E},\overline{E})$ \\   
 $(E,\overline{O},\overline{E}) $  &  / &  $(\overline{E},\overline{E}), (E,E)$ \\  
         \hline              
    $(E,\overline{O},\overline{O},E) $  &  $(\overline{O},E)$ & $(\overline{E},\overline{O},E), (E,O,E),(E,\overline{O},\overline{E}), (E,O),(O,E)$ \\   
   $(O,\overline{E},\overline{O},E) $ &  $(\overline{E},\overline{O},E),(\overline{O},E)$ & $(\overline{E},\overline{O},E),(O,E,E),(O,\overline{E},\overline{E}), (O,E)$ \\
  $(O,\overline{E},\overline{E},O) $&   $(\overline{E},\overline{E},O) ,(\overline{E},O)$ & $(\overline{E},\overline{E},O), (O,O,O), (O,\overline{E},\overline{E}), (O,E), (E,O)$ \\
    $(\overline{E}, O_{1},\overline{E},O_{1}) $& $(\overline{E},O)$ & $(\overline{E},\overline{E},O) , (\overline{E},O,\overline{E}) ,(\overline{E},\overline{O}), (E,O)$  \\
$(\overline{E_{1}}, \overline{E_{2}},\overline{E_{1}}, \overline{E_{2}}) $& / & $(O,\overline{E},\overline{E}) ,(\overline{E},O,\overline{E}) ,(\overline{E},\overline{E},O) ,(\overline{E},\overline{O}),(\overline{O},\overline{E})$ \\
         \hline    
    $(O_{1}, \overline{E_{1}},O_{1},\overline{E_{1}}, O_{1})$ & $(\overline{E_{1}},O_{1},\overline{E_{1}}, O_{1}),$ & $(\overline{E},O_{1},\overline{E}, O_{1}), (O_{1}, \overline{E},\overline{E}, O_{1}), (O_{1}, \overline{E},O_{1},\overline{E}),$ \\
    & $(O_{1},\overline{E}, O_{1})$ & $(\overline{O},\overline{E},O),(O,\overline{E},\overline{O}), (O,O)$\\
    \hline
  \end{tabular}
\end{description}
It refers to the expression of the derivations $D_{2r+1}$ (Lemma $\ref{drz}$):
\begin{multline}\label{eq:derhonorary}
 D_{2r+1} \left(\zeta^{\mathfrak{m}} \left(n_{1}, \ldots , n_{p} \right)\right) = \textsc{(a0) }  \delta_{2r+1 = \sum_{k=1}^{i} \mid n_{k} \mid} \zeta^{\mathfrak{l}} (n_{1}, \ldots , n_{i}) \otimes \zeta^{\mathfrak{m}} (n_{i+1},\cdots n_{p}) \\
\textsc{(a,b) }   \sum_{1\leq i < j \leq p \atop 2r+1=\sum_{k=i}^{j} \mid n_{k}\mid  - y } \left\lbrace  \begin{array}{l}
  -\delta_{2\leq y \leq \mid n_{j}\mid } \zeta_{\mid n_{j}\mid -y}^{\mathfrak{l}} (n_{j-1}, \ldots ,n_{i+1}, n_{i}) \\
  +\delta_{2\leq y \leq \mid n_{i}\mid} \zeta_{\mid n_{i}\mid -y}^{\mathfrak{l}} (n_{i+1},  \cdots ,n_{j-1}, n_{j})
 \end{array} \right.  \otimes \zeta^{\mathfrak{m}} (n_{1}, \ldots, n_{i-1},\prod_{k=i}^{j}\epsilon_{k} \cdot y,n_{j+1},\cdots n_{p}). \\
\textsc{(c)  } + \sum_{1\leq i < j \leq p\atop {2r+2=\sum_{k=i}^{j} \mid n_{k}\mid} } \delta_{ \prod_{k=i}^{j} \epsilon_{k} \neq 1}  \left\lbrace  \begin{array}{l} 
+  \zeta_{\mid n_{i}\mid -1}^{\mathfrak{l}} (n_{i+1},  \cdots ,n_{j-1}, n_{j}) \\
- \zeta_{\mid n_{j}\mid -1}^{\mathfrak{l}} (n_{j-1},  \cdots ,n_{i+1}, n_{i})
 \end{array} \right. \otimes \zeta^{\mathfrak{m}} (n_{1}, \ldots, n_{i-1},\overline{1},n_{j+1},\cdots n_{p}).
\end{multline}
\end{proof}

\paragraph{Sufficient condition. }\label{sufficientcondition}
Let $\mathfrak{Z}= \zeta^{\mathfrak{m}}(n_{1}, \ldots, n_{p})$ a motivic Euler sum. These four conditions are \textit{sufficient} for $\mathfrak{Z}$ to be unramified:
\begin{description}
	\item [\textsc{c}1]: No $\overline{1}$ in $\mathfrak{Z}$.
	\item [\textsc{c}2]: For each $(n_{1}, \ldots, n_{i})$ of odd weight, the MES $\zeta^{\mathfrak{m}}(n_{i+1}, \ldots, n_{p})$ is a MMZV.
	\item [\textsc{c}3]: If a cut removes an odd-weight part (such that there is no symmetric cut possible), the remaining MES (right side in terms \textsc{a,b}), is a MMZV.
	\item [\textsc{c}4]: Each sub-sequence $(n_{i}, \ldots, n_{j})$ of even weight such that $\prod_{k=i}^{j} \epsilon_{k} \neq 1$ is symmetric.
\end{description}
\begin{proof}
The condition $\textsc{c}1$ implies that $D_{1}(\mathfrak{Z})=0$; conditions $\textsc{c}2$, resp. $\textsc{c}3$ take care that the right side of terms \textsc{a0}, resp. \textsc{a,b} are unramified, while the condition $\textsc{c}4$ cancels the (disturbing) terms \textsc{c}: indeed, a single ES with an $\overline{1}$ can not be unramified.\\
Note that a MES $\mathfrak{Z}$ of depth $2$, weight $n$ is unramified if and only if $ \left\lbrace \begin{array}{l}
D_{1}(\mathfrak{Z})=0\\
D_{n-1}(\mathfrak{Z})=0
\end{array}\right.$. 
\end{proof}
\noindent
\texttt{Nota Bene:} This criterion is not \textit{necessary}: it does not cover the unramified $\mathbb{Q}$-linear combinations of motivic ES, such as those presented in section $4$, neither some isolated (\textit{symmetric enough}) examples where the unramified terms appearing in $D_{2r+1}$ could cancel between them. However, it embrace general families of single Euler sums which are unramified.\\
Nevertheless, 
\begin{framed}
\emph{If we \textit{comply with these conditions}, the \textit{only} general families of single MES obtained are the one listed in Theorem $D.1$.}
\end{framed}
\begin{proof}[Sketch of the proof]
Notice first that the condition \textsc{c}$4$ implies in particular that there are no consecutive sequences of the type (since it would create type $\textsc{c}$ terms):
$$\textsc{ Seq. not allowed :  }  O\overline{O}, \overline{O}O, \overline{E}E, E\overline{E}.$$
It implies, from depth $3$, that we can't have the sequences (otherwise one of the non allowed depth $2$ sequences above appear in $\textsc{a,b}$ terms):
$$\textsc{ Seq. not allowed : }  \overline{E}\overline{E}\overline{O}, \overline{E}\overline{E}\overline{E}O, \overline{E}\overline{E}O\overline{E}, E\overline{O}\overline{E},EE\overline{O}, \overline{O}EE, \overline{E}OE, \overline{E}\overline{O}\overline{E}, \overline{O}\overline{O}\overline{O}.$$
Going on in this recursive way, carefully, leads to the previous theorem.\\
\texttt{For instance,} let $\mathfrak{Z}$ a MES of even weight, with no $\overline{1}$, and let detail two cases:
\begin{description}
\item[\texttt{Depth} $3$:]  The right side of $D_{2 r+1}$ has odd weight and depth smaller than $2$, hence is always MMZV if there is no $\overline{1}$ by depth $2$ results. It boils down to the condition $\textsc{c}4$: $\mathfrak{Z}$ must be either symmetric (such as $O_{1}E O_{1}$ or $E_{1}EE_{1}$ with possibly one or three overlines) either have exactly two overlines.  Using the analysis above of the allowed sequences in depth $2$ and $3$ for condition $\textsc{c3,4}$ leads to the following:
$$(E,\overline{O},\overline{O}),(\overline{O},\overline{O},E), (O,\overline{E},\overline{O}), (\overline{O},\overline{E}, O), (\overline{O},E,\overline{O}), (\overline{O_{1}}, \overline{E},\overline{O_{1}}), (O_{1}, \overline{E},O_{1}), (\overline{E_{1}}, \overline{E},\overline{E_{1}}) .$$
\item[\texttt{Depth} $4$:] Let $\mathfrak{Z}=\zeta^{\mathfrak{m}}\left( n_{1}, \ldots, n_{4}\right) $, $\epsilon_{i}=\textrm{sign}(n_{i})$. To avoid terms of type $
\textsc{c}$ with a right side of depth $1$: if $\epsilon_{1}\epsilon_{2}\epsilon_{3}\neq 1$, either $n_{1}+n_{2}+n_{3}$ is odd, or $n_{1}=n_{3}$ and $\epsilon_{2}=-1$;  if $\epsilon_{2}\epsilon_{3}\epsilon_{4}\neq 1$, either $n_{2}+n_{3}+n_{4}$ is odd, or $n_{2}=n_{4}$ and $\epsilon_{3}=-1$. The following sequences are then not allowed:
$$ (\overline{E}, O,O,\overline{E}), (\overline{E}, \overline{O},\overline{O},\overline{E}), (\overline{E}, \overline{O},E,O), (\overline{E}, \overline{E},O,O), (O,O,\overline{E}, \overline{E}).$$
\end{description}
\end{proof}

\section{Proof of Theorem $\autoref{lzgt}$}

To prove the identity of conjecture $\autoref{lzg}$ of Theorem $\autoref{lzgt}$, we would need to prove, by recursion, that the coaction is equal on both side, and use the analytic version to conclude. We prove $(i)$ and $(ii)$ successively, in a same recursion on the weight:
\begin{itemize}
\item[$(i)$] Using the formulas of the coactions $D_{r}$ for these families (Lemma $A.2$ and $A.4$), we can gather terms in both sides according to the right side, which leads to three types:
$$ \begin{array}{llll}
(a) &  \zeta^{\star, \mathfrak{m}} (\cdots,\boldsymbol{2}^{a_{i}}, \alpha, \boldsymbol{2}^{\beta}, c_{j+1}, \cdots)  & \longleftrightarrow &   \zeta^{\sharp ,\mathfrak{m}}(B_{0} \cdots, B_{i}, \textcolor{magenta}{1^{\gamma}, B}, 1^{\gamma_{j+1}}, \ldots, B_{p}) \\
(b) &   \zeta^{\star, \mathfrak{m}} (\cdots,\boldsymbol{2}^{a_{i-1}}, c_{i}, \boldsymbol{2}^{\beta}, c_{j+1}, \cdots)    & \longleftrightarrow &   \zeta^{\sharp ,\mathfrak{m}}(B_{0} \cdots, B_{i-1}, 1^{\gamma_{i}}, \textcolor{green}{B}, 1^{\gamma_{j+1}}, \ldots, B_{p})   \\
(c) &   \zeta^{\star, \mathfrak{m}} (\cdots, c_{i}, \boldsymbol{2}^{\beta}, \alpha, \boldsymbol{2}^{a_{j}}, \cdots)    & \longleftrightarrow & \zeta^{\sharp, \mathfrak{m}}(B_{0} \cdots, 1^{\gamma_{i+1}},\textcolor{cyan}{ B, 1^{\gamma}}, B_{j+1}, \ldots, B_{p}) 
\end{array},$$
with $ \gamma=\alpha-3$ and $B=2\beta +3-\delta_{c_{j+1}}$, or $B=2\beta+3 - \delta_{c_{i}}- \delta_{c_{j+1}}$ for $(b)$.\\
The third case, antisymmetric of the first case, may be omitted below. By recursive hypothesis, these right sides are equal and it remains to compare the left sides associated:
\begin{enumerate}
\item On the one hand, by lemma $A.2$, the left side corresponding:
$$ \delta_{3\leq \alpha \leq c_{i+1}-1 \atop 0\leq \beta  a_{j}}  \zeta^{\star, \mathfrak{l}}_{c_{i+1}-\alpha}- (\boldsymbol{2}^{ a_{j}-\beta}, \ldots, \boldsymbol{2}^{a_{i+1}}).$$
On the other hand (Lemma $A.4$), the left side is:
$$-\delta_{2 \leq B \leq B_{j} \atop 0\leq\gamma\leq\gamma_{i+1}-1}\zeta^{\sharp,\mathfrak{l}}(B_{j}-B+1, 1^{\gamma_{j}}, \ldots, 1^{\gamma_{i+1}-\gamma-1}).$$
They are both equal, by $\ref{eq:toolid}$, where $c_{i+1}-\alpha+2$ corresponds to $c_{1} $ and is greater than $3$.\\
\item By lemma $A.2$, the left side corresponding for $\zeta^{\star}$:
$$\hspace*{-1.2cm}\begin{array}{llll}
-& \delta_{c_{i}>3} \zeta^{\star\star, \mathfrak{l}}_{2} (\boldsymbol{2}^{a_{i}}, \ldots, \boldsymbol{2}^{ a_{j}-\beta-1}) & + & \delta_{c_{j}>3} \zeta^{\star\star, \mathfrak{l}}_{2} (\boldsymbol{2}^{a_{j}}, \ldots, \boldsymbol{2}^{ a_{i}-\beta-1})  \\
 - & \delta_{c_{i}=1} \zeta^{\star\star, \mathfrak{l}} (\boldsymbol{2}^{a_{j}-\beta}, \ldots, \boldsymbol{2}^{ a_{i}}) & +&  \delta_{c_{j+1}=1} \zeta^{\star\star, \mathfrak{l}} (\boldsymbol{2}^{a_{i}-\beta}, \ldots, \boldsymbol{2}^{ a_{j}})\\
   + & \delta_{c_{i+1}=1 \atop \beta> a_{i}} \zeta^{\star\star, \mathfrak{l}}_{1} (\boldsymbol{2}^{a_{i}+a_{j}-\beta}, \ldots, \boldsymbol{2}^{ a_{i+1}})  & - & \delta_{c_{j}=1 \atop \beta> a_{j}} \zeta^{\star\star, \mathfrak{l}}_{1} (\boldsymbol{2}^{a_{i}+a_{j}-\beta}, \ldots, \boldsymbol{2}^{ a_{j-1}})\\
   - & \delta_{a_{j}< \beta \leq a_{i}+ a_{j}+1}  \zeta^{\star\star, \mathfrak{l}}_{c_{j}-2} (\boldsymbol{2}^{a_{j-1}}, \ldots, \boldsymbol{2}^{a_{i}+ a_{j}-\beta+1})  &  + & \delta_{a_{i}< \beta \leq a_{j}+a_{i}+1} \zeta^{\star\star, \mathfrak{l}}_{c_{i+1}-2} (\boldsymbol{2}^{a_{i+1}}, \ldots, \boldsymbol{2}^{ a_{i}+ a_{j} -\beta+1}) .
\end{array}$$
It should correspond to (using still lemma $A.4$), with $B_{k}=2a_{k}+3-\delta_{c_{k}}-\delta_{c_{k+1}}$, $\gamma_{k}=c_{k}-3+2\delta_{c_{k}}$ and $B=2\beta+3 - \delta_{c_{i}}- \delta_{c_{j+1}}$:
$$\left( \delta_{B_{i}< B}\zeta^{\sharp\sharp,\mathfrak{l}}_{B_{i}+B_{j}-B}(1^{\gamma_{j}}, \ldots, 1^{\gamma_{i+1}}) - \delta_{B_{j}< B}\zeta^{\sharp\sharp,\mathfrak{l}}_{B_{i}+B_{j}-B}(1^{\gamma_{i+1}}, \ldots, 1^{\gamma_{j}}) \right.  $$
$$\left. + \zeta^{\sharp\sharp,\mathfrak{l}}_{B_{i}-B}(1^{\gamma_{i+1}}, \ldots, B_{j}) - \zeta^{\sharp\sharp,\mathfrak{l}}_{B_{j}-B}(1^{\gamma_{j}}, \ldots, B_{i})\right) .$$
The first line has even depth, while the second line has odd depth, as noticed in Lemma $A.4$.
Let distinguish three cases, and assume $a_{i}<a_{j}$:\footnote{The case $a_{j}<a_{i}$ is anti-symmetric, hence analogue.}
\begin{itemize}

\item[$(a)$]  When $\beta< a_{i}<a_{j}$, we should have:
\begin{equation}\label{eq:ci}
 \zeta^{\sharp\sharp,\mathfrak{l}}_{B_{i}-B}(1^{\gamma_{i+1}}, \ldots, B_{j}) - \zeta^{\sharp\sharp,\mathfrak{l}}_{B_{j}-B}(1^{\gamma_{j}}, \ldots, B_{i})   \text{ equal to:} 
\end{equation} 
$$\begin{array}{llll}
-  \delta_{c_{i} >3} & \zeta^{\star\star, \mathfrak{l}}_{2} (\boldsymbol{2}^{a_{j}-\beta-1}, \ldots, \boldsymbol{2}^{ a_{i}})  & - \delta_{c_{i}=1}  &  \zeta^{\star\star, \mathfrak{l}} (\boldsymbol{2}^{a_{j}-\beta}, \ldots, \boldsymbol{2}^{ a_{i}}) \\
+\delta_{c_{j+1}>3} & \zeta^{\star\star, \mathfrak{l}}_{2} (\boldsymbol{2}^{a_{i}-\beta-1}, \ldots, \boldsymbol{2}^{ a_{j}})  & + \delta_{c_{j+1}=1}  & \zeta^{\star\star, \mathfrak{l}} (\boldsymbol{2}^{a_{i}-\beta}, \ldots, \boldsymbol{2}^{ a_{j}}) 
\end{array}$$
\begin{itemize}
\item[$\cdot$] Let first look at the case where $c_{i}>3, c_{j+1}>3$. Renumbering the indices, using $\textsc{Shift}$ for odd depth for the second line, it is equivalent to, with $\alpha=\beta +1$, $B_{p}=2a_{p}+3, B_{0}=2a_{0}+3$:
$$\begin{array}{llll}
&\zeta^{\star\star, \mathfrak{l}}_{2} (\boldsymbol{2}^{a_{0}-\alpha},c_{1},\cdots,c_{p},\boldsymbol{2}^{a_{p}}) & -& \zeta^{\star\star, \mathfrak{l}}_{2} (\boldsymbol{2}^{a_{0}},c_{1},\cdots,c_{p},\boldsymbol{2}^{a_{p}-\alpha}) \\
\equiv & \zeta^{\sharp\sharp,  \mathfrak{l}} _{B_{0}-B}(1^{\gamma_{1}},\cdots, 1^{\gamma_{p}},B_{p}) & - & \zeta^{\sharp\sharp,  \mathfrak{l}} _{B_{p}-B}(B_{0}, 1^{\gamma_{1}},\cdots, 1^{\gamma_{p}})\\
\equiv & \zeta^{\sharp\sharp,  \mathfrak{l}} _{B_{p}-1}(B_{0}-B+1,1^{\gamma_{1}},\cdots, 1^{\gamma_{p}}) & -& \zeta^{\sharp\sharp,\mathfrak{l}} _{B_{p}-B}(B_{0}, 1^{ \gamma_{1}},\cdots, 1^{\gamma_{p}})\\
\equiv & \zeta^{\sharp,  \mathfrak{l}} _{B_{p}-1}(B_{0}-B+1,1^{\gamma_{1}},\cdots, 1^{\gamma_{p}}) & -& \zeta^{\sharp,\mathfrak{l}} _{B_{p}-B}(B_{0}, 1^{ \gamma_{1}},\cdots, 1^{\gamma_{p}}).
\end{array}$$
This boils down to $(\ref{eq:conjid})$ applied to each $\zeta^{\star\star}_{2}$, since by \textsc{Shift} $(\ref{eq:shift})$ the two terms of the type $\zeta^{\star\star}_{1}$ get simplified.\\
\item[$\cdot$] Let now look at the case where $c_{i}=1, c_{j+1}>3$ \footnote{The case $c_{j+1}=1, c_{i}>3$ being analogue, by symmetry.}; hence $B_{i}=2a_{i}+2-\delta_{c_{i+1}}$, $B=2\beta+2$. In a first hand, we have to consider:
$$ \zeta^{\star\star, \mathfrak{l}}_{2} (\boldsymbol{2}^{a_{i}-\beta-1},c_{i+1},\cdots,c_{j},\boldsymbol{2}^{a_{j}}) - \zeta^{\star\star, \mathfrak{l}} (\boldsymbol{2}^{a_{j}-\beta},c_{j},\cdots,c_{i+1},\boldsymbol{2}^{a_{i}}).$$
By renumbering indices in $\ref{eq:ci}$, the correspondence boils down here to the following $\boldsymbol{\diamond} = \boldsymbol{\Join}$, where $B_{0}=2a_{0}+3-\delta_{c_{1}}$, $B_{i}=2a_{i}+3-\delta_{c_{i}}-\delta_{c_{i+1}}$, $B=2\beta +2$:
$$ (\boldsymbol{\diamond}) \quad  \zeta^{\star\star, \mathfrak{l}}_{2} (\boldsymbol{2}^{a_{0}-\beta},c_{1},\cdots,c_{p},\boldsymbol{2}^{a_{p}}) - \zeta^{\star\star, \mathfrak{l}} (\boldsymbol{2}^{a_{0}+1},c_{1},\cdots,c_{p},\boldsymbol{2}^{a_{p}-\beta})$$
$$(\boldsymbol{\Join}) \quad \zeta^{\sharp\sharp,\mathfrak{l}}_{B_{0}-B+1}(1^{\gamma_{1}}, \ldots,1^{\gamma_{p}}, B_{p}) - \zeta^{\sharp\sharp,\mathfrak{l}}_{B_{p}-B}(1^{\gamma_{p}}, \ldots,1^{\gamma_{1}}, B_{0}+1).$$
Turning in $(\boldsymbol{\diamond})$ the second term into a $\zeta^{\star, \mathfrak{l}}(2, \cdots)+ \zeta^{\star\star, \mathfrak{l}}_{2} (\cdots)$, and applying the identity $(\ref{eq:conjid})$ for both terms $\zeta^{\star\star, \mathfrak{l}}_{2}(\cdots)$ leads to:
$$\hspace*{-0.7cm}(\boldsymbol{\diamond}) \left\lbrace \begin{array}{lll}
 + \zeta^{\star\star, \mathfrak{l}}_{1} (\boldsymbol{2}^{a_{0}-\beta+1},c_{1}-1,\cdots,c_{p},\boldsymbol{2}^{a_{p}}) & - \zeta^{\star\star, \mathfrak{l}}_{1} (\boldsymbol{2}^{a_{0}+1},c_{1}-1,\cdots,c_{p},\boldsymbol{2}^{a_{p}-\beta})& \quad (\boldsymbol{\diamond_{1}})  \\
+ \zeta^{\sharp, \mathfrak{l}}_{B_{0}-B+1} (\boldsymbol{1}^{\gamma_{1}},\cdots,\boldsymbol{1}^{\gamma_{p}},B_{p}) &- \zeta^{\sharp, \mathfrak{l}}_{B_{0}-1} (\boldsymbol{1}^{\gamma_{1}},\cdots,\boldsymbol{1}^{\gamma_{p}},B_{p}-B+2)& \quad(\boldsymbol{\diamond_{2}})  \\
- \zeta^{\star, \mathfrak{l}} (\boldsymbol{2}^{a_{0}+1},c_{1},\cdots,c_{p},\boldsymbol{2}^{a_{p}-\beta}) &  & \quad (\boldsymbol{\diamond_{3}}) \\
\end{array} \right. $$
The first line, $(\boldsymbol{\diamond}_{1}) $ by $\textsc{Shift}$ is zero. We apply $\textsc{Antipode}$ $\ast$ on the terms of the second line, then turn each into a difference $\zeta^{\sharp\sharp}_{n}(m, \cdots)- \zeta^{\sharp\sharp}_{n+m}(\cdots)$; the terms of the type $\zeta^{\sharp\sharp}_{n+m}(\cdots)$, are identical and get simplified:

$$(\boldsymbol{\diamond_{2}}) \quad \begin{array}{lll}
\equiv & \zeta^{\sharp\sharp, \mathfrak{l}}_{B_{0}-B+1} (B_{p},\boldsymbol{1}^{\gamma_{p}},\cdots,\boldsymbol{1}^{\gamma_{1}}) & - \zeta^{\sharp\sharp, \mathfrak{l}}_{B_{0}-B+1+ B_{p}} (\boldsymbol{1}^{\gamma_{p}},\cdots,\boldsymbol{1}^{\gamma_{1}}) \\
& -\zeta^{\sharp\sharp, \mathfrak{l}}_{B_{0}-1} (B_{p}-B+2,\boldsymbol{1}^{\gamma_{p}},\cdots,\boldsymbol{1}^{\gamma_{1}}) & + \zeta^{\sharp\sharp, \mathfrak{l}}_{B_{0}-B+1+ B_{p}} (\boldsymbol{1}^{\gamma_{p}},\cdots,\boldsymbol{1}^{\gamma_{1}}) \\
\equiv & \zeta^{\sharp\sharp, \mathfrak{l}}_{B_{0}-B+1} (B_{p},\boldsymbol{1}^{\gamma_{p}},\cdots,\boldsymbol{1}^{\gamma_{1}}) & - \zeta^{\sharp\sharp, \mathfrak{l}}_{B_{0}-1} (B_{p}-B+2,\boldsymbol{1}^{\gamma_{p}},\cdots,\boldsymbol{1}^{\gamma_{1}}).
\end{array}$$
Furthermore, applying the recursion hypothesis (i), i.e. conjecture $\autoref{lzg}$ on $(\boldsymbol{\diamond}_{3})$, and turn it into a difference of $\zeta^{\sharp\sharp}$:
$$(\boldsymbol{\diamond_{3}})\quad \begin{array}{ll}
& - \zeta^{\star, \mathfrak{l}} (\boldsymbol{2}^{a_{0}+1},c_{1},\cdots,c_{p},\boldsymbol{2}^{a_{p}-\beta})\\
\equiv & - \zeta^{\sharp, \mathfrak{l}} (B_{p}-B+1,\boldsymbol{1}^{\gamma_{p}},\cdots,\boldsymbol{1}^{\gamma_{1}},B_{0})\\
\equiv & - \zeta^{\sharp\sharp, \mathfrak{l}} (B_{p}-B+1,\boldsymbol{1}^{\gamma_{p}},\cdots,\boldsymbol{1}^{\gamma_{1}},B_{0}) + \zeta^{\sharp\sharp, \mathfrak{l}}_{B_{p}-B+1} (\boldsymbol{1}^{\gamma_{p}},\cdots,\boldsymbol{1}^{\gamma_{1}},B_{0})
\end{array}$$
When adding $(\boldsymbol{\diamond_{2}})$ and $(\boldsymbol{\diamond_{3}})$ to get $(\boldsymbol{\diamond})$, the two last terms (odd depth) being simplified by $\textsc{Shift}$, it remains:
$$(\boldsymbol{\diamond}) \quad  \zeta^{\sharp\sharp, \mathfrak{l}}_{B_{0}-B+1} (B_{p},\boldsymbol{1}^{\gamma_{p}},\cdots,\boldsymbol{1}^{\gamma_{1}}) - \zeta^{\sharp\sharp, \mathfrak{l}} (B_{p}-B+1,\boldsymbol{1}^{\gamma_{p}},\cdots,\boldsymbol{1}^{\gamma_{1}},B_{0}). $$
This, applying \textsc{Antipode} $\ast$ to the first term, $\textsc{Cut}$ and $\textsc{Shift}$ to the second, corresponds to $(\boldsymbol{\Join})$.\\
\end{itemize}
\item[$(b)$] When $\beta >a_{j}>a_{i}$, we should have:
$$\begin{array}{lll}
& - \zeta^{\star\star, \mathfrak{l}}_{c_{j}-2} (\boldsymbol{2}^{a_{j-1}}, \ldots, \boldsymbol{2}^{a_{i}+ a_{j}-\beta+1}) & + \zeta^{\star\star, \mathfrak{l}}_{c_{i+1}-2} (\boldsymbol{2}^{a_{i+1}}, \ldots, \boldsymbol{2}^{ a_{i}+ a_{j} -\beta+1})\\
\equiv & + \zeta^{\sharp\sharp,\mathfrak{l}}_{B_{i}+B_{j}-B}(1^{\gamma_{j}}, \ldots, 1^{\gamma_{i+1}}) & -\zeta^{\sharp\sharp,\mathfrak{l}}_{B_{i}+B_{j}-B}(1^{\gamma_{i+1}}, \ldots, 1^{\gamma_{j}}).
\end{array}$$
Using \textsc{Shift} $(\ref{eq:shift})$ for the first line, and renumbering the indices, it is equivalent to, with $c_{1},c_{p} \geq 3$ and $a_{0}>0$:
\begin{equation} \label{eq:corresp3}
\zeta^{\star\star, \mathfrak{l}}_{1} (\boldsymbol{2}^{a_{0}},c_{1}-1,\cdots,c_{p})- \zeta^{\star\star, \mathfrak{l}}_{1} (\boldsymbol{2}^{a_{0}},c_{1},\cdots,c_{p}-1)
\end{equation}
$$ \equiv \zeta^{\sharp\sharp,  \mathfrak{l}} _{B_{0}+2}(1^{\gamma_{1}},\cdots, 1^{\gamma_{p}})-\zeta^{\sharp\sharp,  \mathfrak{l}} _{B_{0}+2}(1^{\gamma_{p}},\cdots, 1^{\gamma_{1}}) \equiv \zeta^{\sharp,  \mathfrak{l}} _{B_{0}+2}(1^{\gamma_{1}}, \ldots, 1^{\gamma_{p}}).$$
The last equality comes from Corollary $\ref{esharprel}$, since depth is even. By $(\ref{eq:corresp})$ applied on each term of the first line
$$\zeta^{\star\star, \mathfrak{l}}_{1} (\boldsymbol{2}^{a_{0}},c_{1}-1,\cdots,c_{p})- \zeta^{\star\star, \mathfrak{l}}_{1} (\boldsymbol{2}^{a_{0}},c_{1},\cdots,c_{p}-1)$$
$$\hspace*{-1.5cm} \equiv \zeta^{\star\star, \mathfrak{l}}_{2} (\boldsymbol{2}^{a_{0}-1},c_{1},\cdots,c_{p})+ \zeta^{\sharp  \mathfrak{l}} _{2a_{0}}(3,1^{\gamma_{p}},\cdots, 1^{\gamma_{1}}) - \zeta^{\star\star, \mathfrak{l}}_{2} (c_{p},\cdots,c_{1},\boldsymbol{2}^{a_{0}-1}) - \zeta^{\sharp\sharp,  \mathfrak{l}}_{2}(2a_{0}+1,1^{\gamma_{p}},\cdots, 1^{\gamma_{1}}).$$
By Antipode $\shuffle$, the $\zeta^{\star\star}$ get simplified, and by the definition of $\zeta^{\sharp\sharp}$, the previous equality is equal to: 
$$\equiv- \zeta^{\sharp  \mathfrak{l}} _{2a_{0}+3}(1^{\gamma_{p}},\cdots, 1^{\gamma_{1}}) + \zeta^{\sharp\sharp  \mathfrak{l}} _{2a_{0}}(3,1^{\gamma_{p}},\cdots, 1^{\gamma_{1}}) + \zeta^{\sharp\sharp  \mathfrak{l}} _{2a_{0}+4}(1^{\gamma_{p}-1},\cdots, 1^{\gamma_{1}})  $$
$$ - \zeta^{\sharp\sharp  \mathfrak{l}} _{2}(2a_{0}+1, 1^{\gamma_{p}},\cdots, 1^{\gamma_{1}}) + \zeta^{\sharp\sharp  \mathfrak{l}} _{2a_{0}+3}(1^{\gamma_{r}},\cdots, 1^{\gamma_{1}}).$$
Then, by \textsc{Shift} $(\ref{eq:shift})$, the second and fourth term get simplified while the third and fifth term get simplified by \textsc{Cut} $(\ref{eq:cut})$. It remains:
$$- \zeta^{\sharp , \mathfrak{l}} _{2a_{0}+3}(1^{\gamma_{p}},\cdots, 1^{\gamma_{1}}), \quad \text{ which leads straight to } \ref{eq:corresp3}.$$
\item[$(c)$] When $a_{i}< \beta <a_{j}$, we should have:
\begin{multline}\nonumber
 - \zeta^{\star\star, \mathfrak{l}}_{2} (\boldsymbol{2}^{a_{i}}, \ldots, \boldsymbol{2}^{a_{j}-\beta-1})  +  \zeta^{\star\star, \mathfrak{l}}_{c_{i+1}-2} (\boldsymbol{2}^{a_{i+1}}, \ldots, \boldsymbol{2}^{a_{i}+ a_{j} -\beta+1})  \\
\equiv \zeta^{\sharp\sharp,\mathfrak{l}}_{B_{i}+B_{j}-B}(1^{\gamma_{j}}, \ldots, 1^{\gamma_{i+1}})-\zeta^{\sharp\sharp,\mathfrak{l}}_{B_{j}-B}(1^{\gamma_{j}}, \ldots, B_{i}). 
\end{multline}
Using resp. \textsc{Antipode} \textsc{Shift} $(\ref{eq:shift})$ for the first line, and re-ordering the indices, it is equivalent to, with $c_{1}\geq 3$, $B_{0}=2a_{0}+1-\delta c_{1}$ here:
\begin{equation} \label{eq:corresp} \zeta^{\star\star, \mathfrak{l}}_{2} (\boldsymbol{2}^{a_{0}-1},c_{1},\cdots,c_{p},\boldsymbol{2}^{a_{p}})- \zeta^{\star\star, \mathfrak{l}}_{1} (\boldsymbol{2}^{a_{0}},c_{1}-1,\cdots,c_{p},\boldsymbol{2}^{a_{p}})  
\end{equation}
$$\equiv \zeta^{\sharp\sharp,  \mathfrak{l}} _{B_{p}-1}(B_{0}, 1^{\gamma_{1}},\cdots, 1^{\gamma_{p}}) - \zeta^{\sharp\sharp,  \mathfrak{l}} _{B_{0}+B_{p}-1}(1^{ \gamma_{p}},\cdots, 1^{\gamma_{1}}) \equiv  \zeta^{\sharp,  \mathfrak{l}} _{B_{0}-1}(1^{\gamma_{1}},\cdots, 1^{\gamma_{p}},B_{p}).$$
This matches with the identity $\ref{eq:conjid}$; the last equality coming from $\textsc{Shift}$ since depth is odd.
\end{itemize}
\item Antisymmetric of the first case.\\
\end{enumerate}

\item[$(ii)$] Let us denote the sequences $\textbf{X}=\boldsymbol{2}^{a_{1}}, \ldots , \boldsymbol{2}^{a_{p}}$ and $\textbf{Y}= \boldsymbol{1}^{\gamma_{1}-1}, B_{1},\cdots,  \boldsymbol{1}^{\gamma_{p}} $.\\
We want to prove that:
\begin{equation} \label{eq:1234567}
\zeta^{\sharp,\mathfrak{l}} (1,\textbf{Y},B_{p})\equiv -\zeta^{\star,\mathfrak{l}}_{1} (c_{1}-1,\textbf{X})
\end{equation}
Relations used are mostly these stated in Section 3. Using the definition of $\zeta^{\star\star}$:
\begin{equation}\label{eq:12345}
\begin{array}{ll}
-\zeta^{\star,\mathfrak{l}}_{1} (c_{1}-1,\textbf{X}) & \equiv -\zeta^{\star\star,\mathfrak{l}}_{1} (c_{1}-1,\textbf{X})+ \zeta^{\star\star,\mathfrak{l}}_{c_{1}} (\textbf{X})\\
& \equiv - \zeta^{\star\star,\mathfrak{l}} (1,c_{1}-1,\textbf{X})+ \zeta^{\star,\mathfrak{l}} (1,c_{1}-1,\textbf{X})+ \zeta^{\star\star,\mathfrak{l}}(c_{1},\textbf{X})- \zeta^{\star,\mathfrak{l}} (c_{1},\textbf{X}) \\
& \equiv \zeta^{\star,\mathfrak{l}} (1,c_{1}-1,\textbf{X})- \zeta^{\star,\mathfrak{l}} (c_{1},\textbf{X})- \zeta^{\star,\mathfrak{l}}(c_{1}-1,\textbf{X},1).
\end{array}
\end{equation}
There, the first and third term in the second line, after applying \textsc{Shift}, have given the last $\zeta^{\star}$ in the last line.\\
Using then Conjecture $\ref{lzg}$, in terms of MMZV$^{\sharp}$, then MMZV$^{\sharp\sharp}$, it gives:
\begin{multline}
\zeta^{\sharp,\mathfrak{l}} (2,\textbf{Y},B_{p}-1)+ \zeta^{\sharp,\mathfrak{l}}  (1,1,\textbf{Y},B_{p}-1)+ \zeta^{\sharp,\mathfrak{l}} (1,\textbf{Y},B_{p}-1,1)\\
 \equiv \zeta^{\sharp\sharp,\mathfrak{l}} (2,\textbf{Y},B_{p}-1)- \zeta^{\sharp\sharp,\mathfrak{l}} _{2}(\textbf{Y},B_{p}-1)+ \zeta^{\sharp\sharp,\mathfrak{l}}  (1,1,\textbf{Y},B_{p}-1)\\
-\zeta^{\sharp\sharp,\mathfrak{l}}_{1} (1,\textbf{Y},B_{p}-1)+ \zeta^{\sharp\sharp,\mathfrak{l}} (1,\textbf{Y},B_{p}-1,1)-\zeta^{\sharp\sharp,\mathfrak{l}}_{1} (\textbf{Y},B_{p}-1,1)   
\end{multline}
First term (odd depth)\footnote{Since weight is odd, we know also depth parity of these terms.} is simplified with the last, by $\textsc{Schift}$. Fifth term (even depth) get simplified by \textsc{Cut} with the fourth term. Hence it remains two terms of even depth:
$$\equiv - \zeta^{\sharp\sharp,\mathfrak{l}} _{2}(\textbf{Y},B_{p}-1)+ \zeta^{\sharp\sharp,\mathfrak{l}}  (1,1,\textbf{Y},B_{p}-1) \equiv - \zeta^{\sharp\sharp,\mathfrak{l}} _{1}(\textbf{Y},B_{p})+ \zeta^{\sharp\sharp,\mathfrak{l}}_{B_{p}-1}  (1,1,\textbf{Y}) , $$
where \textsc{Minus} resp. \textsc{Cut} have been applied. This matches with $(\ref{eq:1234567})$ since, by $\textsc{Shift}:$
$$\equiv - \zeta^{\sharp\sharp,\mathfrak{l}} _{1}(\textbf{Y},B_{p})+ \zeta^{\sharp\sharp,\mathfrak{l}} (1,\textbf{Y},B_{p})\equiv \zeta^{\sharp,\mathfrak{l}}(1,\textbf{Y},B_{p}). $$
The case $c_{1}=3$ slightly differs since $(\ref{eq:12345})$ gives, by recursion hypothesis I.($\autoref{lzg}$):
$$ -\zeta^{\star,\mathfrak{l}}_{1} (2,\textbf{X})\equiv \zeta^{\sharp,\mathfrak{l}} (B_{1}+1,\textbf{Y}',B_{p}-1)+ \zeta^{\sharp,\mathfrak{l}}  (1,B_{1},\textbf{Y}',B_{p}-1)+ \zeta^{\sharp,\mathfrak{l}} (B_{1},\textbf{Y}',B_{p}-1,1),$$
where $\textbf{Y}'= \boldsymbol{1}^{\gamma_{2}},\cdots,  \boldsymbol{1}^{\gamma_{p}} $, odd depth.
Turning into MES$^{\sharp\sharp}$, and using identities of $\S 3.$ in the same way than above, leads to the result. Indeed, from:
$$\equiv \zeta^{\sharp\sharp,\mathfrak{l}} (B_{1}+1,\textbf{Y}',B_{p}-1)+ \zeta^{\sharp\sharp,\mathfrak{l}}  (1,B_{1},\textbf{Y}',B_{p}-1)+ \zeta^{\sharp\sharp,\mathfrak{l}} (B_{1},\textbf{Y}',B_{p}-1,1)$$
$$-\zeta^{\sharp\sharp,\mathfrak{l}}_{B_{1}+1} (\textbf{Y}',B_{p}-1)- \zeta^{\sharp\sharp,\mathfrak{l}}_{1}  (B_{1},\textbf{Y}',B_{p}-1)-\zeta^{\sharp\sharp,\mathfrak{l}}_{B_{1}} (\textbf{Y}',B_{p}-1,1)$$
First and last terms get simplified via $\textsc{Shift}$, while third and fifth term get simplified by $\textsc{Cut}$; besides, we apply \textsc{minus}  for second term, and \textsc{minus} for the fourth term, which are both of even depth. This leads to $\ref{eq:toolid}$, using again $\textsc{Shift}$ for the first term:
$$\begin{array}{l}
\equiv\zeta^{\sharp\sharp,\mathfrak{l}}_{B_{p}-1}  (1,B_{1},\textbf{Y}')-\zeta^{\sharp\sharp,\mathfrak{l}}_{B_{1}} (\textbf{Y}',B_{p}) \\
\equiv \zeta^{\sharp\sharp,\mathfrak{l}} (B_{1},\textbf{Y}',B_{p})-\zeta^{\sharp\sharp,\mathfrak{l}}_{B_{1}} (\textbf{Y}',B_{p}) \\
\equiv \zeta^{\sharp,\mathfrak{l}} (B_{1},\textbf{Y}',B_{p}).
\end{array}$$
\end{itemize}
This ends the proof.

\section{Missing coefficients}

In Lemma $\autoref{lemmcoeff} (v)$, the coefficients $D_{a,b}$ appearing are the only one which are not conjectured. Albeit these values are not required for the proof of Theorem $5.1$, let present here a table of values in small weights for them. Let examine the coefficient corresponding to $\zeta^{\star}(\boldsymbol{2}^{n})$ instead of $\zeta^{\star}(2)^{n}$, which is (by $(i)$ in Lemma $\autoref{lemmcoeff}$), with $n=a+b+1$:
\begin{equation}
\widetilde{D}^{a,b}\mathrel{\mathop:}= \frac{(2n)!}{6^{n}\mid B_{2n}\mid (2^{2n}-2)} D^{a,b}  \quad \text{  and   }  \quad \widetilde{D}_{n}\mathrel{\mathop:}= \frac{(2n)!}{6^{n}\mid B_{2n}\mid (2^{2n}-2)}D_{n} .
\end{equation}
We have an expression $(\ref{eq:coeffds})$ for $D_{n}$, albeit not very elegant, which would give:
\begin{equation} \label{eq:coeffdstilde}
\widetilde{D}_{n}= \frac{2^{2n} (2n)!}{(2^{2n}-2)\mid B_{2n}\mid }\sum_{\sum m_{i} s_{i}=n \atop m_{i}\neq m_{j}} \prod_{i=1}^{k} \left( \frac{1}{s_{i}!} \left( \frac{\mid B_{2m_{i}}\mid (2^{2m_{i}-1}-1) } {2m_{i} (2m_{i})!}\right)^{s_{i}}  \right).
\end{equation}
Here is a table of values for $\widetilde{D}_{n}$ and $\widetilde{D}^{k,n-k-1}$ in small weights:\\
\\
  \begin{tabular}{| l || c | c |c | c |}
   \hline
     $\cdot \quad \quad \diagdown  n$ & $2$ &  $3$ & $4$ &  $5$   \\ \hline
     $\widetilde{D_{n}}$ & $\frac{19}{2^{3}-1}$ & $\frac{275}{2^{5}-1}$& $\frac{11813}{3(2^{7}-1)}$ & $\frac{783}{7}$\\
     & & & & \\ \hline
    $\widetilde{D}_{k,n-1-k}$ &$\frac{-12}{7}$ & $\frac{-84}{31}, \frac{160}{31}$& $\frac{1064}{127}, \frac{-1680}{127}, \frac{-9584}{381}$ & $\frac{189624}{2555}$,$\frac{-137104}{2555}$,$\frac{-49488}{511}$,$\frac{-17664}{511}$ \\ 
     & & & &\\
    \hline
     \hline
     $\cdot \quad \quad \diagdown  n$ & $6$ &  $7$ & $8$ &  $9$  \\ \hline
     $\widetilde{D_{n}}  \quad \quad $ & $\frac{581444793}{691(2^{11}-1)}$& $\frac{263101079}{21(2^{13}-1)}$& $\frac{6807311830555}{3617(2^{15}-1)}$& $\frac{124889801445461}{43867(2^{17}-1)}$\\ 
      & & & & \\
     \hline
  \end{tabular} \\
  \\
  \\
The denominators of  $\widetilde{D_{n}},\widetilde{D}_{k,n-1-k}$ can be written as $(2^{2n-1}-1)$ times the numerator of the Bernoulli number $B_{2n}$. No formula has been found yet for their numerators, that should involve binomial coefficients. These coefficients are related since, by shuffle:
$$\begin{array}{lll}
 & \zeta^{\star\star, \mathfrak{m}}_{2} (\boldsymbol{2}^{n})+ \sum_{k=0}^{n-1}\zeta^{\star\star, \mathfrak{m}}_{1} (\boldsymbol{2}^{k},3,\boldsymbol{2}^{n-k-1}) & =0\\
&  \zeta^{\star\star, \mathfrak{m}}(\boldsymbol{2}^{n+1})-\zeta^{\star, \mathfrak{m}}(\boldsymbol{2}^{n+1}) \sum_{k=0}^{n-1}\zeta^{\star\star, \mathfrak{m}}_{1} (\boldsymbol{2}^{k},3,\boldsymbol{2}^{n-k-1}) & =0.
\end{array}$$
Identifying the coefficients of $\zeta^{\star}(\boldsymbol{2}^{n})$ in formulas $(iii),(v)$ in Lemma $\autoref{lemmcoeff}$ leads to:
\begin{equation}\label{eq:coeffdrel}
1-\widetilde{D_{n}}= \sum_{k=0}^{n-1} \widetilde{D}_{k,n-1-k}.
\end{equation}

\end{appendices}

\end{document}